\documentclass[11pt]{article}
\usepackage[margin=2cm]{geometry}
\usepackage{etex}
\usepackage{authblk}

\usepackage{amsmath}
\usepackage[usenames,dvipsnames,svgnames,table]{xcolor}
\usepackage{enumerate}
\usepackage{graphicx}
\usepackage{mathrsfs}
\usepackage{hyperref,comment}
\usepackage{subfig,xspace}

\usepackage{tikz}
\usepackage{pgf}
\usepackage{pgflibraryarrows}
\usepackage{pgffor}
\usepackage{pgflibrarysnakes}
\usetikzlibrary{fit} 
\usetikzlibrary{positioning}
\usepgflibrary{shapes}
\usetikzlibrary{snakes,automata}
\usetikzlibrary{shadows}

\tikzset{
  shadowed/.style={preaction={
      transform canvas={shift={(2pt,-1pt)}},draw opacity=.2,#1,preaction={
        transform canvas={shift={(3pt,-1.5pt)}},draw
        opacity=.1,#1,preaction={
          transform canvas={shift={(4pt,-2pt)}},draw
          opacity=.05,#1,
  }}}},
}

\makeatletter
\newif\if@restonecol
\makeatother

\usepackage{ntheorem}

\newtheorem{theorem}{Theorem}	
\renewtheorem*{theorem*}{Theorem}	

\newtheorem{definition}{Definition}
\newtheorem{assumption}{Assumption}
\newtheorem{lemma}{Lemma}

\newtheorem{remark}{Remark}

\newtheorem{corollary}{Corollary}
\newtheorem{proposition}[theorem]{Proposition}  

\usepackage{amsmath,mathrsfs}
\usepackage{amssymb}

\def\R{\mathbb{R}}

\def\d{\delta}

\newcommand{\argmin}{\operatornamewithlimits{argmin}}

\newcommand{\G}{\mathcal{G}}

\usetikzlibrary{patterns}

\tikzset{
  ashadow/.style={opacity=.25, shadow xshift=0.07, shadow yshift=-0.07},
}

\hypersetup{colorlinks=true,linkcolor={blue},citecolor={Maroon}}
           
\newenvironment{proof}{\paragraph{Proof:}}{\hfill$\square$}

\newtheorem{manualtheoreminner}{Assumption}
\newenvironment{manualtheorem}[1]{%
  \renewcommand\themanualtheoreminner{#1}
  \manualtheoreminner
}{\endmanualtheoreminner}

\usepackage{bbm}

\usepackage{enumerate}
\usepackage{times}

\usepackage{centernot}
\usepackage{url}
\usepackage{cite}

\usepackage{amsfonts,mathrsfs}
\usepackage{amssymb,amsmath}
\usepackage{verbatim}
\usepackage{acronym}
\usepackage{mathtools}

\usepackage{graphicx}

\usepackage{hyperref}

\newcommand{\remove}[1]{}

\def\fskip#1{}

\def\tpmax{\tilde p^{\max}}

\def\R{\mathbb{R}}

\def\pmax{p^{\max}}

\def\G{\mathcal G}

\def\allzero{\mathbf{0}}
\def\allone{\mathbf{1}}

\def\I{\mathcal I}

\def\dmin{d_{\min}}

\usepackage{amsmath}

\DeclarePairedDelimiterX{\infdivx}[2]{(}{)}{%
  #1\;\delimsize\|\;#2%
}

\usepackage[normalem]{ulem}

\def\N{\mathbb{N}}

\usepackage{algorithm}
\usepackage[noend]{algpseudocode}

\begin{document}

\title{Optimal Interventions in Coupled-Activity Network Games: Application to
Sustainable Forestry}

 \author{Rohit Parasnis}
 \author{Saurabh Amin} 
 \affil{Laboratory for Information and Decision Systems\\ Massachusetts Institute of Technology}
\date{}
\maketitle
\begin{abstract}We address the challenge of promoting sustainable practices in production forests managed by strategic entities (agents) that harvest agricultural commodities under concession agreements. These entities engage in activities that either follow sustainable production practices or expand into protected forests for agricultural growth, which leads to unsustainable production. Our study uses a network game model to design optimal pricing policies that incentivize sustainability and discourage environmentally harmful practices. Specifically, we model interactions between agents, capturing both intra-activity (within a single production activity) and cross-activity (between sustainable and unsustainable practices) influences on agent behavior.

We solve the problem of maximizing welfare while adhering to budgetary and environmental constraints -- particularly, limiting the aggregate level of unsustainable effort across all agents. Although this problem is NP-hard in general, we derive closed-form solutions for various realistic scenarios, including cases with regionally uniform pricing and the use of sustainability premiums or penalties. Remarkably, we find that it is possible to achieve both welfare improvement and reduction in unsustainable practices without reducing any agent's utility, even when there is no external budget for increasing premiums.

We also introduce a novel node centrality measure to identify key agents whose decisions most influence the aggregate  level of unsustainable effort. Empirical validation confirms our theoretical findings, offering actionable insights for policymakers and businesses aiming to promote sustainable resource management in agricultural commodity markets. Our work has broader implications for addressing sustainability challenges in the presence of network effects, offering a framework for designing incentive structures that align economic objectives with environmental stewardship.
\end{abstract}

\section{Introduction}
\label{sec:intro}

A major driver of climate change is the conversion of tropical forests into large-scale agricultural plantations, resulting in over 1.47 gigatons of CO\textsubscript{2} emissions per year. An important class of public policies for combating climate change focuses on deforestation caused by agricultural activities (\cite{indonesia_president_1999}). On one hand, these policies include penalties and restrictions to market access imposed by government bodies and supranational organizations such as the European Union to discourage cultivators such as agricultural plantation owners from engaging in unsustainable and illegal harvesting practices. On the other hand, they  include sustainability premiums (i.e., price raises) and certifications awarded to cultivators by non-profit organizations such as the Roundtable on Sustainable Palm Oil (RSPO) to promote sustainable agricultural practices.

This article is motivated by two key issues that limit successful implementation of current policies. 
First, the incentives for sustainable practices are often inadequate. While current 
premiums for producing sustainable goods (i.e., goods from sustainability-certified cultivators (agents)) are set to ensure that these goods are sold at higher prices compared to their unsustainable counterparts, 
current certification costs 
are too high to incentivize agents to choose sustainable practices over unsustainable ones (\cite{syarfi2019impact,chinadialogue2021}). These economic impediments are likely to become worse under market restrictions imposed on unsustainable cultivators. In particular, the European Union Deforestation Regulation forbids the trade of commodities produced using lands deforested after 2020 (\cite{european_commission_2023}). Second, while current policies are based on guiding principles that recognize how synergistic interactions should be promoted between various stakeholders (\cite{asean2023regional,carlson2018effect,greenpeace2021deceased}),  these guidelines do not account for  agent-to-agent interactions that are inherent to forest concession networks and further facilitated by their ownership/management structure. We highlight the key interactions below. 
\begin{enumerate}
    \item \textbf{\textit{Synergistic network interactions:}} Agents managing neighboring plantations use pooling of capacities and knowledge-sharing for  enhancing their agricultural yields (\cite{velten2021success}). 
    \item \textbf{\textit{Joint subversion of forest protection rules:}} Production of agricultural goods  often involves clearing forested areas and building or expanding agricultural plantations on the cleared areas. This encourages cultivators to jointly engage in the unsustainable activity of clearing protected forests, especially when the incentives for  sustainable agriculture are inadequate (\cite{carlson2018effect}). 
    \item \textbf{\textit{Shadow companies:}} Recently, it has been reported that some oil palm companies  set up ``networks of smaller shadow companies that are not themselves subject to no-deforestation commitments in order to circumvent their own sustainability policies''(\cite{mongabay2023investigation}). Detecting and preventing covert deforestation activities by shadow companies is difficult because their activities are often disguised as side businesses separate from listed entities, the ownership of their assets is concealed by multiple layers of corporate control, and their administrative ownership is repeatedly transferred to obfuscate the true ownership of these companies~(\cite{kuepper2018shadow}).
\end{enumerate}

One can expect these network interactions to be facilitated by the spatial configurations of forest concessions and the manner in which they are managed by the agents. Importantly, these interactions introduce  mutual dependencies across agents and influence their incentives to engage in sustainable production. The outcomes of these interactions can have a significant impact on the agricultural deforestation rates (\cite{robalino2012contagious}). 

Thus, there is a need for a systems-focused approach for designing price-shaping policy interventions that promote sustainable agricultural practices, while limiting the environmental and societal consequences of deforestation due to the network interactions highlighted above. In this work, we address the question of designing \textit{price-shaping interventions that reduce the agents' effort in unsustainable agricultural deforestation while accounting for  network interactions and maximizing   welfare.} Our results offer useful insights to facilitate the engagement of the private sector (major cultivators and plantation companies) with  government bodies and planning bodies to address the pressing sustainability challenge of agricultural deforestation, especially in tropical countries that are major suppliers of palm oil. 
Our approach builds on a game-theoretic model of agents' effort in producing forest commodities via sustainable agriculture (which is carried out by certified plantations or listed companies owned or managed by the agents) and unsustainable agriculture (carried out by non-certified plantations and/or shadow companies managed by the agents). We account for network interactions that together include synergistic interactions among cultivators and their collective incentive to subvert forest protection rules while possibly leveraging shadow companies:  (i) \textit{intra-activity} interactions, which are mutually beneficial (utility-improving) interactions between cultivators performing the same kind of activity -- either sustainable production or unsustainable production, but not both; and (ii) \textit{cross-activity} interactions, each of which is a mutually beneficial interaction between a cultivator involved in sustainable production and another involved in unsustainable production. Intra-activity and cross-activity interactions together constitute the \textit{strategic interaction network} between the agents, and we use the term \textit{network effects} to refer to the impact of these interactions on the total welfare and aggregate levels of unsustainable effort.  
 Specifically, we consider a network game in which each agent can produce agricultural commodities via both sustainable and unsustainable agriculture and is subject to both intra-activity and cross-activity interactions. 
 The game incorporates strategic complementarities in agent-to-agent interactions and strategic substitutabilities across the two activities for any given agent. We use this network game model to design pricing policies 
that address the economic objective of maximizing the welfare while meeting a price adjustment budget and incentivizing the agents to keep their aggregate equilibrium effort in unsustainable production below a pre-defined tolerance. We also consider variants of this problem to take into account different practical considerations, including (a) the availability versus unavailability of external budget to promote sustainable production, (b) the feasibility of imposing monetary penalties for unsustainable production, and (c) factoring in the geospatial structure of concession networks for situations in which uniform pricing across cultivations that share borders is preferred. The approach we develop applies to all these variants and enables a systematic comparison of the resulting policy recommendations.

Our policy design problem poses a few technical challenges. First, due to the presence of both strategic complementarities and substitutabilities, the underlying network game is not supermodular like the games studied in~\cite{vives1990nash}. Moreover, the agents' effort levels are constrained to be non-negative in our model,  which complicates the equilibrium characterization of our game. Second, our design problem is NP-hard  as it entails non-convex optimization over a polytope  defined by a number of constraints (including budget constraints, lower bound constraints on prices, and a tolerance constraint on aggregate unsustainable effort) that is more than twice the network size, which results in an exponential number of extreme points in the worst case. 

We address these challenges by establishing the non-negativity of the difference of Leontief matrices (see~\cite{kuznets1941structure}) that facilitate equilibrium characterization and by leveraging the algebraic properties of the welfare function. Additionally, we build on the technical approach of~\cite{li2023efficient}, which extends the work of~\cite{edirisinghe2016efficient}.  Their approach uses conjugate duality theory to develop an algorithm for solving a variable-separable non-convex quadratic optimization problem with a coupling \textit{equality} constraint. Here, we extend their analysis significantly to show that if uniform pricing is enforced within each connected component of the concession network, our intervention design problem becomes a variable-separable concave quadratic optimization problem with a coupling \textit{inequality} constraint, which we tackle by  exploiting the properties of the subgradients of the dual objective of the convex relaxation of the original problem. 

\subsection*{Our Contributions}
\begin{enumerate}
    \item \textbf{\textit{Conditions for jointly maximizing welfare and sustainability:}} We derive tight necessary and sufficient conditions for our intervention design problem and its variants to yield policies that jointly maximize welfare and reduction in aggregate unsustainable effort regardless of the structure of the strategic interaction network (Theorem~\ref{thm:one}). In practice, these conditions can be satisfied by ensuring that cross-activity interactions, which occur between sustainable and unsustainable concessions, are weaker than the intra-activity interactions.  Importantly, if network interactions violate these conditions, then price-shaping interventions aimed at  welfare improvement can lead to an increase rather than a reduction in unsustainable effort.
    \item \textbf{\textit{Welfare maximization in zero and non-zero budget settings:}} We derive closed-form expressions for the  policy that maximizes welfare while reducing aggregate unsustainable effort under budget constraints for various cases of practical interest (Theorem~\ref{thm:general_case} and Corollary~\ref{cor:second}). We also consider the zero-budget setting (Theorem~\ref{thm:redistribution}) and show that redistributing the prices of sustainable and unsustainable goods, which can be implemented by penalizing unsustainable effort and awarding higher sustainability premiums, can significantly reduce aggregate unsustainable effort and maximize welfare while ensuring that no agent's utility decreases as a result of the redistribution.
    \item \textbf{\textit{Welfare Maximization under Component-Wise Uniform Pricing Constraints:}} To address situations in which  the planner and the cultivators prefer geographically co-located cultivators to be offered identical prices, we consider the problem variant where the prices of sustainable goods are constrained to be uniform within every geographical region, and consequently, within every connected component of the geospatial network of concessions (which determines the strategic interaction network). We show that reducing the tolerance on aggregate unsustainable effort by an appropriate amount results in the optimal solution of the problem being equal to the optimal solution of its convex relaxation, which we characterize using closed-form expressions under mild assumptions (Theorem~\ref{thm:component_wise_pricing}). 
    \item \textbf{\textit{Empirical Analysis-based Policy Guidelines:}} We study the case of palm oil cultivation in Indonesia (Section~\ref{sec:empirical}), and we use real data on concession networks and the reported prices of certified and non-certified crude palm oil to assess the performance of the optimal policies resulting from our intervention design problems on welfare improvement as well as reduction of aggregate unsustainable effort. This analysis validates our theoretical results and shows that our policies significantly outperform the policies  that are currently in place (as defined by the conventional palm oil prices, sustainability premiums, and certification costs reported in \cite{wwf2022business}) by both the yardsticks. We also uncover two novel insights: (a) welfare improvement is guaranteed if the planner's external budget is adequate, but the more skewed the pre-intervention prices are towards unsustainable production, the lower is the percent increase in welfare relative to current policies for a given budget, and (b) even in the absence of an external budget, price redistribution (which generates premiums from penalties) can achieve significantly greater reduction in unsustainable production than zero-penalty policies that use external budgets to raise sustainability premiums.
    \item \textbf{\textit{Characterization of Nash Equilibria under Non-negative Effort Constraints:}} To analyze cases in which the prices of sustainable goods are high enough to drive some agents' unsustainable effort levels to zero, we characterize the Nash equilibria of our network game model under non-negativity constraints on the agents' efforts. We focus on cases where a subset of equilibrium expressions, derived in~\cite{chen2018multiple} (where this model was first proposed), are negative. This result (Theorem~\ref{thm:non-negative}), which helps us examine the impact of high post-intervention prices, is of independent interest as it applies to the  coupled-activity game regardless of the context.
\end{enumerate}

\textit{Notation: } Let $\N$ denote the set of natural numbers, let $\N_0:=\N\cup\{0\}${, and for a given $n\in\N$, let $[n]:=\{1,2,\ldots,n\}$}. Let $\R$ denote the set of real numbers, $\R^n$ denote the set of $n$-dimensional real-valued column vectors, and let $\R^{m\times n}$ denote the set of square matrices with $m$ rows, $n$ columns, and real entries in all rows and columns. For a  matrix $A\in\R^{n\times n}$, we let $a_{ij}=(A)_{ij}$ denote the entry in the $i$-th row and the $j$-th column of $A$. Likewise, for a vector $v\in\R^n$, we let $v_i=(v)_i$ denote the entry in the $i$-th row of $v$. 

Let {$I_n$ (respectively, $O_n$)} denote the identity matrix {(respectively, the all-zeros matrix) in $\R^{n\times n}$, let $O_{m\times n}$ denote the all-zeros matrix in $\R^{m\times n}$}, let {$\mathbf 0_n\in\R^n$}  {(respectively, $\mathbf 1_n\in\R^n$)} denote the {$n$-dimensional} vector with all entries equal to zero {(one, respectively)}{, and let $e_n\in\R^n$ denote the $n$-th canonical basis vector, i.e., the vector with $1$ in its $n$-th entry and zeros in all other entries.} We drop  subscripts such as $_n$ and $_{m\times n}$ when the matrix dimensions are clear from the context. We assume that all matrix and vector inequalities hold entry-wise, e.g., $A\geq B$ means each entry of $A$ is no less than the corresponding entry of a matrix $B$ (of compatible dimension).

{For a vector $v\in\R^n$  and a subset $S\subset [n]$, we let $v_S=(v)_S\in \R^{|S|} $ denote the restriction of $v$ to the index set $S$. Similarly,} for a matrix $A\in\R^{n\times n}$, we let $A_S$ be the principal sub-matrix of $A$ corresponding to the rows and columns indexed by $S$. Additionally, for $T:=[n]\setminus S$, we let $A_{S T}$  denote the sub-matrix of $A$ corresponding to the rows indexed by $S$ and the columns indexed by $T$.

\section{Intervention Design Problem}\label{sec:prob_formulation}

To motivate our intervention design problem, consider a set of forest regions within a given administrative zone. Each of these regions consists of a contiguous stretch of oil palm \textit{concessions}, which are plots of land sanctioned by the government for palm oil cultivation; see Fig.\ \ref{fig:new} for an example.
\begin{figure}[h]
    \centering
    \includegraphics[scale=0.45]{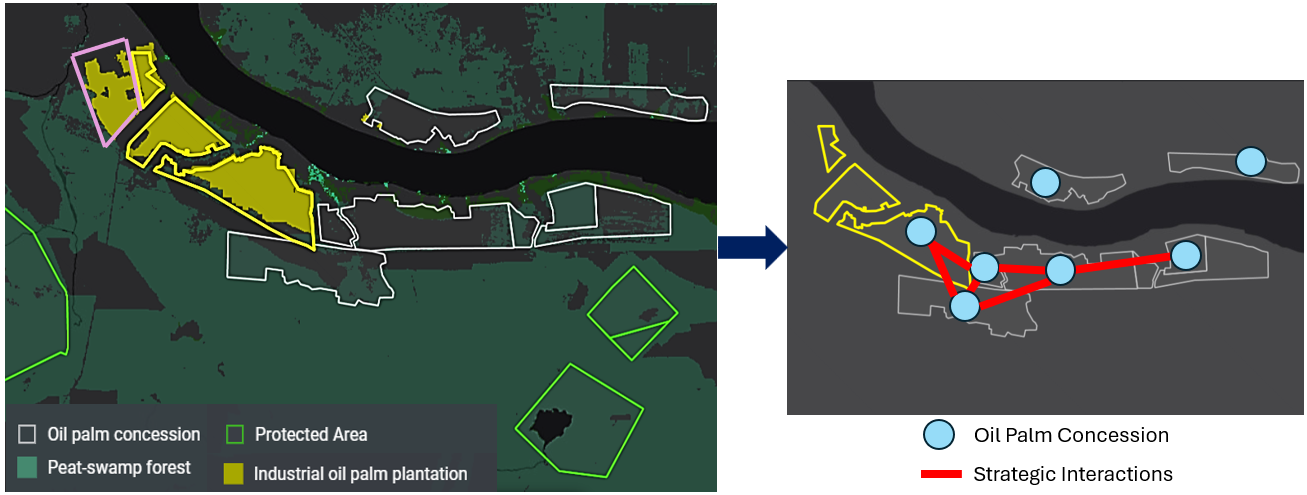}
    \caption{A Forest Region and the Corresponding Strategic Interaction Network}
    \small A subset of concessions in Kalimantan, Indonesia (left) induce a strategic interaction network (right) with three connected components, two of which are isolated nodes and the largest consists of five nodes. A single concession can be a union of multiple disjoint plots of land (e.g., the concession highlighted in bright yellow). When a plantation expands outside designated concession boundaries (e.g., to include the region outlined in magenta on the left), it can encroach on protected areas (outlined in bright green). Such activities are unsustainable. \textit{Source: \cite{nusantaraatlasNusantaraAtlas}}.
    
    \label{fig:new}
    \normalsize
\end{figure}
These concessions are managed by $n$ strategic cultivators (companies or plantation managers), each of whom oversees an assigned number of concessions. Cultivators of  neighboring concessions interact with each other by sharing capabilities or knowhow to enhance their agricultural yields. We model the network of these interactions as a graph, with an adjacency matrix $G\in\{0,1\}^{n\times n}$ whose binary entries $\{g_{ij}:i,j\in[n]\}$ indicate whether or not cultivators $i$ and $j$ interact with each other. In particular $g_{ij}=1$ if and only if cultivators $i$ and $j$  interact. As cultivators can own and manage multiple concessions, the structure of this \textit{strategic interaction network} is determined by the geospatial co-locations of the concessions, as well as their ownership and management structure. 

The cultivators use concessions to produce and sell agricultural commodities such as palm oil and timber  in a market. The production of these commodities results from two kinds of activities: (a)  \textit{sustainable production} or activity $A$, defined as an activity that does not encroach on protected forests (namely, primary forests and High Conservation Value (HCV) areas; see \cite{gibson2011primary} and \cite{fsc2004principles}), and (b)   \textit{unsustainable production} or activity $B$, which involves encroaching on and/or clearing protected forest land for expanding palm oil plantations. Activity $B$ is known to contribute to depletion of protected forests (\cite{sumarga2016benefits}).

Cultivators engaging in activity $A$ include listed companies and plantation managers who have been awarded sustainability certification by agencies such as RSPO, and those engaging in activity $B$ include shadow companies of listed entities or non-certified plantations (as described in Section~\ref{sec:intro}). 
The term \textit{sustainable good} refers to the commodity produced as a result of activity $A$, and $p_i^A$ denotes the price of the sustainable good produced by cultivator $i\in [n]$ per unit of the cultivator's effort. 
Similarly, the term \textit{unsustainable good} refers to the commodity resulting from activity $B$, and $p_i^B$ denotes the price of the unsustainable good produced by cultivator $i\in [n]$ per unit of the cultivator's effort. Practically, a \textit{unit of effort} refers to the amount of exertion (including intensity and effectiveness of labor and skills) towards producing a standard unit (e.g., one cubic meter or one metric ton) of agricultural commodity (sustainable or unsustainable good).  

Now, consider a planner who seeks to increase the economic incentive for sustainable production or suppress unsustainable production via price-shaping interventions, or equivalently, via policies that modify $\{p_i^A:i\in[n]\}$ and/or $\{p_i^B:i\in [n]\}$. Such policies are implemented by modifying the sustainability premiums that are added to the base prices of sustainable goods (see Appendix~\ref{sec:monetary_component}).

At first glance, it might seem that one can design such interventions by considering the aggregate utility (welfare) of all $n$ agents (cultivators), where the utility function of agent $i$ is defined as  $u_i(x^A, x^B; p^A, p^B):= p_i^Ax_i^A  + p_i^B x_i^B - \frac{1}{2}(x_i^A)^2 - \frac{1}{2}(x_i^B)^2- \beta x_i^A x_i^B$, where $x_i^A\in\R^n_{\ge 0}$ (respectively, $x_i^B\in\R^n_{\ge 0}$) denotes the amount of effort by agent $i$ in activity $A$ (respectively, activity $B$), and the terms $-\frac{1}{2}(x_i^A)^2$ and $-\frac{1}{2}(x_i^B)^2$ model diminishing returns from exploiting labor and resources. Importantly, the cross-term $-\beta x_i^A x_i^B$ with $\beta\in(0,1)$  denotes the effect of strategic \textit{substitutability} between the two activities, which model an agent's cost of dividing their effort across sustainable and unsustainable production.  To design an intervention that maximizes welfare, defined as the sum of all the $n$ individual utilities, one can solve the following problem: 
 ${\boldsymbol P_0: \textit{Maximize }\sum_{i=1}^n u_i}$ \textit{subject to }$p^{A0} \le p^A\le \pmax$, where $p^{A0}$ is the vector of pre-intervention prices and $\pmax\ge p^{A0}$ is a pre-defined vector of upper bounds that depends on the planner's price adjustment budget. The optimal solution of  $\boldsymbol{P}_0$ is given by 
\begin{align}\label{eq:pre-modeling}
    p_i^A = 
    \begin{cases}
        \pmax_i\quad&\text{if }\pmax_i\ge2\beta p_i^B - p_i^A,\\
        p_i^{A0}\quad&\text{otherwise.}
    \end{cases}
\end{align}

However, this intervention has two critical shortcomings. First, it seeks to incentivize sustainable production from a purely ``economistic'' viewpoint without imposing any constraint on unsustainable production - a recipe for forest governance failure (\cite{maxton2020forest}).  
Second,  
the model does not account for the strategic interaction network and its impact on the agents' effort levels in both the activities (\cite{robalino2012contagious}).

\def\dmin{d_{\min}}

To address the first shortcoming, we formulate an intervention design problem that seeks to maintain economic welfare while limiting aggregate levels of unsustainable production. We do this by incorporating a sustainability constraint into $\boldsymbol P_0$ -- one that keeps the aggregate unsustainable effort below a pre-defined tolerance. To address the second shortcoming, we incorporate network effects into the agents' utility structure. Specifically, we consider two kinds of strategic interactions and their impacts on individual utilities. The first of these, which we call \textit{intra-activity network effect}, captures strategic complementarities in intra-activity agent-to-agent interactions that occur either within  activity $A$  (i.e., between listed/certified concessions) or activity $B$ (i.e., between non-certified concessions and/or shadow companies), but not both. In particular, the network effect due to interactions between agents $i$ and $j$ is modeled by adding the terms $\delta x_i^Ax_j^A$ and $\delta x_i^Bx_j^B$ to the utility functions of agents $i$ and $j$, where $\delta>0$ denotes the intra-activity network effect parameter. The second  network effect, which we call \textit{cross-activity network effect}, captures strategic complementarities in cross-activity interactions (\cite{chen2018competitive}), which occur between  certified plantations and non-certified plantations or shadow companies managed by distinct cultivators. We model such interactions between agents $i$ and $j$ by adding the terms $\mu x_i^Ax_j^B$ 
and $\mu x_i^Bx_j^A$  to the utility functions of agents $i$ and $j$, where $\mu>0$ denotes the cross-activity network effect parameter. One can consider the intra-activity interactions $\{\delta x_i^A x_j^A:g_{ij}>0\}$ overt and other interactions covert since they involve unsustainable production.

The resulting utility function of agent $i\in [n]$ becomes
\begin{align}\label{eq:main}
&u_i\left(x^A, x^B; p^A, p^B \right)\cr
&=  p_i^A x_i^A+ p_i^B x_i^B\cr
&\quad-\left\{\frac{1}{2}\left(x_i^A\right)^2+\frac{1}{2}\left(x_i^B\right)^2+\beta x_i^A x_i^B\right\} \cr
&\quad+\delta \sum_{j=1}^n g_{i j} x_i^A x_j^A+\delta \sum_{j=1}^n g_{i j} x_i^B x_j^B\cr
&\quad +\mu \sum_{j=1}^n g_{i j} x_i^A x_j^B+\mu \sum_{j=1}^n g_{i j} x_i^B x_j^A.
\end{align}

The network game $\left([n],\R^n\times \R^n ,\{u_i\}_{i\in[n]} \right)$,  was introduced in \cite{chen2018multiple},  where it was shown to admit a unique Nash equilibrium under Assumption~\ref{assum:uniqueness}, stated below. We modify this game by imposing non-negativity constraints on the effort levels, resulting in the modified game  $\left([n],\R^n_{\ge 0}\times \R^n_{\ge 0} ,\{u_i\}_{i\in[n]} \right)$. As we will show in Theorem~\ref{thm:non-negative}, determining the equilibria of $\left([n],\R^n\times \R^n ,\{u_i\}_{i\in[n]} \right)$ is helpful in characterizing the equilibria of $\left([n],\R^n_{\ge 0}\times \R^n_{\ge 0} ,\{u_i\}_{i\in[n]} \right)$.

\begin{assumption} Let $\rho(G)$ denote the spectral radius (i.e., the maximum of the absolute values of the eigenvalues)  of $G$. Then \label{assum:uniqueness}
    $$
        \max \left\{\frac{|\delta+\mu|}{1+\beta}, \frac{|\delta-\mu|}{1-\beta}\right\} \rho({G})<1.
    $$
\end{assumption}

This assumption requires that network effects remain sufficiently limited to prevent agents' effort levels from diverging to infinity in the best-response dynamics. It is analogous to standard assumptions in the literature on network games that require the network effects modeled by these games to be small enough for the  uniqueness of Nash equilibrium  (see, for example, Theorem 1 in~\cite{ballester2006s}). We adopt this assumption throughout the paper.

\begin{remark}\label{rem:econ}
    The values of the parameters $\beta$, $\delta$, and $\mu$ can be obtained by estimating econometric models (\cite{robalino2012contagious,cohen2018multivariate}) that can capture relationships between geospatial co-locations of the concessions, their management structure, and observed production levels due to activities $A$ and $B$ (i.e., agents' effort levels $\{x_i^A:i\in[n]\}$ and $\{x_i^B:i\in[n]\}$), which includes deforestation due to unsustainable production. 
\end{remark}   

Under Assumption~\ref{assum:uniqueness}, the unique Nash equilibrium of the game $\left([n],\R^n\times \R^n ,\{u_i\}_{i\in[n]} \right)$ is as follows~\cite[Theorem 4]{chen2018multiple}:\begin{subequations}\label{eq:equilibrium}\begin{align}
\begin{split}\label{subeq:x-A} 
    x^{A*} &= \frac{1}{2}\left( (M^+ + M^-)p^A + (M^+ - M^-)p^B \right), 
\end{split}\\
\begin{split}
\label{subeq:x-B}
    x^{B*} &= \frac{1}{2}\left( (M^+ - M^-)p^A + (M^+ + M^-)p^B \right).
\end{split}
\end{align}    
\end{subequations}

\def\plim{p^{\lim} }

Here, $M^+:=((1+\beta)I - (\delta+\mu)G)^{-1}$ and $M^-:=((1-\beta)I - (\delta-\mu)G)^{-1}$ are scaled Leontief matrices  of the form $\alpha(I-A)^{-1}$ where $A$ is a non-negative matrix and $\alpha$ is a positive scalar. They capture the contributions made by the prices $\{p_i^A:i\in [n]\}$ and $\{p_i^B:i\in[n]\}$  to equilibrium effort levels as their effects propagate over the strategic interaction network via walks of all lengths. In particular, when a link is encountered while traversing these walks, these contributions are discounted by the factors  $\delta+\mu$ (respectively, $\delta-\mu$), which can be viewed as the network effect parameters associated with a hypothetical single-activity linear-quadratic game with effort levels given by $\{x_i^{A*}+x_i^{B*}:i\in[n]\}$ (respectively, $\{x_i^{A*}-x_i^{B*}:i\in[n]\}$).

\def\pbold{\boldsymbol{P}}

We formulate the intervention design problem in two ways. The first variant, denoted by $\pbold$,  is an incentive design problem that treats the pre-intervention prices $\{p_i^{A0}:i\in[n]\}$ and $\{p_i^{B0}:i\in[n]\}$ as fixed parameters and the post-intervention prices $\{p^A_i:i\in[n]\}$ as the decision variables.

Formally, $\pbold$ can be stated as follows: Maximize the welfare of the network of cultivators while ensuring that (i)  the aggregate unsustainable effort  remains below a pre-defined tolerance $\tau^B\ge 0$, 
and (ii) the post-intervention per-unit prices of sustainable goods are no less than their pre-intervention counterparts and no greater than the pre-defined bounds $\{\pmax_i:i\in[n]\}$. 
\begin{subequations}
\begin{align}
\begin{split} 
    \pbold:\quad\underset{\{p_i^A:i\in[n]\} }{\text{Maximize}}  &\sum_{i=1}^n u_i\left( x^{A*}, x^{B*}; p^A, p^B \right) 
\end{split}\nonumber\\
\begin{split}
\label{eq:unsustainable_effort}\text{such that}\quad \sum_{i=1}^n x_{i}^{B*} &\le \tau^B,
\end{split}\\
\begin{split}
    \label{eq:price_up_second} p_i^A &\ge p^{A0}_i \quad\text{for all }i\in[n],
\end{split}\\
\begin{split}
\label{eq:price_threshold_second} p_i^A &\le \pmax_i\quad\text{for all }i\in[n].
\end{split}
\end{align}    
\end{subequations}
Here, $\{u_i:i\in[n]\}$ are given by~\eqref{eq:main}, and  $x^{A*}$ and $x^{B*}$ are vectors of equilibrium effort levels. These vectors are given by~\eqref{eq:equilibrium} provided the right-hand-side of~\eqref{subeq:x-B} is non-negative (see Theorem~\ref{thm:non-negative} for a more comprehensive characterization of the equilibrium). We use the term \textit{policy} to refer to the vector $p^A$ of the post-intervention prices of sustainable goods. To ensure that $\pbold$ is feasible, we must ensure that the policy $\pmax$, which maximally incentivizes sustainable effort, leads to an equilibrium that satisfies the tolerance constraint~\eqref{eq:unsustainable_effort} on aggregate unsustainable effort. In light of~\eqref{eq:equilibrium}, $\pmax$ is constrained to satisfy the inequality $\allone^\top (M^- - M^+)\pmax \ge \allone^\top (M^+ + M^-)p^B - 2\tau^B$. In addition, by imposing $\pmax\ge p^{A0}$, we ensure that~\eqref{eq:price_up_second} and~\eqref{eq:price_threshold_second} are satisfied simultaneously.

\def\pbr{\pbold_{\text{R}}}

Note, however, that the condition $\pmax\ge p^{A0}$ can be satisfied only if either $\pmax = p^{A0}$ (i.e., no intervention since~\eqref{eq:price_up_second} and~\eqref{eq:price_threshold_second} imply that $p^A = p^{A0}$) or if the planner can afford to increase the prices of sustainable goods. If the planner's price adjustment budget $\pmax - p^{A0}$ is small or zero, then we can use an idea similar to ``symmetric tax-and-subsidy'' programs that are arguably effective at reducing carbon emissions from deforestation (\cite{busch2012structuring,busch2015reductions}). That is, the planner increases the prices of sustainable goods by imposing monetary penalties for unsustainable production and distributes the collected fines to incentivize sustainable production. In other words, the planner reduces the prices of unsustainable goods by enforcing $p^B\le p^{B0}$ and increase the prices of sustainable goods by enforcing $p^A\ge p^{A0}$ to ensure that $p_i^A - p_i^{A0} \ge p_i^{B0} - p_i^B$. This variant of the design problem, denoted by $\pbr$, treats the prices of both sustainable and unsustainable goods as decision variables, leading to the modified intervention design problem with redistribution:
\def\rhomax{\rho^{\max}}
\begin{subequations}\label{eq:redistribution}
\begin{align}
\begin{split} 
    \pbr:\quad\underset{\{p_i^A, p_i^B :i\in[n]\} }{\text{Maximize}}  &\sum_{i=1}^n u_i\left(x^{A*}, x^{B*}; p^A, p^B \right) 
\end{split}\nonumber\\
\begin{split}
\label{eq:unsus_redist}\text{s.t.}\quad \sum_{i=1}^n x_{i}^{B*} &\le \tau^B,
\end{split}\\
\begin{split}
    \label{eq:price_up_redist} p^A &\ge p^{A0} 
\end{split}\\
\begin{split}
\label{eq:price_down_redist} p^{B0} - \rhomax &\le p^B \le p^{B0}
\end{split}\\
\begin{split}
\label{eq:redist_budget} p^A + p^B &\le p^{A0} + p^{B0} + b 
\end{split}
\end{align}    
\end{subequations}
where $\rhomax_i\in [0, p_i^{B0}]$ is the maximum penalty per unit of unsustainable effort by agent $i$, and  $b_i>0$ is the price-adjustment budget of the planner for agent $i$. Here, we use the term \textit{policy} to refer to the pair $(p^A,p^B)$. 

While $\pbr$ is more suitable than $\pbold$ when the planner's budget is zero or negligible relative to the pre-intervention prices of sustainable goods, $\pbold$ is more suitable in situations where cultivators can evade penalties by exploiting weaknesses in the enforcement of these penalties.

We now state an assumption underlying all our  results.

\begin{assumption}\label{assum:mu-delta}
    $\mu<\delta$, i.e., the cross-activity network effect is weaker than the intra-activity  effect.
\end{assumption}

This assumption is reasonable because the covert interactions between agents' own certified concessions and other cultivators' non-certified concessions (including shadow companies) are likely to have a lesser impact on their effort levels in comparison to the intra-activity network effect.

\subsection*{Related Works}

The observed impacts of commercially driven deforestation has motivated a line of works that examines the link between harmful environmental and societal consequences of deforestation on one hand and the  
commercial interests of the deforesting entities on the other. A comprehensive survey of these works can be found in~\cite{balboni2023economics}. 

\textbf{\textit{Managerial Insights into Sustainable Forestry:}} 
The study by \cite{hart2000business} examines 21 cases in the forest products industry, demonstrating that companies aligning sustainable forest management (SFM) practices with their core strategic goals gain a competitive advantage over those treating SFM as merely a social responsibility. This finding supports the argument that environmental and economic goals should be integrated rather than isolated, which justifies our constrained welfare maximization problem. Similarly, \cite{lee2018socially} underscores the importance of integrating economic and environmental objectives within supply chain sustainability. They introduce the concept of ``socially and environmentally responsible value chains'' and suggest new directions for operations management research in this area. 

In related research, \cite{li2017supply} explore the role of suppliers in global supply chain sustainability, showing that strong knowledge-sharing connections with partners amplify the positive impact of market-focused sustainability efforts on performance. Such sustainability-enhahcing effects of synergistic interactions (such as knowledge exchange) among sustainable concessions are accounted for by our game-theoretic model. Moreover, large firms can go further and contribute to global sustainability by forming supportive partnerships with local communities and offering skill-building opportunities. For instance, \cite{mcgahan2023there} analyze the efforts of Natura, a Brazilian cosmetics company, which successfully encouraged Amazon communities to shift from land clearing to cultivating diverse forest products, thereby supporting forest conservation.  
In contrast, our model also considers how firms that possess sustainability certification may hinder sustainability efforts by interacting synergistically with unsustainable concessions, thereby contributing to unsustainable production. Our work thus fills a gap in the literature by accounting for two types of cultivator-cultivator interactions  while addressing welfare maximization under tolerance constraints on unsustainable production. 

On a different yet related topic, \cite{gopalakrishnan2021incentives} investigate how to allocate carbon emissions across a supply chain's constituent firms. They propose an ``emissions responsibility allocation scheme'' based on the Shapley value from cooperative game theory, demonstrating that this approach incentivizes firms to stay near socially optimal pollution levels, even when abatement costs remain private.

\textit{\textbf{Sustainable Palm Oil Cultivation in Southeast Asia:}} The work by \cite{gunarso2013oil} analyzes the land use changes brought about by oil palm cultivation in Indonesia, Malaysia, and Papua New Guinea over two decades using Landsat satellite images and identifies unsustainable logging as a major contributor to deforestation. \cite{gaveau2013reconciling}, on the other hand, focus on one region of Indonesia, namely, Kalimantan, and perform empirical analyses to show that protecting Kalimantan's forests from illegal encroachments can provide employment opportunities while safeguarding pristine forests. In effect, this signifies the need to strike a balance between economic and environmental goals, which aligns with our intervention design problem. Shifting focus to the penalties in place, \cite{busch2015reductions} studies the performance of Indonesia's moratorium on new concessions for oil palm and timber plantations and highlights the advantages of price-based policies (such as ``symmetric tax-and-subsidy programs" similar to those we  design via $\pbr$) for curbing palm oil-driven deforestation. Related are the case study in \cite{henderson2016gotong} 
and the counterfactual analysis in \cite{carlson2018effect}, which support our hypothesis that existing sustainability certification program has been somewhat effective but additional incentives are needed to bring deforestation within desirable limits. More recent is the spatial analysis in \cite{gaveau2021forest}, which shows that Indonesian New Guinea risks losing 4.5 MHa of forests to industrial plantations by 2036, thereby calling for improved sustainability policies, such as those we propose in this paper.

\textit{\textbf{Network Analysis and Game Theory for Sustainable Forestry:}} Previous studies have applied the tools from network science and game theory in advancing our understanding of forest management and the economics of deforestation. For example,~\cite{rodrigues2009game}  study two-person games played by agents choosing between forest conservation and deforestation and shows that the dilemmas affecting the agents' decision-making depend on environmental factors such as forest regeneration rates. Recently~\cite{warnes2023area} focus on sustainable forestry and use cooperative game theory to design incentives for indigenous communities to either prevent deforestation or to prevent the economic use of deforested areas in their localities so as to allow the forest to regenerate. Since forest regeneration can take several decades, the policies proposed by~\cite{warnes2023area} are suited for long-term conservation, whereas our work focuses on short-run trade-off between welfare and sustainability. On the other hand,~\cite{filotas2023network} describes the usefulness of network analysis in studying habitat networks and other networks that capture the relationships between different components of forest ecosystems.  Our work contributes to this arch of works by using game theory and network analysis for sustainable forestry. 

\textit{\textbf{Intervention Design for Network Games:}} Our work contributes to the growing literature on intervention design for network games, some examples of which are~\cite{khanafer2014information, chen2018competitive, como2021optimal, parise2021analysis, jin2021multi, parasnis2024cdc, parise2023graphon, bervoets2023public, shakarami2023dynamic} and \cite{xiong2024cost}. With the exception of  \cite{chen2018competitive} and follow-up works such as \cite{parasnis2024cdc} and \cite{parise2021analysis,parise2023graphon}, this literature focuses on single-activity games. One of the first seminal papers in this category is~\cite{ballester2006s}, which studies a linear best-response network game and proposes the concept of intercentrality to solve the problem of identifying the ``key player'' -- the agent whose removal from the network causes the aggregate effort level to decrease the most.  
\cite{belhaj2016efficient} consider a similar model of a network game with quadratic utilities and show that nested split graphs maximize welfare over the space of all possible network topologies.  
The work in \cite{li2023designing} generalizes these results to directed network topologies and shows that, under mild assumptions, hierarchincal networks are optimal. Unlike~\cite{belhaj2016efficient} and~\cite{li2023designing}, we treat the network topology as a given, because the strategic interaction network depends on the geospatial co-locations of the concessions, over which we have no control. 

\textbf{\textit{Characteristics Interventions for Network Games:}} The interventions designed in the works described above 
are \textit{network interventions}, i.e., they achieve the planner's objective by altering the structure of the strategic interaction network. By contrast, we focus on \textit{characteristics interventions} (\cite{sun2023structural}), which seek to achieve the planner's objective by modifying agent characteristics (such as marginal utilities) rather than the network structure. A related work by~\cite{demange2017optimal} studies a budget-constrained problem of maximizing the aggregate action in a single-activity network game 
and characterizes the planner's optimal strategies for linear and non-linear best responses. 
Another variant was studied in~\cite{galeotti2020targeting}, which 
solves the problem of maximizing  welfare subject to quadratic cost adjustment constraints using ``standalone marginal utilities'' as the optimization variables. A key insight of~\cite{galeotti2020targeting} is that the desired (optimal) interventions are determined by the  eigenvectors of the graph adjacency matrix. Our work also considers the problem of maximizing welfare, but differs in two ways. First, unlike the single-activity model of \cite{galeotti2020targeting}, we adopt a \textit{coupled}-activity network game model  to incorporate cross-activity and intra-activity network effects. 
Second, \cite{galeotti2020targeting} primarily focus on problems where the  sum of the squares of all the agents' marginal utilities is constrained to be within a fixed budget, whereas we focus on problems that impose 
heterogeneous upper bounds individually on each of the $n$ per-unit profits (and thereby on the $n$  per-unit premiums). In the context of interventions for sustainability, it is reasonable to impose an upper bound on the per-unit premium awarded to every agent in the network, as certification costs are also applied on a per-unit basis. We can verify that this is equivalent to imposing an upper bound on every agent's marginal utility individually. However, our setting does not admit a natural justification for bounding the \textit{sum of the squares} of the marginal utilities. Besides, such a constraint would add to the computational complexity of the problem by making it a quadratically constrained non-convex program, which explains why \cite{galeotti2020targeting} do not obtain closed-form optimal solutions to their  design problem for the cases they consider. While~\cite{galeotti2020targeting} focus on welfare maximization subject to one quadratic budget constraint and single-activity equilibrium constraints, we focus on welfare maximization subject to multiple affine budget constraints and coupled-activity equilibrium constraints. Moreover, the welfare function in our problem encapsulates strategic substitutabilities as well as cross-activity and intra-activity complementarities.

It is also worth noting that~\cite{sun2023structural} unifies the two lines of work -- those that study network interventions with those that study characteristic inteventions -- by showing that for every network intervention in a quadratic-utility game 
there exists a characteristic intervention that yields the same post-intervention Nash equilibrium. However,~\cite{sun2023structural}, too, focuses on the single-activity case unlike our work.

\textit{\textbf{Interventions in Multi-Activity Network Games:}} 
The recent work~\cite{kor2023multi} considers the 
problem of maximizing welfare in a network game involving an arbitrary number of interdependent activities subject to quadratic price adjustment constraints. 
It provides explicit closed-form expressions for price adjustments that are asymptotically optimal in the limit as the price adjustment budget approaches either zero or infinity. There are, nevertheless, two major differences between our work and~\cite{kor2023multi}. First, we impose both lower bound and upper bound constraints on the post-intervention per-unit prices rather than just a single quadratic constraint. This enables us to consider a wider range of price adjustment budgets rather than just the limiting cases of vanishingly small or infinite budgets.  Second, our intervention design problems impose distinct sets of constraints on the the agents' effort levels in distinct activities. As a result, our methodology involves the additional step of analyzing the algebraic expression for the welfare function and leveraging its dependence on  prices and network effect parameters.  

\section{Essential Feasibility of Interventions}\label{sec:essential}  

We now identify conditions under which the intervention design problems $\pbold$ and $\pbr$ result in ``useful'' pricing policies -- those that ensure  positive reductions in the aggregate unsustainable effort at equilibrium. This motivates the concept of \textit{essential feasibility}, which we define below.

\begin{definition}
    [\textbf{Essential Feasibility}] \label{def:essential_feasibility}  The intervention design problem $\pbold$ (respectively, $\pbr$) is \textit{essentially feasible} if  for some large enough tolerance $\tau^B$, its feasible set contains a policy $p^A$ (respectively, $(p^A,p^B)$) that ensures $\sum_{i=1}^n x_i^{B*} < \sum_{i=1}^n x_{i}^{B0}$ in the post-intervention equilibrium. 
    If the problem is not essentially feasible, we say that it is \textit{essentially infeasible}.
\end{definition}

Let us examine the essential feasibility of $\pbold$ before considering that of $\pbr$. 
 
 We can verify that the feasible set of $\pbold$ has a non-empty interior for a wide range of parameter values. Indeed, we see from~\eqref{eq:price_up_second} that every feasible policy other than the pre-intervention policy $p^{A0}$ increases the per-unit prices of sustainable goods for either all or a subset of the agents, thereby increasing the aggregate incentive for sustainable production. Naively, one might expect this increase to lead to a positive reduction in the aggregate unsustainable effort at equilibrium, but this conjecture requires careful scrutiny. While an increase in the incentive for sustainable production leads to an increase in the equilibrium levels of sustainable effort, this increase has two conflicting effects on the incentive for unsustainable production: (a) via $\beta$ (the intra-concession substitutability), it \emph{decreases} the incentive to produce unsustainably (i.e., as the agents increase their participation in sustainable production, their utility reduction from dividing  effort  between the two activities causes them to decrease their effort towards unsustainable production), and (b) via cross-activity agent-to-agent complementarities (quantified by $\mu$), it \emph{increases} the incentive to produce unsustainably, because higher  effort in activity $A$  implies more avenues for synergistic cross-activity interaction, thereby contributing positively towards the effort in activity $B$ (unsustainable production).  
 
 To find conditions under which one of these effects dominates the other, we need to examine the matrix $M_\Delta:=M^- - M^+$, the negative of which, according to~\eqref{subeq:x-B}, quantifies the changes in the equilibrium levels of unsustainable effort $\{x_i^B:i\in [n]\}$  for unit changes in the prices $\{p_i^A:i\in [n]\}$. In particular, we need to find the conditions under which $M_\Delta$ is positive, as these would be the conditions under which raising one or more of $\{p_i^A:i\in[n]\}$ reduces the values of $\{x_i^B:i\in [n]\}$. By Definition~\ref{def:essential_feasibility}, this will also lead to conditions under which $\pbold$ is essentially feasible. Our first result establishes these conditions.

Before we state this result, a clarification and some notation definitions are in order. First, we assume, without loss of generality, that the strategic interaction network is connected (if it is disconnected, the  discussion and results in the remainder of this section apply to each component individually). Next, we let $\dmin\in\{0,1,2,\ldots\}$ denote the minimum node degree of the strategic interaction network. Additionally, we define the following mutually exclusive outcomes:
 \begin{enumerate}
     \item [$s^+$:] $M_\Delta$ has positive entries and the equilibrium value of the aggregate unsustainable effort $\sum_{i=1}^n x_i^{B*}$ decreases monotonically in each of the prices $\{p_j^A:j\in[n]\}$ of sustainable goods.
     \item [$s^-$:] $M_\Delta$ has negative off-diagonal entries and the equilibrium value of the aggregate unsustainable effort $\sum_{i=1}^n x_i^{B*}$ increases monotonically in each of the prices $\{p_j^A:j\in[n]\}$ of sustainable goods.
 \end{enumerate}

\begin{theorem}\label{thm:one} 
Under Assumptions~\ref{assum:uniqueness} and~\ref{assum:mu-delta}, the following assertions apply to the problem $\pbold$. 
\begin{enumerate} 
    \item [(i)] If $\mu<\beta\delta$, then  $s^+$ is true and $\pbold$ is essentially feasible regardless of the network structure. 
    \item [(ii)] If $\mu>\beta\delta$,  there exist $G\in \{0,1\}^{n\times n}$ and $\delta>0$ for which $s^-$ is true and $\pbold$ is essentially infeasible.
    \item [(iii)] If 
    $\mu >\max\left\{\frac{2\beta}{1+\beta^2}\delta, \frac{\beta}{\dmin} \right\} $, then $s^-$ is true and $\pbold$ is essentially infeasible regardless of the network structure.
\end{enumerate}
\end{theorem}

Assertions (i) and (iii) of Theorem~\ref{thm:one} provide a necessary condition and a sufficient condition for the essential feasibility of $\pbold$. Condition (i) ensures that  the aggregate unsustainable effort in the network can be reduced by increasing the per-unit prices of sustainable goods provided the cross-activity network effect is dominated by the combination of the intra-activity network effect (quantified by $\delta$) and intra-concession substitutability (quantified by $\beta$). On the other hand, (ii) states that essential feasibility is violated if the cross-activity network effect is strong (as quantified by  $\max\left\{\frac{2\beta}{1+\beta^2}\delta, \frac{\beta}{\dmin} \right\} $).  In practice, the condition $\mu<\beta\delta$ 
can be enforced by monitoring and limiting the covert interactions that take place between certified concessions and listed entities on one side, and non-certified concessions and shadow companies on the other side. 

For $\mu\in \left(\beta\delta, \max\left\{\frac{2\beta}{1+\beta^2}\delta, \frac{\beta}{\dmin} \right\}  \right)$, the essential feasibility of $\pbold$ depends on the structure of the network (as captured by $G$) and the values of $\delta$ and $\beta$. Nevertheless,  (ii) asserts that for every value of $\mu$ in this range, there exists a network and a magnitude of the intra-activity network effect parameter $\delta$ for which $\pbold$ is essentially infeasible. 
The takeaway  is that, without restricting the covert interactions between sustainable and unsustainable concessions, 
it may be impossible to suppress unsustainable effort by merely increasing the prices of sustainable goods.  Therefore, unless stated otherwise, we will henceforth assume the condition in (i), which we state formally below. 

\begin{manualtheorem}{2'}\label{assum:domination}
The cross-activity network effect is dominated by the intra-activity network effect and the intra-concession substitutability, i.e., $\mu<\beta\delta$.
\end{manualtheorem}

Note that since $\beta<1$, Assumption~\ref{assum:domination} implies Assumption~\ref{assum:mu-delta}, thereby superseding the latter.

    Having explained the central message of Theorem~\ref{thm:one}, we now highlight its implications.

    First, suppose we are interested in a pricing policy $p^A$ that maximally suppresses the aggregate level of unsustainable effort subject to the price adjustment constraints~\eqref{eq:price_up_second} and~\eqref{eq:price_threshold_second}. Then, as a straightforward consequence of Theorem~\ref{thm:one}, the optimal policy in this case would be the price-maximizing policy $\pmax$  under Assumption~\ref{assum:domination} and the pre-intervention policy $p^{A0}$ if $\mu>\beta\delta$.
    
    Second, observe from~\eqref{subeq:x-B} that $\sum_{i=1}^n x_i^{B*} - \sum_{i=1}^n x_i^{B0} = -b_\Delta^\top (p^A - p^{A0})$, where $b_\Delta:= M_\Delta\allone$. Thus, the greater the value of $|b_{\Delta i}|$, the more sensitive is the aggregate   unsustainable effort $\sum_{i=1}^n x_i^{B*}$ to changes in the prices of sustainable good produced by agent $i$. Considering that $b_{\Delta i }$ is positive by Theorem~\ref{thm:one}-(i), the preceding observation implies that  agents with higher values of $b_{\Delta i}$ are more central because they are better-positioned than other agents for the purpose of transferring the unsustainability-suppressing effects of  their individual price interventions  to  the rest of the network. This motivates the following definition, which we use to interpret some of our subsequent results.

    \begin{definition}\label{def:centrality}
        The \textit{centrality} of an agent $i\in [n]$ is $b_{\Delta i}=(M_{\Delta}\allone)_i = \sum_{j=1}^n (M_{ij}^- - M_{ij}^+)$. 
    \end{definition}

Finally, Theorem~\ref{thm:one} has an important implication for policy design. Note that the entries of $x^{A*}$ and $x^{B*}$ denote equilibrium effort levels in our model, and hence, the expressions given by~\eqref{eq:equilibrium} are valid in the context of our intervention design problem if and only if they are non-negative. However, under Assumption~\ref{assum:domination}, it follows from Theorem \ref{thm:one}-(i)  that $M_\Delta$ is positive, and hence, there exists a price threshold $\plim\in\R^n_{\ge 0}$ such that increasing the prices of sustainable goods beyond $\plim$ (i.e., choosing $p^A\ge \plim$) causes the right-hand-side of~\eqref{subeq:x-B} to be negative if the prices of unsustainable goods are held fixed (as is the case with $\pbold$). To handle such cases, it is imperative to find non-negative expressions for  equilibrium effort levels. This is accomplished by our next main result (Theorem~\ref{thm:non-negative}), which 
holds under the following  assumption.

\begin{assumption}\label{assum:positive_pre-intervention}
    $x^{A0}>\allzero$ and $x^{B0}>\allzero$, i.e., each agent expends a positive amount of effort in both activities in the pre-intervention equilibrium.
\end{assumption}

If Assumption~\ref{assum:positive_pre-intervention} is violated, our results can be naturally refined under the milder assumption that the partial first derivatives of every agent's  utility function with respect to the agent's own effort levels are zero in the pre-intervention equilibrium, i.e.,  $\frac{\partial u_i}{\partial x_i^A}\big\lvert_{(x^A,x^B)=(x^{A0},x^{B0})} = \frac{\partial u_i}{\partial x_i^B}\big\lvert_{(x^A,x^B)=(x^{A0},x^{B0})} = 0$ for all $i\in [n]$.
 
We state Theorem~\ref{thm:non-negative} below.

\begin{theorem}\label{thm:non-negative}
    Suppose Assumptions~\ref{assum:uniqueness},~\ref{assum:domination}, and~\ref{assum:positive_pre-intervention} hold, and suppose $p^A\ge p^{A0}$ and $p^{A0}\le \plim$, where \\${\plim := M_\Delta^{-1}(M^+ + M^-)p^{B0}}$ exists according to Lemma~\ref{lem:invertibility}. Then, the Nash equilibrium of the network game $\left([n], \R^n_{\ge 0}\times \R^n_{\ge 0} ,\{u_i\}_{i\in[n]} \right)$ has the following properties.
    \begin{enumerate}
        \item [(i)] If $ p^A<\plim$, then there exists a unique equilibrium with positive effort levels in activity $A$. This equilibrium is given by~\eqref{eq:equilibrium}.
        \item [(ii)] If $p^A \ge \plim$, then the equilibrium is unique and is given by $x^{A*} = (I-\delta G)^{-1} p^A$ and $x^{B*} = \allzero$.
        \item [(iii)] If there exist $i,j\in [n]$ such that $p^A_i> \plim_i$ and  $p^A_j \le \plim_j$, then the following are true.
        \begin{enumerate}
            \item If $(M^+ - M^-)p^A +  (M^+ + M^-)p^B\ge \allzero$, then~\eqref{eq:equilibrium} defines the unique equilibrium with $x^{A*}> \allzero$.
            \item If $(M^+ - M^-)p^A +  (M^+ + M^-)p^B\ngeq \allzero$, then there exists a non-empty set of indices  \\ ${S\subseteq\left\{i\in[n]:\left((M^+ - M^-)p^A +  (M^+ + M^-)p^B\right)_i< 0 \right\}}$ such that for $T:=[n]\setminus S$, the matrix inverse $\left(I - \delta G_S - (\delta^2+\mu^2) G_{ST}G_{TS}\right)^{-1}$ exists and the effort levels given by\begin{subequations}\label{eq:aliter}
\begin{align}
\begin{split}
    \hat x^{A*}(S) &= \frac{1}{2}\left( (M^+ + M^-)p^A + (M^+ - M^-)\hat p^B \right), 
\end{split}\\
\begin{split}
    \hat x^{B*}(S) &= \frac{1}{2}\left( (M^+ - M^-)p^A + (M^+ + M^-)\hat p^B \right),
\end{split}
\end{align}    
\end{subequations}
where       \begin{align*}
                \hat p^B_S&:=\hat p^{B}_{S}(S,p^A)\cr
                &:= \left((M^++M^-)_S \right)^{-1}\left(  (M^- - M^+)_{ST} p^{A} _{T} + (M^- - M^+)_S p^{A} _{S} -(M^+ + M^-)_{ST} p_T^{B}  \right)
            \end{align*}             and ${\hat p^B_T:= p^B_T}$, satisfy $\hat x^{B*}(S)\ge \allzero$. If $S$ is a minimal set with the property $\hat x^{B*}(S)\ge \allzero$, then $(\hat x^{A*}(S),\hat x^{B*}(S))$ is the unique Nash equilibrium with positive effort levels in activity $A$. 
        \end{enumerate}
    \end{enumerate}
\end{theorem}

Theorem~\ref{thm:non-negative} states that the equilibrium is unique if either all or none of the post-intervention prices exceed the corresponding entries of the threshold vector $\plim$. Another case in which the equilibrium is unique is when the right-hand side of~\eqref{subeq:x-B} is non-negative. In all other cases, the expressions in~\eqref{eq:aliter} define an equilibrium in which  the effort levels are identical to those given by~\eqref{eq:equilibrium}, except that for a subset $S$ of the agents, \eqref{eq:aliter} replaces $p^B_S$, the true price incentives for unsustainable production, with the hypothetical price incentives given by $\hat p^B_S$. We will use Theorem~\ref{thm:non-negative} to broaden the generality of the results that follow. 

Finally, we remark that the essential feasibility of the redistribution problem $\pbr$ is not contingent upon Assumption~\ref{assum:domination} being satisfied provided at least one agent can be penalized for engaging in unsustainable production. We summarize this observation below.

\begin{proposition}\label{prop:one}
    Under Assumptions~\ref{assum:uniqueness} and~\ref{assum:mu-delta}, $\pbr$ is essentially feasible if  $\max_{i\in [n]}\rho_i>0$.
\end{proposition}

In other words, imposing penalties can reduce aggregate unsustainable effort even if there exist no restrictions on the synergistic interactions between sustainable and unsustainable activities.

\section{Optimal Policies}\label{sec:optimal}

Our goal is to solve the intervention design problems $\pbold$ and $\pbr$ and  interpret the  optimal policies. The first step in solving $\pbold$ is to express its objective function in terms of the decision vector $p^A$. 
 we evaluate the welfare at the post-intervention equilibrium  by plugging in the expressions for equilibrium effort levels~\eqref{eq:equilibrium} into the utility expression~\eqref{eq:main}. This yields
\begin{align}\label{eq:welfare_expression}
    \phi(p^A)&:= \sum_{i=1}^n u_i\left( x^{A*},  x^{B*}, p^A,p^{B0} \right)\cr
    &= \frac{ 1}{4}\left((p^A)^\top Q p^A - v^\top p^A + (p^{B0})^\top Q p^{B0} \right),    
\end{align}
where $Q:=  (1+\beta)(M^+)^2 + (1-\beta)(M^-)^2\in\R^{n\times n}$ 
and  $v:=2Rp^{B0}\in\R^n$ with $R$ being defined as ${R:=  (1-\beta)(M^-)^2 - (1+\beta)(M^+)^2\in\R^{n\times n}}$. On the basis of~\eqref{eq:welfare_expression}, we say that, between two candidate pricing policies $p^A,\tilde p^A\in\R^n_{\ge 0}$, the policy $\tilde p^A$ is more \textit{welfare-generating} than $p^A$ if $\phi(\tilde p^A)>\phi(p^A)$, i.e., if $\tilde p^A$ results in a higher post-intervention welfare. 
Furthermore, we use \eqref{eq:equilibrium} to express~\eqref{eq:unsustainable_effort} as  $b_\Delta^{\top} p^A \geq k_0$, where 
$k_0:= \allone^\top (M^+ + M^-)^\top p^{B0} - 2\tau^B$. The constraint $b_\Delta^{\top} p^A \geq k_0$ quantifies the extent to which the prices of sustainable goods must be raised to keep the aggregate unsustainable effort below the tolerance $\tau^B$. 

With the help of~\eqref{eq:welfare_expression}, $\pbold$ can be rewritten as follows: 
\begin{subequations}\label{eq:constraints}
\begin{align}
\begin{split}
\pbold\textbf{: }\text{Maximize }  (p^A) ^{\top} Q &p^A -v^{\top} p^A + (p^{B0}) ^{\top} Q p^{B0} 
\end{split}\nonumber\\
\begin{split}
    \text { s.t. }\label{eq:unsustainable_second_modified} \quad b_\Delta^{\top} & p^A \geq k_0 
\end{split}\\
\begin{split}
\label{eq:price_up_second_modified} &p^A \geq p^{A0}
\end{split}  \\
\begin{split}
\label{eq:price_threshold_second_modified} &p^A \leq p_{\max }.
\end{split}
\end{align}
\end{subequations}

Thus, the objective function of $\pbold$ depends on $p^{B0}$, $Q$, and $R$ (via $v$). We now establish the important properties of the matrices $Q$ and $R$ below.

\begin{lemma}\label{lem:positive-definite}
    Under Assumptions~\ref{assum:uniqueness} and~\ref{assum:domination}, the following assertions are true.
    \begin{enumerate}
        \item [(i)] $Q$ is positive-definite and non-negative.
        \item [(ii)] $R$ is non-negative and $v$ is a positive vector.
        \item [(iii)] $Q\ge R$.
    \end{enumerate}
\end{lemma}

Lemma~\ref{lem:positive-definite} clarifies that welfare as given by~\eqref{eq:welfare_expression} receives both positive and negative contributions from $p^A$. Indeed,  per-unit prices of sustainable goods have two conflicting effects on welfare, as enumerated below. We will refer to these effects to interpret the optimal policies.
\begin{enumerate}
    \item [] \text{\textbf{Direct Effect ($(p^A)^\top Q p^A$)}:} Sustainable good prices contribute to the utility gains from sustainable effort, thereby contributing positively to welfare.
    \item [] \text{\textbf{Substitutability Effect ($-v^\top p^A$)}:} Via the intra-concession substitutability ($-\beta x_i^{A*} x_i^{B*}$), the positive contributions of sustainable good prices towards the agents' sustainable effort levels result in keeping the incentive for unsustainable production in check, thereby making negative contributions towards the agents' utility gains from unsustainable effort. 
\end{enumerate}
Therefore, a key challenge in optimizing welfare lies in determining which of these two effects dominates the other.

Besides, Lemma~\ref{lem:positive-definite} implies that $\pbold$ is equivalent to the minimization of the strongly concave quadratic function~\eqref{eq:welfare_expression} 
over the polytope $\mathcal P$ defined by the $2n+1$ constraints~\eqref{eq:unsustainable_second_modified} -~\eqref{eq:price_threshold_second_modified},with more than ${2n \choose n}=\Omega(2^n)$ extreme points in the worst case. Hence, $\pbold$ can be expressed as~\cite[Problem (1)]{pardalos1991quadratic}. In general, such problems are NP-hard (\cite{sahni1974computationally}). 
Nevertheless, as the following result and its corollary illustrate, we can solve $\pbold$ analytically for cases of high practical significance. 

\begin{theorem}\label{thm:general_case}
    Suppose Assumptions~\ref{assum:uniqueness},~\ref{assum:domination}, and~\ref{assum:positive_pre-intervention} hold. Then, for any two candidate pricing policies $p^{A(1)},p^{A(2)}\in \R^n$ that satisfy $p^{A0}\le p^{A(2)}\le p^{A(1)}\le \plim$, if we have
    \begin{align}\label{eq:thm_condition_1}
        Q(p^{A(1)} + p^{A(2)})\ge v,
    \end{align}
    then $\phi(p^{A(1)})\ge \phi(p^{A(2)})$, i.e.,  $p^{A(1)}$ achieves higher welfare than $p^{A(2)}$. On the other hand, if 
    \begin{align}\label{eq:thm_condition_2}
        Q(p^{A(1)} + p^{A(2)})< v,
    \end{align}
    then $\phi(p^{A(1)}) <\phi(p^{A(2)})$. In particular, if~\eqref{eq:thm_condition_1} holds with $p^{A(1)}=\pmax\le \plim$ and $p^{A(2)}=p^{A0}$, then the policy $\pmax$ is an optimal solution of $\pbold$. Under this policy, aggregate unsustainable effort is minimal in the post-intervention equilibrium. On the other hand, if \eqref{eq:thm_condition_2} holds with $p^{A(1)}=\pmax$ and $p^{A(2)}=p^{A0}$
    and the policy $p^{A0}$ satisfies the tolerance constraint~\eqref{eq:unsustainable_second_modified} on aggregate unsustainable effort, then $p^{A0}$ is an optimal solution of $\pbold$. Under this policy, aggregate unsustainable effort is maximal in the post-intervention equilibrium. 
    \def\allone{\mathbf 1}
\end{theorem}

\begin{remark}
    If there exists a subset of indices $S\subset [n]$ such that $p_i^{A(1)}\ge \plim_i$ for $i\in S$, the conclusions of Theorem~\ref{thm:general_case} continue to hold provided  $p^{A}\ge p^{B0}$, i.e., the prices of sustainable goods are constrained to be no less than those of unsustainable goods. 
\end{remark}

Theorem~\ref{thm:general_case} can be interpreted with the help of the following lemma.

\begin{lemma}\label{lem:interpret}
    Let $\psi:\R^n\to \R^n$ be the vector-valued function defined by 
    $$
        \psi(p^A) :=\left|(1+\beta)(M^+)^2 (p^A+p^{B0}) + (1-\beta)(M^-)^2(p^A - p^{B0}) \right|.
    $$ Suppose the assumptions made  in Theorem~\ref{thm:general_case} hold with $p^{A(1)}\ne p^{A(2)}$. Then, the conditions~\eqref{eq:thm_condition_1} and~\eqref{eq:thm_condition_2} are equivalent to $\psi(p^{A(1)})\ge \psi(p^{A(2)})$ and $\psi(p^{A(1)})< \psi(p^{A(2)})$, respectively.
\end{lemma}

Considering Lemma~\ref{lem:interpret}, Theorem~\ref{thm:general_case} asserts the following: given two pricing policies $p^{A(1)}$ and $p^{A(2)}$ such that the prices prescribed by one policy are no less than those prescribed by the other, if  all the entries of $\psi$ are maximized by the same policy, (i.e., $\psi(p^{A(i)})= \max\{\psi^{A(1)}, \psi^{A(2)}\}$ for some $i\in\{1,2\}$), then $p^{A(i)}$ is the more welfare-generating of the two policies. Observe that for any given policy $p^A$, the $i$-th entry of $\psi(p^A)$ is the magnitude of a combination of two quantities, namely, $p_i^A+p_i^{B0}$, which is the $i$-th agent's \textit{total incentive} for production from either activity, and $p_i^A-p_i^{B0}$, which captures the extent to which the agent's monetary incentive for production is skewed or \textit{biased} towards  the sustainable activity.  Hence, for the policy $p^{A(1)}$ to be more welfare-generating than $p^{A(2)}$, it is not enough for it to achieve a greater total incentive for production; it must instead ensure that the combined effect $\psi(p^A)$ of the total incentive $p_i^{A}+p_i^{B0}$ and the incentive bias $p_i^{A}-p_i^{B0}$   is maximized. This can be further explained as follows: between two candidate post-intervention policies $p^{A(1)}$ and $p^{A(2)}$ where $p^{A(1)}\ge p^{A(2)}$, the former will have a greater \textit{direct effect} on the welfare as captured by the total incentive term $(1+\beta)(M^+)^2(p^A + p^{B0})$. However,  if it achieves a lower magnitude of incentive bias than $p^{A(2)}$ (i.e., if $|p^{A(1)}- p^{B0}|<|p^{A(2)}-p^{B0}| 
$, which happens, for example, when $p^{A(2)}-p^{B0}< p^{A(1)}- p^{B0}< \mathbf 0$), then it will cause the agents' effort to be distributed more evenly between the two activities, thereby causing the welfare loss from  the \textit{substitutability effect} to be greater. Therefore, $p^{A(1)}$ will achieve a higher welfare if and only if both these effects combined, as quantified by the weighted sum $\psi$ of the total incentive and the incentive bias, favor $p^{A(1)}$ over $p^{A(2)}$.  

Theorem~\ref{thm:general_case} also describes the aggregate unsustainable effort at optimality in the post-intervention equilibrium. In the best case, $\pmax$ is the optimal solution of $\pbold$, because in such a case it achieves the dual objectives of maximizing welfare  and minimizing aggregate unsustainable effort over the set of all feasible policies. Other feasible policies result in higher aggregate unsustainable effort as compared to $\pmax$, even though they are more welfare-generating than $\pmax$ when the maximum feasible prices are small (as quantified by~\eqref{eq:thm_condition_2} with $p^{A(1)}=\pmax$). In the worst case, these prices are so small that~\eqref{eq:thm_condition_2} holds with $p^{A(2)}=p^{A0}$,  and the optimal solution of $\pbold$ is the pre-intervention policy, which maximizes welfare while also maximizing aggregate unsustainable effort over the set of all feasible policies.

\def\I{\mathcal I}
\begin{remark}
     In the unlikely situation that $\pmax$ and $p^{A0}$ satisfy neither~\eqref{eq:thm_condition_1} nor~\eqref{eq:thm_condition_2}, Theorem~\ref{thm:general_case} can still be used to search for the optimal solution among the extreme points of $\mathcal P$ as follows. Let $p^{(0)}:=p^{A0}$, and for each $i\in [n]$, define the extreme point policy $p^{(i)}$ as the policy at which the constraint $b_\Delta^\top p^A \ge k_0$ from \eqref{eq:unsustainable_second_modified} and the constraints $p^A_j \ge p^{A0}_j$ from~\eqref{eq:price_up_second_modified} are active for all $j\in [n]\setminus\{i\}$. Let $\mathcal I:=\{i\in [n]\cup\{0\}: p^{(i)}\in\mathcal P\}$ index the subset of these policies that are  feasible. Then, we can partition the index set $\I$ using~\eqref{eq:thm_condition_1} and~\eqref{eq:thm_condition_2} as ${\I = \I^+ + \I^- + \I^0}$, where ${\I^+ := \{i\in \I: Q(p^{(i)} +\pmax)\ge 2Rp^{B0}\}}$, ${\I^- := \{i\in \I: Q(p^{(i)} +\pmax)< 2Rp^{B0}\}}$, and ${{\I^0:= \I \setminus\{\I^+ + \I^-\}}}$. By Theorem~\ref{thm:general_case}, for all $i\in \I^+$ and $p^A$ satisfying $p^{(i)}\le p^A\le \pmax$, we have $\phi(\pmax)\ge \phi(p^A)$. Since every extreme point $p$ of $\mathcal P$ that satisfies $p_i = \pmax_i$ also satisfies $p^{(i)}\le p\le \pmax$, it follows that the extreme points of $\mathcal P$ at which the constraint $p_i \le \pmax_i$ is active are no more  welfare-generating than $\pmax$ and can therefore be disregarded. This argument narrows down our search space to the union of $\{\pmax\}$ and the extreme points of $\mathcal P$ at which the constraints $\{p^A_i\le\pmax_i:i\in\I^+\}$ are inactive.  
    Using similar arguments, we can further restrict the search space to the extreme points of $\mathcal P$ at which the constraints $\{p_i^A\le \pmax_i:i\in \I^+\cup \I^-\}$ are inactive. Note that there are at most ${2n+1-|\I^+| -|\I^-| \choose n}$ such extreme points. Finally, we can repeat the above procedure after re-defining $p^{(i)}$ for each $i\in \I^0$ as the extreme point of $\mathcal P$ at which the constraints $p_i^A \le \pmax_i$ from~\eqref{eq:price_threshold_second_modified} and $p_j^A\ge p_j^{A0}$ from~\eqref{eq:price_up_second_modified} are active for all $j\in [n]\setminus \{i\}$ and by updating the index sets $\I^+$, $\I^-$, and $\I^0$ accordingly. Each iteration of this procedure removes at least ${2n + 1 \choose n} - {2n + 1 - |\I^+|-|\I^-| \choose n}$ extreme points from the search space, thereby reducing the search space cardinality exponentially in the worst case. 
\end{remark}

We now derive a corollary of Theorem~\ref{thm:general_case} that has practical implications for policy design.
\begin{corollary}\label{cor:first}
Under Assumptions~\ref{assum:uniqueness},~\ref{assum:domination}, and~\ref{assum:positive_pre-intervention}, the following statements are true.
    \begin{enumerate} 
        \item [(i)] The policy $\pmax$ is optimal if either of the following conditions holds, in which case aggregate unsustainable effort is minimal in the post-intervention equilibrium.
        \begin{enumerate}
            \item 
            $p^{A0}\ge p^{B0}$, i.e., sustainable good prices are no less than unsustainable good prices.
            \item $Qp^{A0} >  R p^{B0}$, i.e., the welfare gradient is positive at the pre-intervention prices $p^{A0}$. 
            \item $\pmax - p^{B0} \ge  p^{B0} - p^{A0}$, i.e., replacing $p^{A0}$ with $\pmax$ replaces the deficits $\{p_i^{B0}-p^{A0}_i:i\in[n]\}$ in the prices of sustainable goods with bonuses no smaller than these deficits. 
        \end{enumerate}
        \item [(ii)] (\textbf{Insufficient Price Adjustment Budget})  Suppose $p^{A0}$ satisfies~\eqref{eq:unsustainable_second_modified}, the tolerance constraint on aggregate unsustainable effort. Then, we  have $p^*= p^{A0}$ if $\pmax_i<\frac{(R p^{B0})_j}{(Q\mathbf 1)_j}$ for all $i,j\in [n]$. In this case, aggregate unsustainable effort is maximal in the pre-intervention equilibrium.
    \end{enumerate}
\end{corollary} 

Statement (i)-(a) of Corollary~\ref{cor:first} applies to cases in which no agent's per-unit price of sustainable goods is less than their per-unit price of unsustainable goods even before any intervention is applied. 
 Since the computation of sustainable good prices involves adding sustainability premiums to the base prices and subtracting certification costs (see Appendix~\ref{sec:monetary_component}). All situations in which the premiums exceed the certification costs fall into this category. This is also the best-case scenario because it enables a single policy to maximize both welfare and the reduction in aggregate unsustainable effort. 

Since both $Q$ and $R$ are positive matrices, statement (i)-(b) has a similar interpretation to that of (i)-(a): if the pre-intervention prices of sustainable goods are significant relative to those of unsustainable goods, $\pmax$ is an optimal policy. The condition $Qp^{A0}>Rp^{B0}$ is equivalent to the gradient of the welfare function~\eqref{eq:welfare_expression} being positive in all its entries over the entire feasible region. Moreover, we can use Lemma~\ref{lem:positive-definite}-(iii) to verify that this  condition is weaker than  (i)-(a) and can be satisfied even when sustainable good prices are smaller than their unsustainable counterparts. Therefore, while (i)-(a)  can be verified even when we have no knowledge of the structure of the strategic interaction network, (i)-(b) can be useful in cases where (i)-(a) is violated. For example, this can happen when sustainability premiums are not adequate to compensate for certification costs (\cite{chinadialogue2021}). 

Complementing (i)-(a) and (i)-(b),  (i)-(c) provides a condition that imposes no restrictions on pre-intervention prices. Instead, it states that in cases where the maximum feasible prices of sustainable goods offer bonuses (as measured relative to the prices of unsustainable goods) that are at least as great as the deficits in the pre-intervention prices of sustainable goods, $\pmax$ is the welfare-maximizing policy. 

In contrast to (i)-(a) - (i)-(c), statement (ii)  applies to situations in  which the planner has ample tolerance for unsustainable effort but a narrow price adjustment window. In such cases, increasing the prices of sustainable goods within the ranges defined by $\pmax$ leads to a drop in welfare due to the substitutability effect dominating the direct effect of price raises on welfare, resulting in the pre-intervention policy being optimal. 

Importantly, Corollary \ref{cor:first}-(ii) informs us that enforcing Assumption~\ref{assum:domination} by restricting cross-activity interactions between sustainable and unsustainable concessions is not sufficient in itself   to reduce unsustainable effort in a manner that does not compromise welfare. To achieve this, it is also essential to have either large enough pre-intervention prices or large enough maximum feasible prices of sustainable goods, as quantified by Corollary  \ref{cor:first}-(i). However, inadequate restrictions on cross-activity interactions can make welfare improvement conflict sharply with the goal of reducing unsustainable effort, as shown by the result below.

\begin{proposition}\label{prop:two}
    There exist $G\in \{0,1\}^{n\times n}$ and $\delta>0$ such that violating Assumption~\ref{assum:domination} with $\mu>\beta\delta$ causes the optimal solution of $\pbold$ to result in maximal unsustainable effort levels $\{x_i^{B*}\}_{i=1}^n$  in the post-intervention equilibrium.
\end{proposition}

To summarize, weak enforcement of restrictions can amplify the cross-activity effect, potentially causing the welfare-maximizing policy to maximize aggregate unsustainable effort over the set of feasible pricing policies.

    Note that Theorem~\ref{thm:general_case} provides expressions for the optimal policy that do not depend on $G$, $\mu$, or $\delta$. However, the conditions under which these policies are optimal, including Assumption~\ref{assum:domination}, indeed depend on $G$, $\beta$, $\mu$, and $\delta$. This is in contrast to  the optimal solution~\eqref{eq:pre-modeling} of the preliminary intervention design problem $\boldsymbol P_0$, as the thresholds defining~\eqref{eq:pre-modeling} only depend on $\beta$. 
   Therefore, for the  problem of maximizing welfare under sustainability constraints, it is in general essential to take into account the structure of the network topology as captured by $G$ and the effects of strategic interactions on the individual utilities as captured by the network effect parameters $\delta$ and $\mu$. We investigate this further in our case study in Section~\ref{sec:empirical} and Appendix~\ref{sec:comp_stat}.

Observe now that the idea of replacing price deficits with bonuses (Corollary \ref{cor:first}-(i)-(c)) suggests that in cases where $p^{B0}\ge p^{A0}$,  we might be able to eliminate deficits  in the prices of sustainable goods by swapping these prices with those of  unsustainable goods. This motivates the idea of redistribution of monetary incentives, as formalized by the problem $\pbr$, defined in~\eqref{eq:redistribution}. Recalling that the maximum feasible penalties and the price-adjustment budgets are given by $\rhomax\in\R^n_{\ge 0}$  and  $b\in \R^n_{\ge 0}$, respectively, we now characterize the redistribution policy that optimally solves $\pbr$.

\begin{theorem}\label{thm:redistribution}
Under Assumptions~\ref{assum:uniqueness},~\ref{assum:domination}, and~\ref{assum:positive_pre-intervention}, the policy   given by ${(p^{A*}_\emph{R}, p^{B*}_\emph{R}) := (p^{A0} +  \rhomax + b, p^{B0}-\rhomax)}$, which imposes maximum penalties and offers maximum premiums, is the policy that minimizes aggregate unsustainable effort over the set of all feasible policies  \eqref{eq:redistribution}. Under this policy,  every agent's unsustainable effort vanishes in the post-intervention equilibrium, i.e., $\sum_{i=1}^n x_i^{B*} = 0$,  if the budget $b$ and the maximum feasible penalties $\rhomax$ are large enough to satisfy $\rhomax\ge\rho^0$ and $b\ge b^0$, where $(\rho^0,b^0)\in \R^n\times \R^n$ is any pair of vectors satisfying
\begin{align}\label{eq:vanish}
    \rho^0 = \frac{1}{2}(M^-)^{-1}(M^+ p^{B0} - p^{A0} - M_\Delta b^0).
\end{align} 
Moreover, if $\rhomax$ and $b$ satisfy the \emph{budget-penalty sufficiency} condition
\begin{align}\label{eq:half-budget}
    \rhomax + \frac{b}{2}\ge p^{B0} - p^{A0},
\end{align}
then $(p_\emph{R}^{A*},p_\emph{R}^{B*})$ is also the optimal solution of $\pbr$, and under this policy, no agent's post-intervention utility is less than their pre-intervention utility.
\end{theorem}

The budget-penalty sufficiency condition~\eqref{eq:half-budget} is motivated by the notion of incentive bias $|p^{B0}-p^{A0}|$ as follows. Consider any  $i\in [n]$ for which  
 $p^{B0}_i-p^{A0}_i>0$ holds, so that the $i$-th scalar inequality in~\eqref{eq:half-budget} is non-trivial. Then, in the absence of a penalty, the post-intervention magnitude of the incentive bias affecting  agent $i$  (given by $|p_i^A - p_i^B|$)  will be greater than its pre-intervention magnitude (given by the price difference $p_i^{B0} - p_i^{A0}$) if and only if the price raise $p^A-p^{A0}$ is at least \textit{twice} this price difference. Therefore, for the incentive bias to increase in magnitude, the price adjustment budget $b_i$ must exceed $2(p_i^{B0}-p_i^{A0})$, or equivalently, half the budget must exceed the difference $p_i^{B0} - p_i^{A0}$. Extending this argument to incorporate penalties can be shown to yield~\eqref{eq:half-budget}. Thus, the joint sufficiency of the budget and the penalties is equivalent to the condition that every agent's incentive bias increases in magnitude after the intervention.  

Under this condition,  Theorem~\ref{thm:redistribution} gives a closed-form expression for the welfare-maximizing policy and guarantees that such a policy will also maximally suppress the aggregate level of unsustainable effort in the network. Moreover, the theorem ensures that no agent will be worse off after the intervention in the sense of having a reduced utility. It also shows that if   the budgets and penalties are large enough  (as quantified by~\eqref{eq:vanish}), the unsustainable activity vanishes entirely after intervention. Notably, Theorem~\ref{thm:redistribution} also applies to zero-budget scenarios provided the maximum feasible penalties are large enough to satisfy \eqref{eq:half-budget} with $b=0$.

However, if the budget-penalty sufficiency condition \eqref{eq:half-budget} is violated, then there is no guarantee that welfare can be improved in the absence of an external budget, as shown by the following result.

\begin{proposition}\label{prop:pessimistic}
    Suppose $b=0$ and $\rhomax<p^{B0}-p^{A0}$. Then the pre-intervention policy $(p^{A0},p^{B0})$ is an optimal solution of $\pbr$.
\end{proposition}

Proposition~\ref{prop:pessimistic} states that price redistribution cannot improve welfare if the budget is zero and the  penalties imposed are so small that they violate all of the $n$ scalar inequalities given by~\eqref{eq:half-budget}.

Finally, Theorem~\ref{thm:redistribution} has an important implication for the variant $\pbold$ of our intervention design problem. Indeed, observe that $\pbold$ is a special case of $\pbr$ with $b=\pmax - p^{A0}$ and $\rhomax=\allzero$. Substituting these expressions for $b$ and $\rhomax$ into the budget-penalty sufficiency condition~\eqref{eq:half-budget} reduces the inequality therein to 
\begin{align}\label{eq:equivalent}
    \pmax - p^{A0} \ge 2(p^{B0}-p^{A0}).
\end{align}
It can now be verified that the conditions defined in Corollaries~\ref{cor:first}-(i)-(a) and~\ref{cor:first}-(i)-(c) individually imply~\eqref{eq:equivalent}, which, according to Theorem~\ref{thm:redistribution}, further implies that $(\pmax,p^{B0})$ is an optimal solution of $\pbr$ under these conditions. Equivalently, $\pmax$ is an optimal solution of $\pbold$ under these conditions. This shows that Corollaries~\ref{cor:first}-(i)-(a) and~\ref{cor:first}-(i)-(c) are both subsumed by Theorem~\ref{thm:redistribution}. Moreover, in this case Theorem~\ref{thm:redistribution} goes further and guarantees that no agent's utility decreases after the intervention. We have thus shown that, in the absence of penalties, if either (i) the pre-intervention prices of sustainable goods exceed those of their unsustainable counterparts, or (ii) the external budget is sufficient to offset price deficits in $p^{A0}$
  with equally large price bonuses, then the welfare-maximizing policy will offer the maximum feasible premiums for sustainable production while ensuring that no agent is worse off after the intervention.

\section{Component-Wise Uniform Pricing}\label{sec:component-wise}

This section is motivated by the existence of concession networks for which the problem $\pbold$ can be simplified without losing practical utility, allowing it to be tackled analytically even when both~\eqref{eq:thm_condition_1}
 and~\eqref{eq:thm_condition_2}
 are violated. Consider, for example, the oil palm concession network of Indonesian New Guinea, in which the largest connected component is comprised of 18 concessions whereas all others consist of at most 10 concessions each (see~\cite{nusantaraatlasNusantaraAtlas} Atlas). In such networks with small numbers of concessions per connected component, it may be desirable to keep sustainability premiums and certification costs uniform within every region (connected component) to ensure equal treatment of all the cultivators in the region. This also makes  policy design more convenient as it reduces the number of optimization variables (i.e., the dimension of $p^A$) by an order of magnitude, thereby causing an exponential reduction in the number of extreme points of the feasible region of $\pbold$ in the worst case. Moreover, since strategic interactions across distinct connected components are constrained to occur over long distances and are therefore negligible, introducing component-wise uniform pricing also has the added advantage of condensing $Q$ to a diagonal matrix, which enables us to sinplify the analysis of the problem, as shown below.
    
    Suppose the concession network is composed of $c$ connected components with vertex sets $V_{(1)}, V_{(2)}, \ldots, V_{(c)}$ of sizes $|V_{(1)}|=n_1, |V_{(2)}|= n_2, \ldots, |V_{(\ell)} |= n_\ell$, respectively. In other words,  $G = \text{diag}[G_{(1)}, G_{(2)}, \ldots , G_{(c)} ]$ is a block-diagonal matrix composed of $c$ block matrices $\{G_{(\ell)}\in \{0,1\}^{n_\ell \times n_\ell } :\ell\in [c]\}$, each of which is irreducible because of the connectedness of the components. Treating the network of concessions as an estimate of the strategic interaction network, we can compute for each $\ell\in[c]$ the equivalents of $M^+$, $M^-$, $Q$, and $R$ for the $\ell$-th connected component as follows:
    \begin{align*}
        M^{+}_{(\ell)} &= ((1+\beta)I - (\delta+\mu)G_{(\ell)})^{-1},\cr
        M^{-}_{(\ell)} &= ((1-\beta)I - (\delta-\mu)G_{(\ell)})^{-1},\cr
        Q_{(\ell)} &= (1+\beta)(M^+_{(\ell)} )^2 + (1-\beta)(M^-_{(\ell)} )^2,\cr
        R_{(\ell)} &= (1+\beta)(M^+_{(\ell)} )^2 - (1-\beta)(M^-_{(\ell)} )^2.
    \end{align*}

    As our goal is to design a policy that enforces uniform pricing within every component of the network, we need our post-intervention price vector to be of the form $p^A =  \sum_{\ell=1}^c  p^A_{(\ell)}\left(\sum_{i\in V_{(\ell)}} e^{(i)}\right) $ where for each $\ell\in[c]$, $p^A_{(\ell)}$ denotes the per-unit price of sustainable goods produced by agents in component $\ell$.  For such a vector $p^A$, the welfare can be verified to be given by
    \begin{align}\label{eq:separable_sum}
        \phi(p^A) &= (p^A)^\top Q p^A - v^\top p^A + (p^{B0})^\top Q p^{B0} \cr
        &= \sum_{\ell=1}^c\left( q_{(\ell)} (p^A_{(\ell)})^2 - v_{(\ell)} p^A_{(\ell)} + q_{(\ell)}(p^{B0}_{(\ell)})^2 \right),
    \end{align}
    where $q_{(\ell)}:= \allone^T Q_{(\ell)} \allone$ and $v_{(\ell)} := -2\allone^\top R_{(\ell)}p^{B0}_{(\ell)}$ with $p_{(\ell)}^{B0}\in\R_{\ge 0}^{n_\ell}$ (respectively, $p_{(\ell)}^{A0}\in \R_{\ge 0}^{n_\ell}$) denoting the restriction of $p^{B0}$ (respectively, $p^{A0}$) to the entries indexed by $V_{(\ell)}$. Thus,  welfare takes the form of a variable-separable summation. Furthermore, we can ignore $q_{(\ell)}(p_{(\ell)}^{B0})^2$ as it does not depend on the decision variables $\{p^A_{(\ell)}\}_{\ell=1}^c$ in the absence of penalties. Consequently, our welfare maximization problem  becomes
\begin{subequations}\label{eq:constraints_component_wise}
        \begin{align}
        \begin{split}
&\widetilde\pbold\textbf{: }\\\underset{\left\{p^A_{(\ell)}:\ell\in [c]\right\} }{\text{Maximize }}  \widetilde \phi(p^A):= &\sum_{\ell=1}^c \left( q_{(\ell)} (p_{(\ell)}^A)^2 - v_{(\ell)} p_{(\ell)}^A\right) 
        \end{split} \cr
        \begin{split}
            \text { s.t. }\label{eq:modified_unsustainable_constraint} \quad \sum_{\ell=1}^c b_{\Delta(\ell)} p_{(\ell)}^A &\geq k_0
        \end{split} \\
        \begin{split}
\label{eq:modified_price_up_constraint} p_{(\ell)}^A &\geq p^{A0}_{(\ell)}
        \end{split}  \\
        \begin{split}            \label{eq:modified_price_threshold_constraint} p &\leq \tilde p_{\max },
        \end{split}
        \end{align}
    \end{subequations}
where $b_{\Delta(\ell)}:=\sum_{i\in V_{(\ell)}} b_{\Delta i}$ for all $\ell \in [c]$, and the $\ell$-th entry $\tpmax_{(\ell)}$ of the vector $\tpmax\in\R^\ell$ denotes the upper bound imposed on the prices of sustainable goods produced in component $\ell\in[c]$.

We observe that the only constraint in $\widetilde\pbold$ that introduces coupling between the $c$ connected components of $G$ is~\eqref{eq:modified_unsustainable_constraint}, the tolerance constraint on  aggregate unsustainable effort. 
As the first step, therefore, we solve $\widetilde\pbold$ after dropping~\eqref{eq:modified_unsustainable_constraint}. This reduces our modified optimization problem to $c$ single-variable quadratic maximization problems. That is, we maximize $\tilde \phi_{(\ell)}(p_{(\ell)}^A):= q_{(\ell)} (p^A_{(\ell)})^2 - v_{(\ell)} p_{(\ell)}^A$ over the interval $p_{(\ell)}^{A0}\le p_{(\ell)}^A \le \tpmax_{(\ell)}$ for each $\ell\in[c]$. The convexity of $\tilde\phi_{(\ell)}$ implies that its maximum is attained at a boundary point of $[p_{(\ell)}^{A0}, \tpmax_{(\ell)}]$. We can verify that the welfare difference corresponding to the two boundary points $p_{(\ell)}^{A0}$ and $\tpmax_{(\ell)}$ is given by
\begin{align}\label{eq:difference_quad}
    &\tilde\phi_{(\ell)}(\tpmax_{(\ell)}) - \tilde\phi_{(\ell)}(p^{A0}_{(\ell)}) \cr
    &= q(\tpmax_{(\ell)} - p^{A0}_{(\ell)})\left(\tpmax_{(\ell)} + p^{A0}_{(\ell)} - \frac{v_{(\ell)} }{q_{(\ell)}} \right) ,
\end{align}
which implies that the choice of $p^A$ that solves all the $c$ maximization problems is $p^{(0*)}\in \R^c$, whose $\ell$-th entry is defined by
\begin{align}\label{eq:p-knot-star}
    p_{(\ell)}^{(0*)} := 
            \begin{cases}
                \tpmax_{(\ell)} \quad &\text{if } \tpmax_{(\ell)}\ge\frac{v_{(\ell)}}{q_{(\ell)}}-p_{(\ell)}^{A0},\\
            p_{(\ell)}^{A0}\quad &\text{otherwise}
            \end{cases}
\end{align}
for all $\ell\in[c]$. Observe that the inequality in~\eqref{eq:p-knot-star} is equivalent to $q_{(\ell)}(\tpmax_{(\ell)} + p^{A0}_{(\ell)})\ge v_{(\ell)}$, which is analogous to~\eqref{eq:thm_condition_1} with $p^{A(1)}=\pmax$ and $p^{A(2)}=p^{A0}$.  Thus, the \textit{direct effect} of offering the price raise $\tpmax_{(\ell)}-p^{A0}_{(\ell)}$ to component $\ell$  dominates the \textit{substitutability effect} (Section~\ref{sec:optimal}) of this raise on welfare if and only if the inequality in~\eqref{eq:p-knot-star} holds. 
Note also that both $q_{(\ell)} = \sum_{i,j} (Q)_{ij}$ and $v_{(\ell)} = 2\sum_{j}(R)_{ij}p^{B0}_j$ depend on the strength of network effects, since all the entries of $Q$ and $R$  are monotonically increasing in the entries of $G$. Using the definitions of $Q$ and $R$, this can be verified to imply that the denser the network, the lower is the value of the threshold $v_{(\ell)}/q_{(\ell)} - p^{A0}_{(\ell)}$, and the easier it is to improve welfare by increasing the incentive for sustainable production.

As an aside, note that if we were to disregard network effects by replacing $G$ with the matrix of all zeros, then  $p^{(0*)}$ would equal the policy defined in~\eqref{eq:pre-modeling}, which optimally solves our preliminary intervention design problem $\pbold_0$. This can be verified by substituting $G=\boldsymbol {O}$ in the expressions defining $Q$ and $R$, which determine $q_{(\ell)}$ and $v_{(\ell)}$, respectively.

We now introduce the 
constraint~\eqref{eq:modified_unsustainable_constraint}. Observe that, if $p_{(\ell)}^{(0*)}$ satisfies~\eqref{eq:modified_unsustainable_constraint}, then it is a feasible solution of $\widetilde\pbold$. As it maximizes $\tilde\phi$ over the superset $\{\tilde p^A\in\R^c: \tilde p^{A0} \le \tilde p^A\le \tpmax\}$ of the feasible set of $\widetilde\pbold$, it follows that $p^{(0*)}$ is an optimal solution of $\widetilde\pbold$. So, unless stated otherwise, we will henceforth assume that $p_{(\ell)}^{(0*)}$ violates~\eqref{eq:modified_unsustainable_constraint}, the tolerance constraint on aggregate unsustainable effort.

For this case, we adopt the technical approach of \cite{li2023efficient} 
and extend it to handle the   constraint~\eqref{eq:modified_unsustainable_constraint}. We outline this approach below and relegate its implementation to  the appendix.
\begin{enumerate}
    \item We first find the convex relaxation of the objective function $\widetilde \phi(\cdot)$ of $\widetilde\pbold$ using the concept of the bi-conjugate of a function, which is the function obtained on conjugating the given function twice \cite[\S 3.3.2]{boyd2004convex}. This gives the following convex relaxation of $\widetilde\pbold$:
    \begin{align}\label{eq:primal_in_text}
    \text{Maximize}\quad - \sum_{\ell=1}^c g_\ell^{**} (\tilde p_{\ell} )  \cr 
   \text{s.t.}\quad \sum_{\ell=1}^c \tilde p_{\ell} &\ge \tilde k_0,
\end{align}
where for each $\ell\in [c]$, we define $\tilde p_{\ell}:= b_{\Delta(\ell)}\left(p^A_{(\ell)} - \frac{v_{(\ell)} }{2q_{(\ell)}} \right)$ and 
$$
    g_\ell^{**}(\tilde p_\ell) := 
    \begin{cases}
        d_\ell((\tilde u_\ell + \tilde p_{0\ell})\tilde p_\ell - \tilde u_\ell \tilde p_{0\ell}) \quad&\text{if }\tilde p_{0\ell} \le \tilde p_\ell \le \tilde u_\ell,\\
        \infty\quad&\text{otherwise},
    \end{cases}
$$
where $d_\ell:=-\frac{q_{(\ell)}}{b_{\Delta(\ell)}^2 }$, $\tilde p_{0\ell}:= b_{\Delta(\ell)}\left(p^{A0}_{(\ell)} - \frac{v_{(\ell)} }{2q_{(\ell)}} \right)$ and $\tilde u_{\ell}:= b_{\Delta(\ell)}\left(\tpmax_{(\ell)} - \frac{v_{(\ell)} }{2q_{(\ell)}} \right)$.  To account for the constants (terms independent of the decision vector $p^A$)  that appear in the welfare expression~\eqref{eq:welfare_expression}, we include an offset in our definition of the relaxed objective function:
$$
    \tilde \phi_{\text{rel}}:=-\sum_{\ell=1}^c{g^{**}_{\ell}}+ \sum_{\ell=1}^c \left( q_{(\ell)} \left(p^{B0}_{(\ell)}\right)^2 - \frac{v_{(\ell)}^2 }{4q_{(\ell)}^2} \right ).
$$

    \item We then formulate the Lagrangian dual of the relaxed problem~\eqref{eq:primal_in_text} as $\text{Minimize}_{\lambda\ge 0} \left\{\sum_{\ell=1}^c g_\ell^*(\lambda)- \lambda\tilde k_0  \right\}$,  where $g_\ell^*$ is the conjugate function of $g_\ell^{**}$ and $\tilde k_0 := k_0 - \sum_{\ell=1}^c\left(\frac{b_{\Delta(\ell)} v_{(\ell)} }{2 q_{(\ell)}}\right)$. We analytically determine the optimal solution $\lambda^*$ of this dual problem.
    \item Finally, we show that the subgradients of  $\{g_\ell^*:\ell\in[c]\}$ at $\lambda^*$ yield the optimal solution of the relaxed problem~\eqref{eq:primal_in_text}, which equals the optimal solution of $\widetilde\pbold$ under the conditions established in Theorem~\ref{thm:component_wise_pricing}.
\end{enumerate}

 Following the above approach yields the next main result. This result and the subsequent discussion uses some notation that we define below. We define $S_L:=\left\{\ell\in [c]: p_{(\ell)}^{A0} + \tpmax_{(\ell)} \le \frac{v_{(\ell)}}{q_{(\ell)}}\right \}$,  and for each $\ell\in [c]$, we define
\begin{align}\label{eq:only_c_policies}  p^{(\ell)}_{(\ell)} := 
    \begin{cases}
     \tpmax_{(\ell)}  \quad &\text{if }1\le m\le \ell \\
     p^{A0}_{(\ell)} &\text{if }\ell+1 \le m\le c.
\end{cases}
\end{align}
In addition, we define $\ell^*:=\min\left\{\ell\in S_L :\sum_{m=1}^\ell b_{\Delta (m)} p^{(\ell)}_{(m)} \ge  k_0  \right\}$, $\ell' := \min S_L$, and let $\bar p\in\R^{c}$ denote the policy whose $\ell$-th entry is given by
\begin{align}\label{eq:relaxed_optimal}  &\bar p_{(\ell)}\cr
&:= 
    \begin{cases}
     p^{(\ell^*)}_{(\ell)}  \quad &\text{if }\ell\ne \ell^*\\
     \frac{\left(k_0 - \sum_{\ell\ne \ell^*} b_{\Delta(\ell)}p^{(\ell)}_{(\ell)} \right)}{b_{\Delta(\ell^*)}}  &\text{if }\ell=\ell^*\ne \ell'-1\\
     p^{A0}_{(\ell)} &\text{if }\ell = \ell^* = \ell'-1.
\end{cases}
\end{align}

\begin{theorem}\label{thm:component_wise_pricing} 
    Suppose $\tpmax\le \plim$ and $b_{\Delta}^\top \tpmax\ge k_0$, and let Assumptions~\ref{assum:uniqueness},~\ref{assum:domination} and~\ref{assum:positive_pre-intervention} hold. Let the connected components of the strategic interaction network be ordered in the descending order of
    $
        \left\{\gamma_\ell:\ell\in [c] \right\}
    $, where $\gamma_\ell:=\frac{q_{(\ell)} }{ b_{\Delta(\ell) }} \left( p_{(\ell)}^{A0} + \tpmax_{(\ell)} - \frac{v_{(\ell)}}{q_{(\ell)}} \right)$, and suppose $ p_{(\ell)}^{A0} + \tpmax_{(\ell)} \ne \frac{v_{(\ell)}}{q_{(\ell)}} $ for all $\ell\in[c]$. 
   Then the policy $\bar p$  (see~\eqref{eq:relaxed_optimal}) optimally solves either $\widetilde\pbold$ or its convex relaxation~\eqref{eq:primal_in_text}, and the optimal value of the relaxed objective function $\tilde \phi_{\emph{rel}}$ provides an upper bound on the optimal welfare at equilibrium, i.e., $\phi( p^*)\le \tilde \phi_{\emph{rel}}(\bar p)$. Moreover, if $\bar p_{(\ell^*)}\in \left\{p_{(\ell^*)}^{A0},\tpmax_{(\ell^*)}\right \}$, then  this upper bound is tight, we have $\bar p = p^{(\ell^*)}$, and $p^{(\ell^*)}$ is an optimal solution of $\widetilde\pbold$.
\end{theorem}

\begin{remark}
    If there exists a subset of indices $S\subset [c]$ such that $\tpmax_{(\ell)}\ge \plim_{(\ell)}$ for $\ell\in S$, the conclusions of Theorem~\ref{thm:component_wise_pricing} continue to hold provided  $p^{A}\ge p^{B0}$.  
\end{remark}

Assuming $\widetilde\pbold$ to be feasible (equivalently, assuming that $\tpmax$ satisfies the tolerance constraint~\eqref{eq:modified_unsustainable_constraint}),  Theorem~\ref{thm:component_wise_pricing} characterizes the optimal solution of $\widetilde\pbold$ for all possible values of $p^{A0}$ and $\tpmax$ except for those violating $p_{(\ell)}^{A0} + \tpmax_{(\ell)} \ne \frac{v_{(\ell)}}{q_{(\ell)}}$ for some $\ell\in [c]$, which lie in a set of Lebesgue measure zero.  

We now derive a corollary of Theorem \ref{thm:component_wise_pricing} that facilitates its interpretation.

\begin{corollary}\label{cor:second}
    The following statements are true for the optimal solution of $\widetilde\pbold$.
    \begin{enumerate}
        \item [(i)] If the policy  $p^{(0*)}$, defined in~\eqref{eq:p-knot-star}, satisfies~\eqref{eq:modified_unsustainable_constraint}, then $p^{(0*)}$ is an optimal solution of $\widetilde \pbold$.
        \item [(ii)]  For  all $\ell\in[c]$ such that  $\tpmax_{(\ell)}\ge\frac{v_{(\ell)}}{q_{(\ell)}}-p_{(\ell)}^{A0}$, there exists an optimal solution $p^*$ of $\widetilde\pbold$ with  $p^*_{(\ell)}= p^{(0*)}_{(\ell)}= \tpmax_{(\ell)}$.
        \item [(iii)] We have $p_{(\ell)}^*\in\left\{p^{A0}_{(\ell)},\tpmax_{(\ell)} \right\}$ for at least $c-1$ values of the index $\ell \in [c]$.
    \end{enumerate}
\end{corollary}

According to Corollary \ref{cor:second}-(i), Theorem~\ref{thm:component_wise_pricing} is consistent with our earlier observation that the policy $p^{(0*)}$ is optimal whenever it satisfies the tolerance constraint on aggregate unsustainable effort. 

We learn from Corollary~\ref{cor:second}-(ii) that for all those connected components whose contribution to welfare is maximized at the maximum feasible prices (given by $\tpmax$) due to the \textit{direct effect} dominating the \textit{substitutability effect} (Section~\ref{sec:optimal}) as a result of~\eqref{eq:difference_quad}, the optimal policy assigns the maximum feasible prices to these components. So, it only remains to examine the prices assigned to the components in which the substitutability effect dominates the direct effect, i.e., those for which $p^{A0}_{(\ell)} + \tpmax_{(\ell)} < \frac{v_{(\ell)}}{q_{(\ell)}}$. These are indexed by $S_L$.

Theorem~\ref{thm:component_wise_pricing} assumes that these components are ordered in the ascending order of $\{|\gamma_\ell|:\ell\in S_L\}$ -- quantities that depend on $\{b_{\Delta(\ell)}:\ell\in S_L\}$ and $\left\{q_{(\ell)}\left( p_{(\ell)}^{A0} + \tpmax_{(\ell)} - \frac{v_{(\ell)}}{q_{(\ell)}}\right):\ell\in S_L\right\}$. Recall from~\eqref{eq:difference_quad} that the smaller the value of $q_{(\ell)}\left| p_{(\ell)}^{A0} + \tpmax_{(\ell)} - \frac{v_{(\ell)}}{q_{(\ell)}}\right|$ for $\ell\in S_L$, the smaller the domination of the substitutability effect over the direct effect, and lower is the resulting welfare loss in transitioning from $p^{A0}_\ell$ to $\tpmax_\ell$. On the other hand, we know from the discussion preceding Definition~\ref{def:centrality} that the greater the value of the centrality $b_{\Delta(\ell)}$, the greater the reduction in aggregate unsustainable effort for every unit increase in $p^A_\ell$. Therefore, the quantities $\{|\gamma_\ell|:\ell\in S_L\}$ capture the trade-off between minimizing welfare loss and reducing  unsustainable effort in price transitions from $p^{A0}_{(\ell)}$ to $\tpmax_{(\ell)}$, and components with smaller values of $\gamma_\ell$ achieve a superior trade-off by causing a more significant reduction in unsustainable effort at the expense of a modest loss in welfare.   
Therefore, if $p^{(0*)}$ is infeasible, then  we can search for the optimal policy among other extreme points of the polytope defined by~\eqref{eq:modified_unsustainable_constraint} - \eqref{eq:modified_price_threshold_constraint} by sequentially introducing the aforementioned price transitions starting from the component with the smallest value of $|\gamma_\ell|$ and checking whether the tolerance constraint~\eqref{eq:modified_unsustainable_constraint} is satisfied each time we introduce an additional transition. This is equivalent to sweeping over the policy sequence $p^{(1)},p^{(2)},\ldots$  and stopping at the first policy that satisfies~\eqref{eq:modified_unsustainable_constraint}, which is $p^{(\ell^*)}$ by the definition of $\ell^*$. This explains why $p^{(\ell^*)}$ is either optimal or gives us all entries but one of the optimal solution $\bar p$ of the convex relaxation~\eqref{eq:primal_in_text} of $\widetilde\pbold$. Moreover, if $\bar p_{\ell^*}\in \{p^{A0}_{(\ell^*)},\tpmax_{(\ell^*)}\}$, then Theorem~\ref{thm:component_wise_pricing} states that the objective functions of~\eqref{eq:primal_in_text} and $\widetilde\pbold$ coincide at $p^A =  p^{(\ell^*)}$, and $p^{(\ell^*)}$ optimally solves $\widetilde\pbold$.    

Finally, 
when Theorem~\ref{thm:component_wise_pricing} cannot determine the optimal policy, which happens if and only if $\bar p_{(\ell^*)}\in \left(p^{A0}_{(\ell^*)},\tpmax_{(\ell^*)} \right)$), it provides an upper bound on the optimal welfare via the optimal value of the relaxed objective function $\tilde \phi_{\text{rel}}$. 

\begin{remark}
    Note from Theorem~\ref{thm:component_wise_pricing} that the condition for the coincidence of the optimal solutions of $\widetilde\pbold$ and its convex relaxation~\eqref{eq:primal_in_text} is $\bar p_{(\ell^*)}=\tpmax_{(\ell^*)}$, which, in light of~\eqref{eq:relaxed_optimal} and the definition of $k_0$, is equivalent to the tolerance  on aggregate unsustainable effort $\tau^B$  satisfying a linear equation defined by the parameters of $\widetilde\pbold$. We can verify that this equation can be satisfied by reducing the value of $\tau^B$ to $\frac{1}{2}\allone^\top\left(M^+ + M^-\right)p^B - b_\Delta^\top p^{(\ell^*)}$. While promising a greater reduction in aggregate unsustainable effort, such a reduction in $\tau^B$ lets us compute the welfare-maximizing policy efficiently. However, this gain in computational efficiency may be accompanied by a reduction in optimal welfare. 
\end{remark}

We conclude this section by linking the redistribution variant $\pbr$ to $\widetilde\pbold$. Observe from Theorem~\ref{thm:redistribution} that the policy $p^{A*}_{\text{R}}$ continues to be the optimal redistribution policy under component-wise uniform pricing constraints provided the pre-intervention prices $\{p^{A0}_i:i\in [n]\}$, the penalties $\{\rhomax_i:i\in [n]\}$, and the budgets $\{b_i:i\in [n]\}$ are also selected to be uniform within every connected component of the given network. However, component-wise uniform pricing may make it impossible for $p^A$ to satisfy \eqref{eq:vanish} with equality, thereby requiring the use of higher penalties or budgets than in the absence of uniform pricing constraints.

\def\probfinal{\widetilde{\boldsymbol P}_\text{sum}}

\def\nbold{\boldsymbol n}

\section{Empirical Analysis}\label{sec:empirical}

The goal of this section is to use real price data and real concession network data to demonstrate the policy implications of our theoretical results, identify  situations of practical importance where one policy prescription is superior to another, and obtain insights for guiding policy recommendations in practice. For this purpose, we consider the case of palm oil cultivation in Indonesia, which we describe in detail below. 

\subsection{Problem Parameters and Inputs}\label{sec:parameters_inputs}

\subsubsection{Pricing:}

In our study, the goods resulting from activities $A$ and $B$ are Certified Sustainable Palm Oil (CSPO) and unsustainable Crude Palm Oil (CPO), respectively. As reported in~\cite{wwf2022business}, the conventional price of CPO was $1056$ US dollars per metric ton (MT) in 2022, i.e., 
$p_i^B= p_0^B:= \$1056$/MT for all $i\in [n]$, where \$ denotes the US dollar. 

 Next, we factor in sustainability premiums and certification costs. 
 The values of the premiums depend on the types of  sustainability certification awarded by  RSPO to the cultivators, and these certification types in turn depend on RSPO's supply chain models that best describe the supply chains to which the cultivators belong (see Appendix \ref{sec:monetary_component}). 
 We do not make any assumptions on the certification type and instead explore the full range of sustainability premiums across all the  models. We know from \cite{wwf2022business} that this range of values can be expressed as a percentage of the unsustainable CPO prices $\{p_i^B:i\in [n]\}$ as $[0.2, 5.2]$. In other words, sustainability premiums take values in the range $[0.002p_i^B, 0.052 p_i^B]$ for each $i\in[n]$. 

 On the other hand, the costs of sustainability certification range from \$1/MT of \textit{fresh fruit bunches} (FFBs) to \$7/MT of FFB. As $5$ FFBs are needed to produce 1 MT of CSPO (\cite{wwf2022business}), it follows that the certification costs lie in the interval $[5, 35]$ in US dollars per MT. Ignoring the regulatory and operational costs of unsustainable production, this implies that the prices of CSPO, namely $\{p_i^A:i\in[n]\}$, lie in the range $[(1 + 0.002 - 35/1056)p_i^B, (1 + 0.052 -  5/1056)p_i^B] = [0.97 p_i^B, 1.05 p_i^B]$ for all $i\in [n]$. For a fair comparison of $\pbold$, $\widetilde\pbold$, and $\pbr$,  we ensure that the network average of the maximum feasible raises in the per-unit prices of sustainable goods is the same for all the problems, i.e., we enforce $\frac{1}{n} \sum_{i=1}^n (\pmax_i - p_i^{A0}) = \frac{1}{n}\sum_{\ell=1}^c n_\ell(\tpmax_{(\ell)} - p^{A0}_{(\ell)}) = \frac{1}{n}\sum_{i=1}^n (\rhomax_i + b_i) = \bar p^{\max} - \frac{1}{n}\sum_{i=1}^n p^{A0}_i$, where $n_\ell$ denotes the number of agents in component $\ell$, and $\bar p^{\max}$, which denotes the network average of all the maximum feasible prices, is a scalar parameter that we vary over a range of values. In addition, we set $b=0$ to model the worst-case redistribution setting wherein the planner has no price adjustment budget for the intervention prescribed by $\pbr$ (unlike in the case of $\pbold$ and $\widetilde\pbold$). 

\subsubsection{Network Topology:}
We construct our strategic interaction network using maps provided by \cite{gfw2015indonesia} on oil palm concessions in Indonesia. 
To extract the geospatial network of concessions, we identify all pairs of neighboring concessions from these maps. 
Due to the presence of shadow companies and the lack of transparency about the ownership of the concessions (\cite{mongabay2023investigation,ran2024exposing}), we consider the on-site concession {managers} rather than the {owners} as the agents (cultivators) in our game-theoretic setup. This results in the strategic interaction network being identical to the extracted geospatial concession network. 

The above steps yield a network consisting of 1846 concessions, which make up 255 connected components in total. 
For our analysis, we selected a sub-network consisting of 10 of these connected components with a total of 191 concessions that collectively span 3,070,931 hectares, which is approximately $20\%$ of the total area of land used for palm oil cultivation in Indonesia.

\subsubsection{Game-Theoretic Model Parameters:} Recall from Theorem~\ref{thm:one} that in order for $\pbold$ and $\widetilde\pbold$ to yield policies that do not increase aggregate unsustainable effort, it is necessary to restrict the cross-activity network effect to the extent quantified by Assumption~\ref{assum:domination}. In all our experiments, the values of $\beta$, $\delta$, and $\mu$  satisfy  Assumptions~\ref{assum:uniqueness} and~\ref{assum:domination}.

In the plots that follow, we choose  parameter values that enable us to extract as many  insights as possible into the  performance of the optimal policies obtained on solving $\pbold$, $\widetilde\pbold$ and $\pbr$.

\subsection{Performance of the Optimal Policies under Real Pre-intervention Prices}\label{subsec:real_prices}
For the problem parameters and inputs selected in Subsection~\ref{sec:parameters_inputs}, we now compare the optimal policies obtained on solving $\pbold$, $\widetilde\pbold$, and $\pbr$ with each other by computing the following quantities: improvement in welfare and reduction in aggregate unsustainable effort as percentages of their pre-intervention values, where the latter are computed for the worst possible pre-intervention policy (which offers the lowest reported premium after levying the highest reported cost of sustainability certification, resulting in $p^{A0}=0.97 p^{B0}$ as computed above) as well as for the best possible policy (which offers the highest reported premium after levying the lowest reported cost of sustainability certification, resulting in $p^{A0}=1.05 p^{B0}$ as computed above).

Below, we plot these quantities against the network average of all the percent increases in the per-unit prices of sustainable goods  for $\beta = 0.2$, $\delta = 0.1$, and $\mu = 0.01$. Additionally, we set $\tau^B = 214.39$, which is the value of aggregate unsustainable effort when $p_i^A = 0.98 p^B_0$ and $p_i^B = p^B_0$ for all $i\in[n]$.  To leverage the flexibility offered by the absence of component-wise uniform pricing constraints in $\pbold$, 
we choose the price bounds $\{\pmax_i\}_{i=1}^n$ to be heterogeneous across  agents. 

\begin{figure}[h]
    \centering
    \begin{minipage}{0.45\textwidth}
        \centering
        \includegraphics[width=1.1\textwidth]{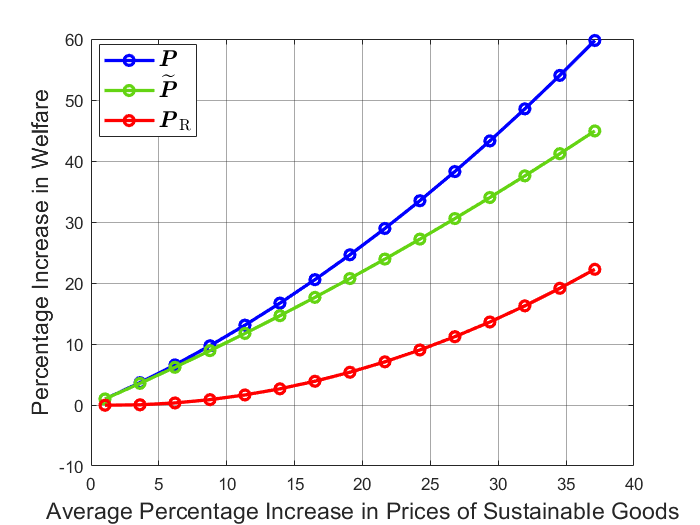} 
        \caption{Percent increase in welfare relative to the worst-case pre-intervention policy ($p^{A0}= 0.97p^{B0}$)}
        \label{fig:welfare_1}
    \end{minipage}\hfill
    \begin{minipage}{0.45\textwidth}
        \centering
        \includegraphics[width=1.1\textwidth]{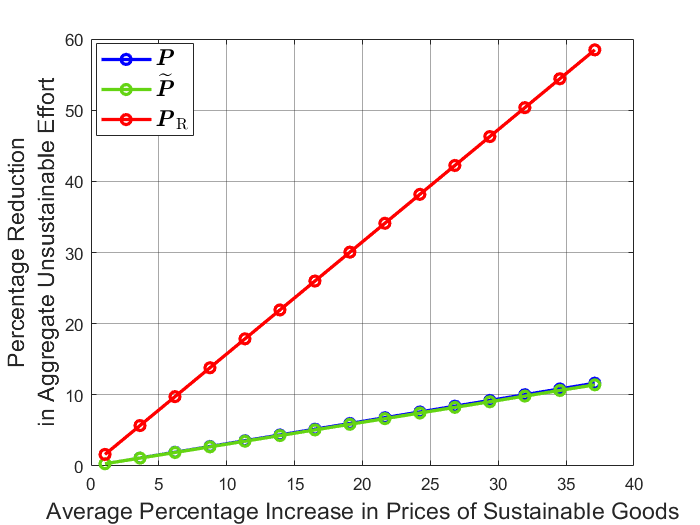} 
        \caption{Percent reduction in aggregate unsustainable effort relative to the worst-case pre-intervention policy ($p^{A0}= 0.97p^{B0}$)}
        \label{fig:agg_unsus_1}
    \end{minipage}
\end{figure}

\begin{figure}[h]    
    \centering
    \begin{minipage}{0.45\textwidth}
        \centering
        \includegraphics[width=1.1\textwidth]{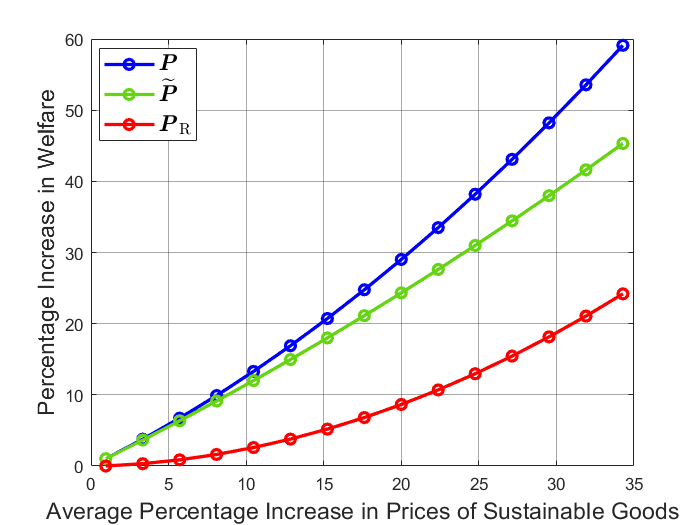} 
        \caption{Percent increase in welfare relative to the best-case pre-intervention policy ($p^{A0}= 1.05p^{B0}$)}
        \label{fig:welfare_2}
    \end{minipage}\hfill
    \begin{minipage}{0.45\textwidth}
        \centering
        \includegraphics[width=1.1\textwidth]{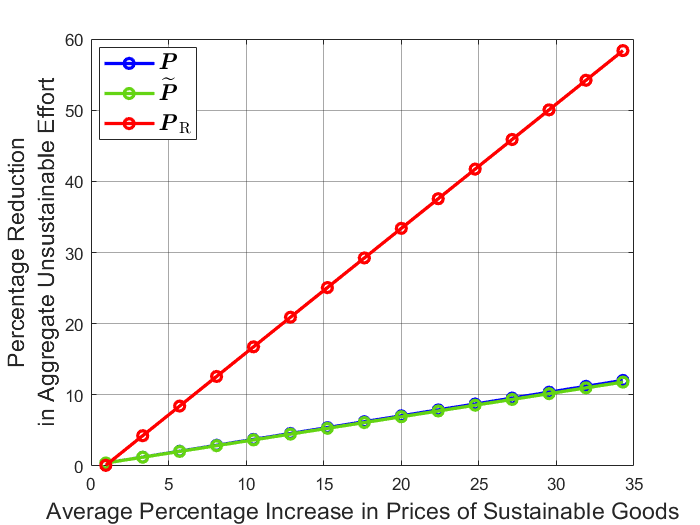} 
        \caption{Percent reduction in aggregate unsustainable effort relative to the best-case pre-intervention policy ($p^{A0}= 1.05p^{B0}$)}
        \label{fig:agg_unsus_2}
    \end{minipage}
\end{figure}

\begin{figure}[h] 
    \centering
    \begin{minipage}{0.45\textwidth}
        \centering
        \includegraphics[width=1.1\textwidth]{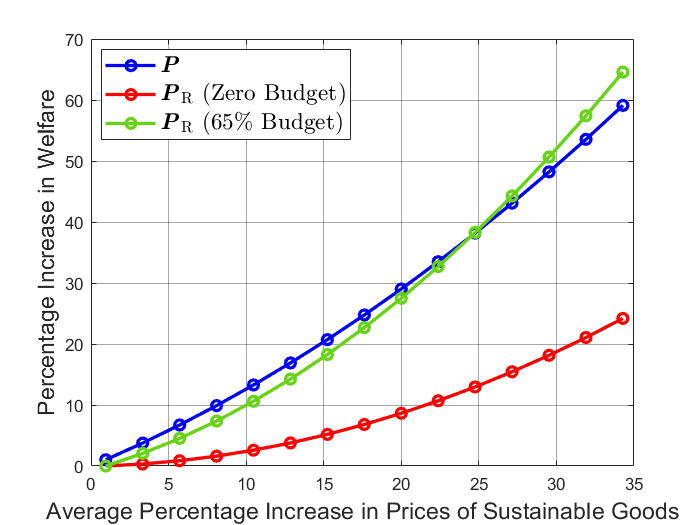} 
        \caption{Percent increase in welfare relative to the best-case pre-intervention policy ($p^{A0}= 1.05p^{B0}$)}
        \label{fig:welfare_3}
    \end{minipage}\hfill
    \begin{minipage}{0.45\textwidth}
        \centering
        \includegraphics[width=1.1\textwidth]{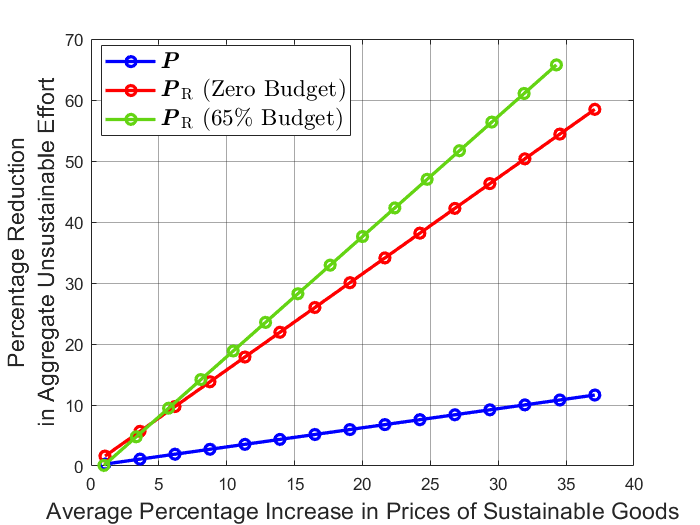} 
        \caption{Percent reduction in aggregate unsustainable effort relative to the best-case policy ($p^{A0}= 1.05p^{B0}$)}
        \label{fig:agg_unsus_3}
    \end{minipage}
\end{figure}

\textbf{\textit{Inferences:}} We observe from Figures \ref{fig:welfare_1} and \ref{fig:welfare_2} that the optimal welfare is monotonically increasing with $\bar p^{\max}$ for $\pbold$. Since we have $Q p^{A0} \ge  R p^B$, this observation is consistent with Corollary~\ref{cor:first}-(i) of Theorem~\ref{thm:general_case}, which implies that the optimal solution of $\pbold$ is $\pmax$ in this case. Similarly, the optimal welfare is monotonically increasing with $\bar p^{\max}$ for $\widetilde\pbold$, and we can use~\eqref{eq:p-knot-star} to verify that this observation is consistent with Corollary~\ref{cor:second}  of Theorem~\ref{thm:component_wise_pricing}. Next, we observe that the optimal policy given by $\pbold$ achieves a considerably higher welfare than that of $\widetilde\pbold$ because, unlike $\widetilde\pbold$, the problem $\pbold$ offers the flexibility of choosing different price bounds for every agent. By contrast, both the policies achieve approximately the same reduction in aggregate unsustainable effort (Figures \ref{fig:agg_unsus_1} and \ref{fig:agg_unsus_2}), which  depends linearly rather than quadratically on the per-unit prices and is therefore less sensitive to price changes. As a result, the plots of aggregate unsustainable effort  corresponding to $\pbold$ and $\widetilde\pbold$ overlap in multiple figures (Figures \ref{fig:agg_unsus_1}, \ref{fig:agg_unsus_2}, and \ref{fig:agg_unsus_4}). Also, observe that the optimal policy resulting from $\pbr$ achieves a remarkably superior reduction in aggregate unsustainable effort compared to those resulting from $\pbold$ and $\widetilde\pbold$ because it not only raises the prices of sustainable goods  but also reduces those of unsustainable goods. However, it does not improve welfare as well as the other two policies, which benefit from a positive budget that helps increase welfare by improving the total incentive $p^A+p^{B0}$ for  engaging in either activity.  

Interestingly, we infer from Figures \ref{fig:welfare_3} and \ref{fig:agg_unsus_3} that 
solving the redistribution problem $\pbr$ with the same penalties as above can achieve welfare improvements that are comparable to those resulting from the most welfare-generating policy (the policy given by $\pbold$) with a significantly lower budget (only $65$\% as much budget as that used by the policy given by $\pbold$).  

Additionally, we consider the case $p^{A0} = 0.7 p^{B0}$ in Appendix \ref{sec:utility} to achieve two objectives. First, we verify  the assertion of Theorem~\ref{thm:redistribution} that no agent's individual utility decreases as a result of redistribution even when the price adjustment budget is zero, provided the penalties are large enough to satisfy the budget-penalty sufficiency condition~\eqref{eq:half-budget}. Second, we validate Proposition~\ref{prop:pessimistic} by demonstrating that the violation of~\eqref{eq:half-budget} causes  redistribution to reduce the utilities of a subset of the agents.

Finally, we conduct comparative statics in Section~\ref{sec:comp_stat}.

\subsection{Comparison of the Optimal Policies under a High Incentive Bias}\label{subsec:high_incentive_bias}

In Subsection~\ref{subsec:real_prices}, we compared the optimal policies given by $\pbold$, $\widetilde\pbold$ and $\pbr$ by setting the values of the pre-intervention prices of sustainable goods to either the lowest or the highest of the CSPO prices reported for Southeast Asia. However, we found that the associated incentive bias magnitudes, quantified by the entries of $|p^{A0}-p^{B0}|$, were small fractions of $p^{B}_0$ (the fraction lay in the interval $[0.03,0.05]$). This feature is specific to our case study and may not apply to the production of commodities other than palm oil. Therefore, we now assess the performance of these optimal policies for cases in which the incentive bias magnitudes are high before the intervention is applied, i.e., when the entries of $|p^{A0}-p^{B0}|$ are significant fractions of $p^{B}_0$.

We now consider one such case: we set $p^{A0}=0.3 p^{B0} =0.3 p^B_0\allone$, which ensures that $|p^{A0}-p^{B0}|=0.7p^B_0$. We then vary $\pmax$ over the interval $[0.3p^B_0, 1.0p^B_0]$ so that the incentive bias vector is $\allzero$ when the prices of sustainable goods are set to their maximum values. For this case, Figures \ref{fig:welfare_4} and \ref{fig:agg_unsus_4} below plot the welfare and the aggregate unsustainable effort for $\delta=0.1$, $\mu = 0.01$, and $\beta=0.2$.

\begin{figure}[h]
    \centering
    \begin{minipage}{0.45\textwidth}
        \centering
        \includegraphics[width=1.1\textwidth]{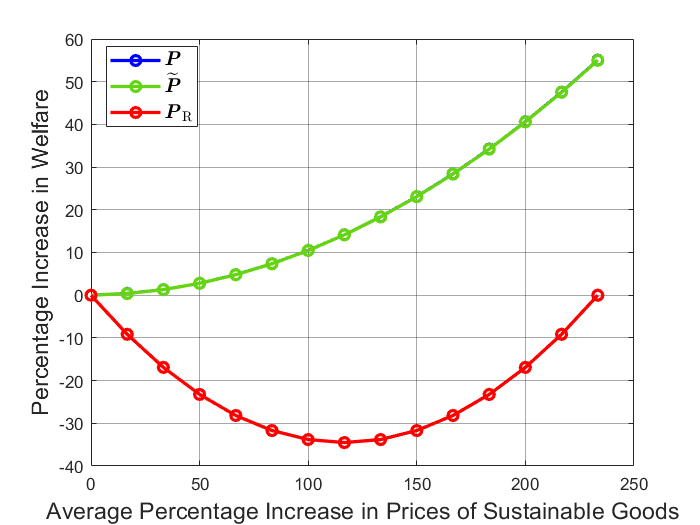} 
        \caption{Percent increase in welfare relative to the  pre-intervention policy $p^{A0}= 0.3p^{B0}$}
        \label{fig:welfare_4}
    \end{minipage}\hfill
    \begin{minipage}{0.45\textwidth}
        \centering
        \includegraphics[width=1.1\textwidth]{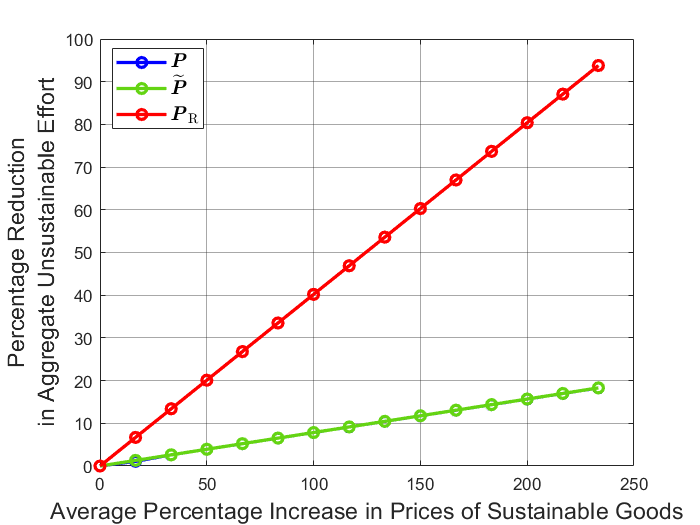} 
        \caption{Percent reduction in aggregate unsustainable effort relative to the  pre-intervention policy $p^{A0}= 0.3p^{B0}$}
        \label{fig:agg_unsus_4}
    \end{minipage}
    
\end{figure}

\textbf{\textit{Inferences:}} 
Showing consistency with our observations in Subsection~\ref{subsec:real_prices}, the optimal redistribution policy remains markedly superior to the zero-penalty policies (those given by $\pbold$ and $\widetilde\pbold$) for reducing aggregate unsustainable effort (Figure~\ref{fig:agg_unsus_4}). However, this policy leads to a reduction in welfare (Figure~\ref{fig:welfare_4}) if the penalties are so small that they violate all the scalar inequalities implied by~\eqref{eq:half-budget}, thereby validating Proposition~\ref{prop:pessimistic}. This is because, although the total incentive for production (as quantified by $p^A + p^B=p^{A0}+p^{B0}$) is preserved  by the  redistribution policy, the magnitudes of incentive biases, which contribute to welfare by keeping the substitutability effect in check, decrease in the process of imposing low penalties for unsustainable effort. As expected, the welfare is at its minimum when $p^{A}=p^B = 0.65p^{B0}$, i.e., when the incentive biases are all zero, which corresponds to a $(100 + (100-150)/3)\% \approx 117\%$ increase in the prices of sustainable goods (Figure~\ref{fig:welfare_4}). Nevertheless, if the penalties are high enough to swap the prices of sustainable goods with those of unsustainable goods so that the absolute values of the incentive biases remain unchanged during the intervention (i.e., $p^A - p^{B} = 
p^{B0} - p^{A0}$), welfare is preserved and aggregate unsustainable effort is suppressed remarkably.

\textbf{\textit{Policy Recommendations:}}  We  now use the above inferences in conjunction with our theoretical results to frame the following policy guidelines for a planner with a fixed price-adjustment budget $b$ and the ability to penalize unsustainable effort with maximum penalties given by $\rhomax$. 
\begin{enumerate}
    \item If the planner prioritizes  unsustainable effort reduction over welfare improvement, then we recommend using the entirety of the budget to increase the prices of sustainable goods and impose maximum penalties for unsustainable effort, effectively setting $(p^A,p^B)=(p^{A*}_{\text R},p^{B*}_{\text R}):= (p^{A0}+b+\rhomax, p^{B0}-\rhomax)$. This policy achieves minimal aggregate unsustainable effort (Theorem~\ref{thm:redistribution}).
    \item If the planner prioritizes welfare improvement over reduction of unsustainable effort, the following cases  warrant consideration.
    \begin{enumerate}
        \item If $b$ and $\rhomax$ are jointly sufficient (i.e.,~\eqref{eq:half-budget} holds), then they are large enough to preserve (or increase) every agent's incentive bias magnitude, in which case it is again advisable to impose maximum penalties and use the budget entirely so that $(p^A,p^B)=(p^{A*}_{\text R},p^{B*}_{\text R})$. This policy guarantees welfare improvement without reducing any agent's individual utility (Theorem~\ref{thm:redistribution}).
        \item If $b$ and $\rhomax$ are not jointly sufficient (i.e.,~\eqref{eq:half-budget} is violated), then it is advisable for the planner to solve $\widetilde\pbold$ (respectively, $\pbold$) if regionally uniform pricing is  required (respectively, not required) to maximize welfare. In this case, Theorems~\ref{thm:general_case} and~\ref{thm:component_wise_pricing} provide efficient means to compute the optimal solution for most practical situations (Sections~\ref{sec:optimal} and~\ref{sec:component-wise}). This step results in replacing $p^{A0}$ with a policy $p^*\ge p^{A0}$ that optimally solves either $\pbold$ or $\widetilde\pbold$.  As the next step, it is  advisable to check whether $b$ and $\rhomax$ are jointly sufficient after replacing $p^{A0}$ with $p^*$ on the right-hand-side of~\eqref{eq:half-budget}. If this modified inequality is satisfied, then imposing maximum penalties and using the entire budget (so that $(p^A,p^B)=(p^{A*}_{\text R},p^{B*}_{\text R})$) not only reduces aggregate unsustainable effort but also guarantees welfare improvement (Theorem~\ref{thm:redistribution}). If the condition is violated, however, the planner may impose no penalties and stick with $p^*$, which nevertheless guarantees a non-negative welfare gain while satisfying the tolerance constraint~\eqref{eq:unsustainable_effort} on aggregate unsustainable effort.
    \end{enumerate}
\end{enumerate}

We summarize these guidelines in the form of Table~\ref{tab:one} below.

\begin{table}[h!]
\caption{Summary of policy recommendations}
\centering
\begin{tabular}{|l|l|l|} 
\hline
\textbf{Primary objective } 
& \textbf{If the budget-penalty} & \textbf{If the budget-penalty} \\
& \textbf{sufficiency condition} & \textbf{sufficiency condition}\\
& \textbf{\eqref{eq:half-budget} is satisfied} & \textbf{\eqref{eq:half-budget} is not satisfied}\\
 \hline
Welfare improvement          &                    Set $(p^A,p^B)=(p^{A*}_{\text R},p^{B*}_{\text R})$             & Solve $\widetilde\pbold$ or $\pbold$ to obtain $p^*$\\
&  (Theorem~\ref{thm:redistribution}) & (Theorems~\ref{thm:general_case} and~\ref{thm:component_wise_pricing})\\
&  & Replace $p^{A0}$ with $p^*$ and reuse table \\ \hline
Unsustainable effort reduction & Set $(p^A,p^B)=(p^{A*}_{\text R},p^{B*}_{\text R})$                             & Set $(p^A,p^B)=(p^{A*}_{\text R},p^{B*}_{\text R})$\\ 
& (Theorem~\ref{thm:redistribution}) & (Theorem~\ref{thm:redistribution})\\ \hline
\end{tabular}
\label{tab:one}
\end{table}
\section{Conclusion}\label{sec:conclude}

Our work is concerned with a topic of growing importance in sustainable forestry: promoting environmentally sustainable production of forest commodities without compromising on the economic development of the cultivators (producers) of these commodities. Specifically, we addressed the problem of designing price-shaping policy interventions that maximize welfare (which models the overall economic well-being) in a  network of cultivators subject to tolerance constraints on the aggregate unsustainable effort while accounting for the impact of strategic interactions between the cultivators on their effort levels in sustainable and unsustainable production.  

Our results (i) show that increasing the prices of sustainable goods achieves a reduction in the aggregate level of unsustainable effort in the network if and only if cross-activity interactions between sustainable and unsustainable concessions are  restricted to ensure that they are dominated by intra-activity interactions, (ii) provide closed-form solutions to the problem of welfare maximization under tolerance constraints on aggregate unsustainable effort for  cases of high practical importance, (iii) offer a method that tackles the realistic problem of welfare maximization under component-wise uniform pricing constraints, and most importantly, (iv) show that in a zero-budget setting, appropriately redistributing monetary incentives across the activities maximizes welfare while minimizing aggregate unsustainable effort in a manner that does not reduce any agent's individual utility. We have also used real data on CPO and CSPO prices along with maps of oil palm concessions in Indonesia to validate our theoretical results empirically and to derive novel policy insights. These results culminated in the policy recommendations summarized in Table~\ref{tab:one}, which recommends the imposition of maximum penalties and full use of the available budget when the planner seeks unsustainable effort reduction at any cost or when the budget and penalties are jointly sufficient for welfare improvement. On the other hand, when the budget-penalty sufficiency condition is violated and the planner prioritizes welfare improvement to unsustainable effort reduction, we recommend using the  optimal policy given by either $\pbold$ or $\widetilde\pbold$, depending on whether regionally uniform pricing is desired or not.

We hope that our game-theoretic intervention design framework is applicable more broadly to other public-private partnership where sustainability-enhancing activity is considered integral to the well-being of all the stakeholders. As such, we believe that our results have the potential to guide policy design not only in sustainable forestry but in all other areas where  sustainable practices need to be incentivized.

\bibliographystyle{ieeetr} 
\bibliography{bib.bib}

\section{Monetary Components of Palm Oil Prices}\label{sec:monetary_component}

In this section, we examine the factors influencing the prices of sustainable and unsustainable goods in practice. These factors are central to the protocols adopted by RSPO for conducting audits and for providing  sustainability certifications and premiums to incentivize sustainable production.

We begin by comparing the prices and production costs of sustainable goods (i.e., sustainability-certified palm oil goods) with those of unsustainable goods. The main difference between the  prices of sustainable and unsustainable goods is that the  prices of sustainable goods include sustainability premiums, which are calculated as percentages of the conventional prices of these goods. On the other hand, sustainable production also incurs certain additional costs. For example, it requires the concession managers to ensure higher levels of staffing and training, which gives rise to staffing and training costs. Besides, RSPO also levies a per-unit cost of sustainability certification in addition to RSPO membership costs and the costs of conducting audits at various points within the palm oil supply chain. Therefore, the price per unit effort by agent $i$ in activity $A$ can be expressed as $p_i^A = (1 + \eta^\text{Pr}_i )p_i^\text{Base}  - (c^\text{A}_i + c^\text{M}_i)/q_i^A - c^\text{C}_i - c^\text{S\&T}_i$, where $\eta^\text{Pr}_i$ is the per-unit sustainability premium awarded to agent $i$, $p_i^\text{Base}$ is the per-unit base price set by the agent, $q_i^A$ denotes the quantity of goods produced as a result of activity $A$, and $c^\text{A}_i$, $c^\text{M}_i$, $c_i^\text{C}$ and $c_i^\text{S\&T}$ denote, respectively, the (fixed) audit cost to the agent, the (fixed) membership cost, and the per-unit costs of certification, staffing and training. Most stakeholders have reported that $c_i^\text{A}$, $c_i^\text{M}$, and $c_i^\text{S\&T}$ are negligible relative to sustainability premiums and certification costs (\cite{wwf2022business}). Therefore, we approximate $p_i^A = (1 + \eta_i^\text{Pr} )p_i^\text{Base}   - c_i^\text{C}$, and it follows that premiums and certification costs are the primary levers that RSPO can use to raise or reduce $p_i^A$ significantly. 

Similar to the above computations, the calculation of prices per unit effort in activity $B$ (unsustainable production) involves certain costs that may not be apparent on a first glance. For one, it has been observed that unsustainable harvesting practices often lead to poorer yields. For example, the draining of peatlands results in subsidence of the soil, which in turn leads to flooding and loss of crop (\cite{ZSL2020}). This can be modeled as an operational cost that is proportional to the quantity of unsustainable goods produced, because the greater the production, the greater is the area of land used for cultivation. Besides, there are certain reputational costs such as consumer boycotts and loss of brand value associated with unsustainable production. These are likely to be remarkable in the coming years -- as per recent studies, we may observe ``value gaps of as much as 70\% arise between fast-moving consumer goods companies (FMCGs), based upon their reputations'' (\cite{Rijk2019,ZSL2020}). Seen in the context of EUDR's demand that all palm oil goods exported to the EU after 2023 be sustainably sourced~\cite{burgi2023breakout}, it is essential to factor in the profit reductions resulting from such costs into the computation of $p_i^B$. We therefore have $p_i^B = p_i^\text{Base} - c_i^\text{Op}- c_i^\text{Rep}$, where $c_i^\text{Op}$ denotes per-unit operational cost to agent $i$ and $c_i^\text{Rep}$ denotes the monetary component of the per-unit reputational cost to the agent.

\def\pmax{p^{\max}}
\def\dmin{d_{\min}}

\section{Effect of Incentive Bias  on Individual Utility Gains}\label{sec:utility}

In this section, our goal is to numerically illustrate that if the maximum feasible penalties are large enough to satisfy the budget-penalty condition~\eqref{eq:half-budget}, then redistribution leads to utility improvement for every agent even when the external budget is zero. For this purpose, we consider the case $p^{A0}=0.7p^{B0}=0.7 p_0 \allone$ where $p_0$ is a constant scalar (as described in Subsection \ref{subsec:real_prices}) and solve the redistribution problem $\pbr$ for a range of values of the maximum feasible penalties given by $\rhomax$. 

Before plotting our numerical results, we define the \textit{minimum utility gain} $\eta_U$ as the minimum of all the $n$ individual \textit{relative} utility gains resulting from the intervention, i.e.,
$$
    \eta_U := \min_{i\in [n]}\left\{ \frac{u_i(x^{A*}, x^{B*}; p^{A}, p^{B})-u_i(x^{A0}, x^{B0}; p^{A0}, p^{B0})}{u_i(x^{A0}, x^{B0}; p^{A0}, p^{B0}) } \right\}.
$$
Similarly, we define the \textit{incentive bias gain} $\eta_I$ as the relative difference between the incentive bias magnitudes before and after intervention:
$$
    \eta_I := \frac{\left |p^{B}_i - p^{A}_i\right | -  \left |p^{B0}_i - p^{A0}_i\right | }{ \left |p^{B0}_i - p^{A0}_i\right | },
$$
which takes the same value for all $i\in [n]$ because of uniform pricing (i.e., $p^A_i$, $p^{A0}_i$, $p^{B0}_i$, and $p^{B}_i$ are constant with respect to $i$). 

Below, we plot the minimum utility gain as a function of the incentive bias gain (whose value varies with $\rhomax$) for the case considered. 
\begin{figure}[h]
    \centering
    \includegraphics[width=0.5\linewidth]{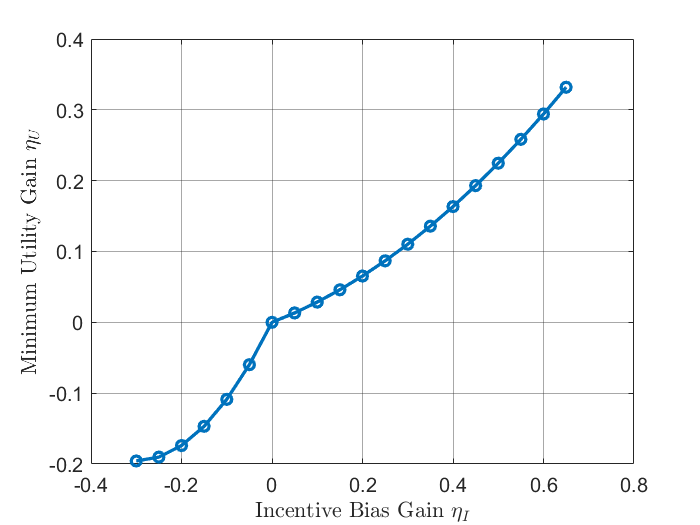}
    \caption{Variation of individual utility gains with incentive bias gains resulting from redistribution in the zero-budget setting}
    \label{fig:enter-label}
\end{figure}

We observe that the minimum utility gain is positive if and only if the incentive bias gain is positive. In other words, the following statements are true:
\begin{enumerate}
    \item If every agent's magnitude of incentive bias increases after intervention, then every agent's utility is guaranteed to increase as a result of the intervention. Since the budget-penalty condition is equivalent to the post-intervention magnitude of every agent's incentive bias exceeding its pre-intervention value (as explained in Section~\ref{sec:optimal}), this validates the second assertion of Theorem~\ref{thm:redistribution}.
    \item If the magnitude of incentive bias decreases after intervention, then there exists at least one agent whose utility decreases as a result of the intervention. Since we have assumed the external budget to be zero ($b=\allzero$), this observation is consistent with the statement of Proposition~\ref{prop:pessimistic}, which says that if there is no external budget and if all the $n$ scalar inequalities defining the budget-penalty condition are violated (so that every agent's post-intervention magnitude of incentive bias  is less than its pre-intervention value), then price redistribution cannot cause welfare to increase, implying that there is at least one agent whose utility either decreases as a result of the intervention or remains equal to its pre-intervention value.  
\end{enumerate}

In conclusion, therefore, interventions that  increase  agents' incentive bias magnitudes are sufficient and in some cases even necessary to improve the welfare of the network.

\section{Comparative Statics}\label{sec:comp_stat}

We now analyze the impact of the game-theoretic parameters, namely the intra-concession substitutability $\beta$ and the network effect parameters $\delta$ and $\mu$, on the welfare and aggregate unsustainable effort resulting from our optimal policies in the post-intervention equilibrium. For this purpose, we restrict our attention to the component-wise uniform pricing problem $\widetilde\pbold$.

\subsubsection{Varying the Intra-Concession Substitutability $\beta$:} Below, we vary $\beta$ over the set $\{0.1, 0.2, 0.3, 0.4\}$ and plot the welfare and aggregate unsustainable effort for $\delta = 0.13$, $\mu = 0.001$, $p_i^{A0} = 0.88 p^B_0$ for all $i\in [n]$, and $\tau^B = 266.69$.

\begin{figure}[h]
    \centering
    \begin{minipage}{0.45\textwidth}
        \centering
        \includegraphics[width=1.1\textwidth]{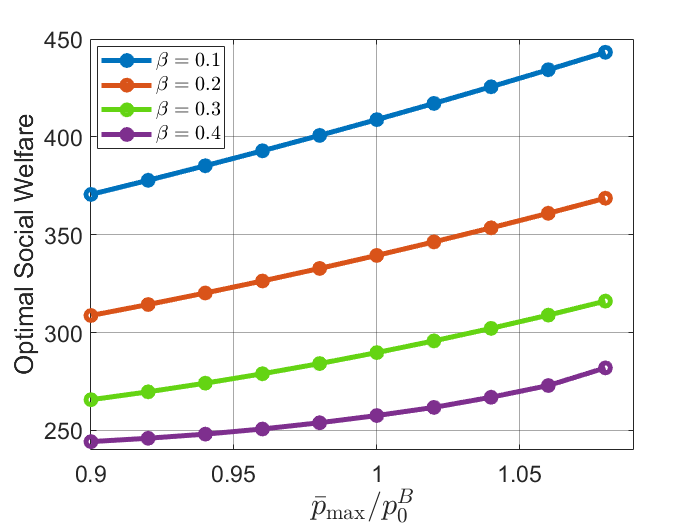} 
        \caption{Optimal welfare vs. upper bounds on CSPO prices}
    \end{minipage}\hfill
    \begin{minipage}{0.45\textwidth}
        \centering
        \includegraphics[width=1.1\textwidth]{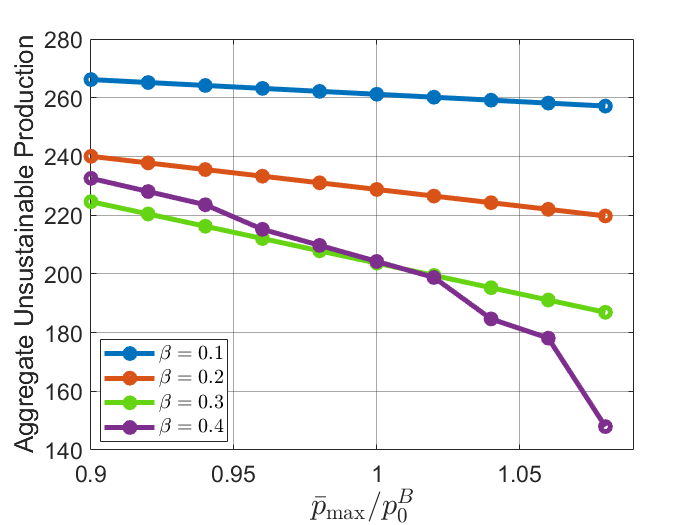} 
        \caption{Aggregate unsustainable effort vs. upper bounds on CSPO prices}
    \end{minipage}
\end{figure}

\textbf{\textit{Inferences:}} The optimal value of equilibrium welfare decreases as the value of $\beta$ increases. This is to be expected because of the negative contribution of intra-concession substitutabilities to  welfare. More non-trivial is the  variation of the post-intervention aggregate unsustainable effort with $\beta$, which can be explained as follows. We know from the definitions of $M^+$ and $M^-$ that the entries of $M^+$ are monotonically decreasing in $\beta$ whereas those of $M^-$ are monotonically increasing in $\beta$.  Since \eqref{eq:equilibrium} can be expressed as $x^{B*} = 0.5(M^+(p^B + p^A) + M^- (p^B - p^A))$,  this means that every agent's level of unsustainable production in the post-intervention equilibrium is monotonically decreasing in $\beta$ provided $p^A\ge p^B$, which holds true for the optimal policy when $\bar p^{\max}/p_0^B\ge 1$. However, we cannot draw such an inference when $p^A\ge p^B$ does not hold.  Therefore, in Figure 5, aggregate unsustainable effort decreases with $\beta$ for $\bar p^{\max}/p^B_0\ge 1$, but the same cannot be said for $\bar p^{\max}/p^B_0\le 1$.

\subsubsection{Varying the Intra-Activity Network Effect Parameter $\delta$:} We now vary $\delta$ over the set $\{0.13, 0.15, 0.17, 0.19\}$ and plot the welfare and aggregate unsustainable effort for $\beta = 0.2$, $\mu = 0.05$, $p_i^{A0}=0.88 p_0^B$ for all $i\in [n]$, and $\tau^B=315.55$.

\begin{figure}[h]
    \centering
    \begin{minipage}{0.45\textwidth}
        \centering
        \includegraphics[width=1.1\textwidth]{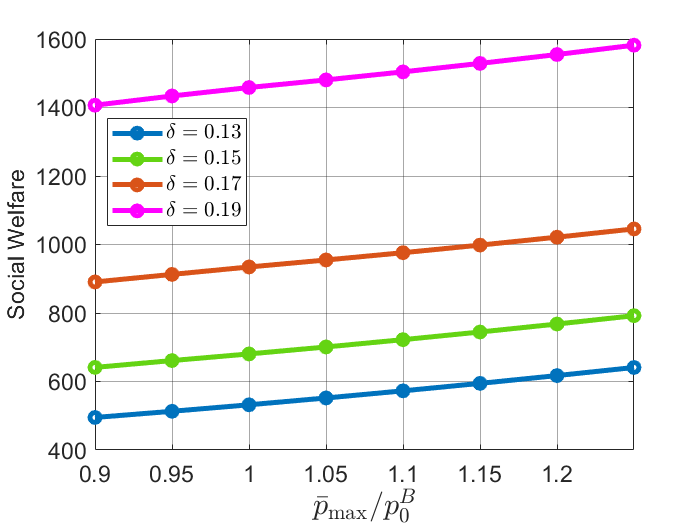} 
        \caption{Optimal welfare vs. upper bounds on CSPO prices}
    \end{minipage}\hfill
    \begin{minipage}{0.45\textwidth}
        \centering
        \includegraphics[width=1.1\textwidth]{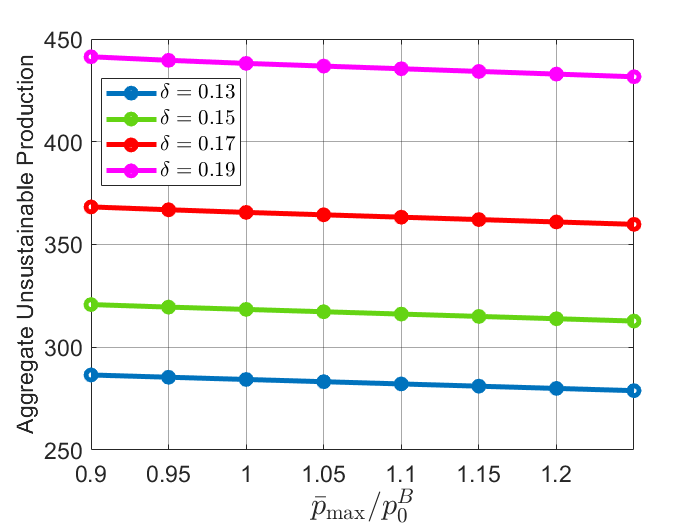} 
        \caption{Aggregate unsustainable effort vs. upper bounds on CSPO prices}
    \end{minipage}
    
\end{figure}

\textit{\textbf{Inferences:}} The optimal welfare and aggregate unsustainable effort both increase as $\delta$ increases. This is because the intra-activity effect contributes positively to welfare, and the prices of unsustainable goods contribute positively to all the  agents' equilibrium levels of unsustainable production, which in turn contribute more positively to their neighbors' equilibrium levels of unsustainable production via increased intra-activity network effect.

\subsubsection{Varying the Cross-Activity Network Effect Parameter $\mu$:}

Below, we vary $\mu$ over the set $\{0.04, 0.05, 0.06, 0.07\}$ and plot the welfare and aggregate unsustainable effort for $\delta = 0.18$, $\beta = 0.4$, $p_i^{A0} = 0.88 p^B_0$ for all $i\in [n]$, and $\tau^B = 520.68$.

\begin{figure}[h]
    \centering
    \begin{minipage}{0.45\textwidth}
        \centering
        \includegraphics[width=1.1\textwidth]{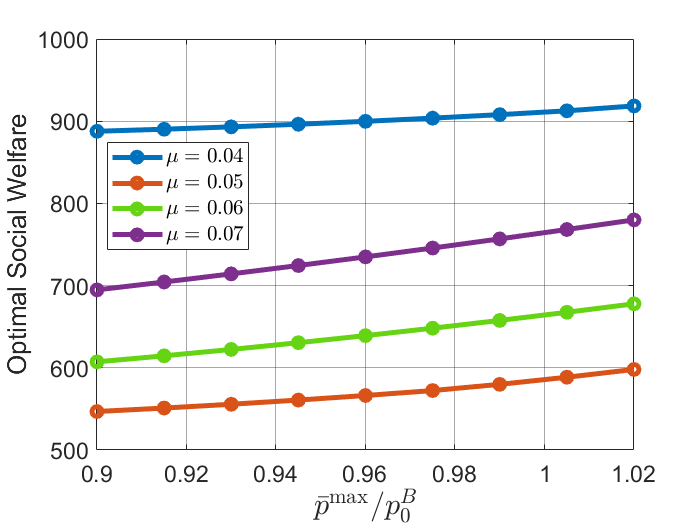} 
        \caption{Optimal welfare vs. upper bounds on CSPO prices}
    \end{minipage}\hfill
    \begin{minipage}{0.45\textwidth}
        \centering
        \includegraphics[width=1.1\textwidth]{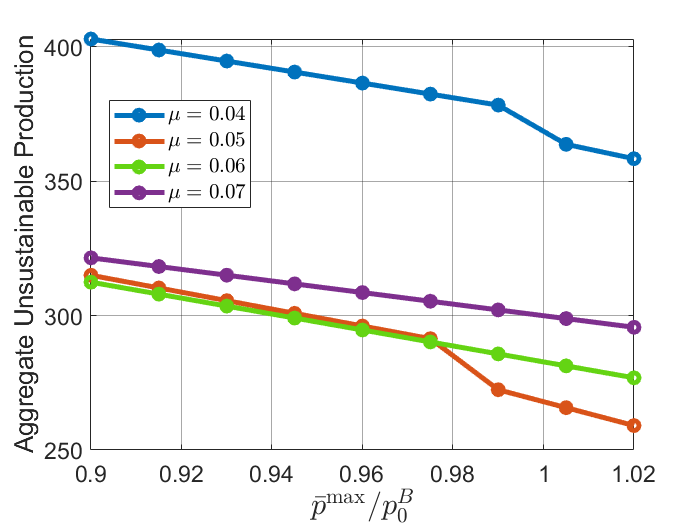} 
        \caption{Aggregate unsustainable effort vs. upper bounds on CSPO prices}
    \end{minipage}
\end{figure}

\textbf{\textit{Inferences:}} We observe the expected results for $\mu\in[0.05,0.07]$, i.e., as $\mu$ varies over this range, both the optimal welfare and aggregate unsustainable effort at the post-intervention equilibrium increase because of increased strength of the cross-activity network effect. This is because the cross-activity network effect contributes positively to welfare, and the prices of sustainable goods contribute positively to sustainable production, which in turn contributes more positively to unsustainable production via the increased cross-activity effect. 

On the other hand, when $\mu=0.04$, the cross-activity network effect is small enough to keep  aggregate unsustainable effort below the desired tolerance for a wider range of pricing policies. This leads to the feasible region of $\widetilde\pbold$ being larger for $\mu = 0.04$ than for $\mu\ge 0.05$, as a result of which the optimal welfare is the highest for $\mu = 0.04$. As for  aggregate unsustainable effort, it is remarkably higher for $\mu = 0.04$ than for $\mu\ge 0.05$. This can be explained as follows: the lower the value of $\mu$, the lower is the value of aggregate unsustainable effort for any given policy (as argued above). Since the tolerance on unsustainable production $\tau^B$ is held constant, this implies that the lower the value of $\mu$, the larger will be the set of feasible policies. Therefore, decreasing $\mu$ from $0.05$ to $0.04$ causes some policies that were previously infeasible due to their violation of the tolerance constraint~\eqref{eq:unsustainable_second_modified} to become feasible. We observe that one of these feasible policies also becomes optimal (note that $\tpmax$ is no longer optimal because it has entries that lie below the threshold defined in~\eqref{eq:p-knot-star}). Since the optimal policy now sets some prices less than the corresponding entries of $\tpmax$, aggregate unsustainable effort increases, which is to be expected in light of $s^+$, which holds as per Theorem~\ref{thm:one}-(i).

\section{Proof of Theorem~\ref{thm:one}}
\begin{proof} We prove the assertions one by one.
\begin{enumerate}

\item [(i)] Suppose $\mu < \beta\delta$. To establish the essential feasibility of $\pbold$, we first claim that ${M_{\Delta}:=M^- - M^+}$ is entry-wise positive whenever $\mu<\beta\delta$. To prove this claim, we observe that
\begin{align*}
    M_{\Delta} &= (1-\beta)^{-1}(I - (1-\beta)^{-1}(\delta-\mu) G)^{-1}\cr
    &\quad- (1+\beta)^{-1}(I - (1+\beta)^{-1}(\delta + \mu)G)^{-1}\cr
    &\stackrel{(a)}\ge(1-\beta)^{-1}(I - (1-\beta)^{-1}(\delta-\mu) G)^{-1}\cr
    &\quad- (1-\beta)^{-1}(I - (1+\beta)^{-1}(\delta + \mu)G)^{-1}\cr
    &\stackrel{(b)}= (1-\beta)^{-1} \sum_{k=0}^\infty \left( \left(\frac{\delta-\mu}{1-\beta}\right)^k - \left(\frac{\delta+\mu}{1+\beta}\right)^k \right)G^k,
\end{align*}
where $(a)$ holds because $\beta>0$ and $(b)$ follows from the Neumann series expansions~\cite[5.6.P26]{horn2012matrix} of $M^+$ and $M^-$, which converge due to Assumption~\ref{assum:uniqueness}. Now, since the strategic interaction network is connected, $G$ is irreducible, i.e., it is the non-negative adjacency matrix of a graph in which there exists a walk from every node $i\in [n]$ to every other node $j\in [n]\setminus\{i\}$. This implies the following: for every pair $(i,j)\in [n]\times [n]$, there exists a $k_{ij}\in \N$ such that $(G^k)_{ij}>0$ for $k=k_{ij}$ (which for $i=j$ implies that there exists a walk from $i$ to itself via another node $\ell\in [n]\setminus\{i\}$). Therefore, it suffices to show that the coefficient of $G^k$ in  the above series expansion is positive for each $k\in\N$. To this end, we observe upon some simplification that $\mu<\beta\delta$ implies $(1-\beta)^{-1}(\delta - \mu) > (1+\beta)^{-1}(\delta + \mu)$.
As a result, we have
\begin{align}\label{eq:handy}
    (1-\beta)^{-k}(\delta - \mu)^k > (1+\beta)^{-k}(\delta + \mu)^k
\end{align}
for all $k\in\N$. This proves that $M_\Delta$ is positive.

Therefore, the function $\R^n\ni z\to-\allone^\top M_{\Delta}^\top z$ is decreasing in every entry of $z$. It now follows from~\eqref{eq:equilibrium} that $2\sum_{i=1}^n x_{i}^{B*} = -\allone^\top M_\Delta^\top p^A + \allone^\top(M^+ + M^-)p^B$ is decreasing in each entry of $p^A$. This completes the proof of $s^+$. Moreover, as all Leontief matrice are positive, $M^+ + M^-$, which is the sum of two scaled Leontief matrices with positive scaling factors, is also positive. Therefore, any price vector $z$ that satisfies $p^{A0} < z< \pmax$ also satisfies~\eqref{eq:unsustainable_effort}, as such a $z$ ensures that 
$2\sum_{i=1}^n x_{i}^{B*} = -\allone^\top M_\Delta ^\top z + \allone^\top (M^+ + M^-)p^B < -\allone^\top M_\Delta^\top p^{A0} + \allone^\top(M^+ + M^-)p^B = 2 \sum_{i=1}^n x_{i}^{B0}$. Hence,
$\pbold$ is essentially feasible.
\item [(ii)] Suppose $\mu = \beta\delta+\varepsilon$ for a given $\varepsilon>0$, and let $G=\allone\allone^\top - I$ be the binary adjacency matrix of the complete graph on $n$ vertices. Additionally, let $\delta^+:=\frac{\delta+\mu}{1+\beta}$ and $\delta^- := \frac{\delta-\mu}{1-\beta}$.

We now analyze the entries of $M_\Delta= M^--M^+$ as follows. Observe that
\begin{align}~\label{eq:immediate}
    &(1\pm \beta)^{-1}M^\pm\cr
    &= (I - \delta^\pm G)^{-1}  \cr
    &=(1+ \delta^{\pm} )^{-1} \left(I - \delta^{\pm} (1+ \delta^{\pm} )^{-1}  \allone\allone^\top \right )^{-1}\cr
    &\stackrel{(a)}=\left(\frac{1}{1+\delta^\pm}\right) I + \frac{\delta^\pm }{(1+\delta^\pm )(1+(n-1)\delta^\pm )}  \allone\allone^\top,
\end{align}
where $(a)$ follows from an application of~\cite[(3.8.1)]{meyer2000matrix}. We can use this derivation along with $M_\Delta = M^- - M^+$ to show that
\begin{align}\label{eq:compact-b-delta}
    M_\Delta =  \frac{ \alpha I - \gamma \allone\allone^\top}{(1-\beta^2)(1+\delta^+)(1+\delta^-)} ,
\end{align}
where 
$$
    \alpha:= \frac{(1+\beta)(1+\delta^-)-(1-\beta)(1+\delta^+)}{(1-\beta^2)(1+\delta^+)(1+\delta^-)}\text{ and }\gamma:=\frac{\delta^+}{1-(n-1)\delta^+} - \frac{\delta^-}{1-(n-1)\delta^- }.
$$

To characterize $\alpha$ and $\gamma$, we now analyze $\delta^+$ and $\delta^-$. To this end, we substitute $\mu=\beta\delta+\varepsilon$ into the expressions defining $\delta^+$ and $\delta^-$ to obtain the relations $\delta^\pm = \delta \pm \frac{\varepsilon}{1\pm \beta}$. We observe from these relations and Assumption~\ref{assum:mu-delta} that $\delta^+>\delta^->0$. Hence, to ensure that Assumption~\ref{assum:uniqueness} is satisfied, it suffices to ensure that $\delta^+\lambda_1(G) < 1$. On the other hand, the connectedness of the network implies that $G$ is irreducible, which, by the Perron-Frobenius theorem~\cite[Page 673]{meyer2000matrix}, further implies that the spectral radius of $G$ equals its greatest eigenvalue, which is given by $\lambda_1(\allone\allone^\top) - 1 =n-1$. Hence, it suffices to have $\delta^+(n-1)<1$ to ensure that Assumption~\ref{assum:uniqueness} holds. To this end, we set $\delta^+=\frac{1-\tilde \varepsilon}{n-1}$ for an appropriate choice of $\tilde \varepsilon\in (0,1)$ that we will make later. 

To analyze $\alpha$, we now substitute $\delta^\pm:=\delta \pm \frac{\varepsilon}{1\pm\beta}$ into the definition of $\alpha$ and then simplify the resulting expression to obtain $\alpha = 2\beta (1+\delta) - 2\beta(1-\beta^2)^{-1}\varepsilon$. It now follows from $\delta^+=\frac{1-\tilde\varepsilon}{n-1}$ that $\alpha$ is bounded on the interval $0<\tilde\varepsilon<1$.

To characterize $\gamma$, we first use $\delta^\pm = \delta \pm\frac{\varepsilon}{1\pm\beta}$ to show that $\delta^- = \delta^+ - \frac{\varepsilon}{1-\beta^2}=\frac{1-\tilde \varepsilon}{n-1}-\frac{2\varepsilon}{1-\beta^2}$. Substituting these expressions for $\delta^+$ and $\delta^-$ into the definition of $\gamma$ results in
\begin{align}\label{eq:nonum}
    \gamma &=\tilde \varepsilon^{-1} \left(\frac{1-\tilde{\varepsilon}}{n-1}\right)\cr
    &\quad - \left(\tilde{\varepsilon}+\frac{2(n-1) \varepsilon}{1-\beta^2}\right)^{-1} \left(\frac{1-\tilde{\varepsilon}}{n-1}-\frac{2 \varepsilon}{1-\beta^2}\right) \cr
    &\geq  \tilde \varepsilon^{-1} \left(\frac{1-\tilde{\varepsilon}}{n-1}\right)-\left(\tilde{\varepsilon}+\frac{2(n-1) \varepsilon}{1-\beta^2}\right)^{-1} \left(\frac{1-\tilde{\varepsilon}}{n-1}\right) \cr
    &=\left(\frac{2\varepsilon}{\tilde \varepsilon} \right)\left(\frac{1-\tilde \varepsilon}{ (1-\beta^2)\tilde \varepsilon + 2(n-1)\varepsilon  } \right).
\end{align}
Observe that the right hand side of~\eqref{eq:nonum} goes to infinity in the limit as $\tilde\varepsilon\to0$. Hence, $\gamma$ can be made arbitrarily large by making $\tilde\varepsilon$ arbitrarily small. As $\alpha$ is bounded, it now follows from~\eqref{eq:compact-b-delta} and the fact that $\beta,\delta^+,\delta^-\in (0,1)$ that all the entries of $-M_\Delta=M^+ - M^-$ can be made positive by choosing  $\tilde \varepsilon$ to be small enough.

For such a choice of $\tilde \varepsilon$, we know from~\eqref{eq:equilibrium} that every agent's level of unsustainable effort at equilibrium is monotonically increasing in $p_i^A$ for every $i\in [n]$. This completes the proof of $s^-$, showing that it is impossible to reduce the aggregate unsustainable effort by choosing $p^A\ge p^{A0}$. This shows that $\pbold$ is essentially infeasible if $\mu=\beta\delta+\varepsilon$. As $\varepsilon>0$ is arbitrary, this completes the proof. 
\item [(iii)] Suppose now that $\mu >\max\left\{\frac{2\beta}{1+\beta^2}\delta, \frac{\beta}{\dmin} \right\} $. As $\beta<1$, we can use $\mu>\frac{2\beta\delta}{1+\beta^2}$ to verify that the following inequality holds for all $k\in \N$:
    \begin{align}\label{eq:zeroth_ineq}
        \frac{(\delta + \mu)^k}{(1+\beta)^{k+1}} - \frac{(\delta - \mu)^k}{(1-\beta)^{k+1}} >0.
    \end{align}
    We now use the Neumann series expansions of $M^+$ and $M^-$ along with the relation $M_\Delta = M^- - M^+$ to obtain
    \begin{align}\label{eq:expansion_M_delta}
        &M_\Delta  \cr
        &= \frac{2\beta}{1-\beta^2} I - \sum_{k=1}^\infty \left(\frac{(\delta + \mu)^k}{(1+\beta)^{k+1}} - \frac{(\delta - \mu)^k}{(1-\beta)^{k+1}}\right)G^k.
    \end{align}
    Thus, the off-diagonal entries of $M_\Delta$ are identical to those of the summation on the right-hand side of~\eqref{eq:expansion_M_delta}, which is a negative matrix because of~\eqref{eq:zeroth_ineq} and because $G^k$ is positive for all $k\ge n-1$ (see the proof of (i)). This shows that the off-diagonal entries of $M_\Delta$ are negative. 
    
    On the other hand, we know from the definition of $\dmin$ that $G\allone\ge \dmin\allone$. Using induction, this can be generalized to $G^k\allone \ge \dmin^k \allone$ for all $k\ge 1$. Consequently, we have
\begin{align*}
    &M_\Delta \allone\cr 
    &= \frac{2\beta}{1-\beta^2}\allone - \sum_{k=1}^\infty \left(\frac{(\delta + \mu)^k}{(1+\beta)^{k+1}} - \frac{(\delta - \mu)^k}{(1-\beta)^{k+1}} \right)G^k\allone\cr
    &\le \frac{2\beta}{1-\beta^2}\allone - \sum_{k=1}^\infty \left(\frac{(\delta + \mu)^k}{(1+\beta)^{k+1}} - \frac{(\delta - \mu)^k}{(1-\beta)^{k+1}} \right)\dmin^k\allone\cr
    &\stackrel{(a)}=\frac{2\beta}{1-\beta^2}\allone  - \frac{(\delta+ \mu)\dmin}{(1+\beta)^2 (1 - (1+\beta)^{-1}(\delta+\mu)\dmin) } \allone\cr
    &\quad+ \frac{(\delta- \mu)\dmin}{(1-\beta)^2 (1 - (1-\beta)^{-1}(\delta-\mu)\dmin) }\allone,
\end{align*}
where $(a)$ follows from geometric series expansion. On simplification, the above inequality reduces to $M_\Delta\allone \le - c_0(2\beta \dmin^2(\delta^2 - \mu^2) + (1-\beta^2)(\dmin\mu-\beta) )\allone$ 
for some $c_0>0$. It now follows from Assumption~\ref{assum:mu-delta} and $\mu> \frac{\beta}{\dmin}$ that $M_\Delta \allone< \allzero$. On the other hand, the strategic interaction network being undirected implies that $G$ is symmetric, which in turn implies that $M_\Delta$ is symmetric. Hence, $M_\Delta\allone < \allzero$ implies that $\allone^\top M_\Delta = (M_\Delta \allone)^\top <\allzero$. It follows that $2\sum_{i=1}^n x_i^{B*} = -\allone^\top M_\Delta^\top p^A + \allone^\top (M^+ + M^-)p^B$ is increasing in each entry of $p^A$, which completes the proof of $s^-$. 
Therefore, the value of  aggregate unsustainable effort is non-decreasing in every entry of $p^A$, and hence,~\eqref{eq:price_up_second} implies that aggregate unsustainable effort is minimized at $p^{A0}$. In other words, $\pbold$ is essentially infeasible. 
\end{enumerate}
\end{proof}

\section{Proof of Theorem~\ref{thm:non-negative}}

We first prove a few auxiliary lemmas that are used in either the proof or the statement of the theorem.

\begin{lemma}\label{lem:invertibility}
    Suppose Assumptions \ref{assum:uniqueness} and \ref{assum:domination} hold. Then, $M_\Delta= M^--M^+$ is invertible.
\end{lemma}

\begin{proof}
    It suffices to show that every eigenvalue of $M_\Delta$ is non-zero. By the definitions of $M^+ $ and $ M^-$, we have the following relation:
    $
        {M^- - M^+ = ((1-\beta)I - (\delta -\mu)G)^{-1} - ((1+\beta)I - (\delta +\mu)G)^{-1}}.
    $
    On the other hand, as the strategic interaction network is modeled as a undirected graph, $G$ is a symmetric matrix. Hence, it is diagonalizable. It now follows from \cite[(7.3.5)]{meyer2000matrix} that for every eigenvalue $\lambda$ of $G$, the corresponding eigenvalue of $M_\Delta$ is given by
    $$
        ((1-\beta) - (\delta -\mu)\lambda)^{-1} - ((1+\beta)I - (\delta +\mu)\lambda)^{-1},
    $$
    which can be simplified to a fraction with a  numerator given by $2(\beta-\mu\lambda)$. Therefore, it is enough to show that $\beta - \mu\lambda$ is non-zero for every eigenvalue $\lambda$ of $G$. 

    Now, if $\lambda<0$, then it is clear from $\beta\in (0,1)$ and $\mu\ge 0$ that $\beta - \mu\lambda\ge \beta>0$. So, suppose $\lambda\ge 0$. We then have  
    \begin{align}\label{eq:quick}
        \beta - \mu\lambda \stackrel{(a)}\ge  \beta-\beta\delta\lambda
        = \beta(1-\delta \lambda) \ge \beta(1-\delta\rho),
    \end{align}
    where $(a)$ follows from $\lambda\ge 0$ and Assumption~\ref{assum:domination}, and $\rho$ denotes the spectral radius of $G$. On the other hand, we have the following for the spectral radius $\rho $ of $G$:
    \begin{align*}
        \delta \rho = \left(\frac{\delta-\beta\delta}{1-\beta}\right) \rho \stackrel{(a)} \le \left(\frac{\delta-\mu}{1-\beta}\right) \rho \stackrel{(b)}< 1,
    \end{align*}
    where $(a)$ holds because of Assumption~\ref{assum:domination} and $1-\beta>0$, and $(b)$ follows from Assumption~\ref{assum:uniqueness}. Thus, $1-\delta\rho>0$. It now follows from $\beta>0$ and \eqref{eq:quick} that $\beta-\mu\lambda>0$. Hence, $\beta-\mu\lambda$ is non-zero, as required. 
\end{proof}

\begin{lemma}\label{lem:very_short}
    Suppose Assumption~\ref{assum:uniqueness} holds. Then the spectral radius of $((\mu G - \beta I)M)^2$ is less than one. 
\end{lemma}

\begin{proof}
    Since the eigenvalues of the square of a matrix are the squares of its eigenvalues, it suffices to show that the absolute value of every eigenvalue of $(\mu G - \beta I)M$ is less than one. We prove this by contradiction and by noting that for every eigenvalue $\lambda$ of $G$, the corresponding eigenvalue of $(\mu G - \beta I)M = (\mu G - \beta I)(I-\delta G)^{-1}$ is given by $\frac{(\mu \lambda-\beta)}{1-\delta\lambda}$. 

    Suppose, therefore, that $\left|\frac{\mu \lambda-\beta}{1-\delta\lambda}\right|\ge 1$. Then, either $\frac{(\mu \lambda-\beta)}{1-\delta\lambda}\ge 1$ or $\frac{(\mu \lambda-\beta)}{1-\delta\lambda}\le -1$. In the former case, we obtain 
    $\mu\lambda -\beta \ge 1-\delta\lambda $ by multiplying both sides of the inequality
    $\frac{(\mu \lambda-\beta)}{1-\delta\lambda}\ge 1$ by $1-\delta\lambda$ (which is positive because $1-\delta\lambda$ exceeds $1-\delta\rho$, which is positive as shown in the proof of Lemma~\ref{lem:invertibility}). Rearranging the terms and dividing both sides by $1+\beta$ transforms this inequality to $\left(\frac{\delta +\mu}{1+\beta}\right)\lambda\ge 1$, which contradicts Assumption~\ref{assum:uniqueness}. Similarly, we can show that in the latter case, the inequality $\frac{(\mu \lambda-\beta)}{1-\delta\lambda}\le -1$ implies $\left(\frac{\delta -\mu}{1-\beta}\right)\lambda\ge 1$, which contradicts Assumption~\ref{assum:uniqueness}. This shows that our assumption that $\left|\frac{\mu \lambda-\beta}{1-\delta\lambda}\right|\ge 1$ was wrong, thereby proving that $\left|\frac{\mu \lambda-\beta}{1-\delta\lambda}\right|<1$. We have thus shown that every eigenvalue of $(\mu G - \beta I) M$ has an absolute value less than one. Therefore, the spectral radius of $(\mu G - \beta I) M$ is less than one. 
\end{proof}

\begin{lemma}\label{lem:submatrix}
    Suppose Assumption~\ref{assum:uniqueness} holds. Then $(M^++M^-)_S$ is invertible for all $S\subset [n]$.
\end{lemma}

\begin{proof}
    From Assumption~\ref{assum:uniqueness} and the fact that $G$ has real eigenvalues owing to its symmetry, we know that all the eigenvalues of $M^+$ and $M^-$ are positive. As both these matrices are symmetric, it follows that they are also positive-definite. Hence, $M^+ + M^-$ is positive-definite. As every principal minor of a positive-definite matrix is positive (see~\cite{meyer2000matrix}), this implies that $(M^+ + M^-)_S$ is invertible. 
\end{proof}

\begin{lemma}\label{lem:short}
    Suppose Assumption~\ref{assum:uniqueness} holds and $\mu<\beta\delta$.  Then $(I-\delta G)^{-1}$ exists and is non-negative. Also, ${(\mu G -\beta I)(I-\delta G)^{-1} \le -\beta I}$.
\end{lemma}
\begin{proof}
    Let $\rho $ denote the spectral radius of $G$, and observe that the spectral radius $\delta \rho$ of $\delta G$ satisfies
    \begin{align*}
        \delta \rho = \left(\frac{\delta-\beta\delta}{1-\beta}\right) \rho \stackrel{(a)} \le \left(\frac{\delta-\mu}{1-\beta}\right) \rho \stackrel{(b)}< 1,
    \end{align*}
    where $(a)$ holds because $\beta\delta>\mu$ and $1-\beta>0$, and $(b)$ follows from Assumption~\ref{assum:uniqueness}. Therefore, $I-\delta G$ is invertible and its inverse is given by the Neumann series expansion $(I-\delta G)^{-1} = \sum_{k=0}^\infty (\delta G)^k$. Since $G$ is non-negative, this implies that $(I-\delta G)^{-1}$ is also non-negative.
    
    We now define $M:=(I-\delta G)^{-1}$ and observe that
    \begin{align*}
        (\mu G -\beta I)(I-\delta G)^{-1} &= (\mu G - \beta I)M\cr
        &\le \beta \delta G M - \beta M\cr
        & = \beta (\delta G - I)(I-\delta G)^{-1}\cr
        & = - \beta (I-\delta G)(I-\delta G)^{-1}\cr
        & = -\beta I,
    \end{align*}
    where the inequality holds because $\mu<\beta\delta$, $G$ is non-negative  and because $M$ is non-negative as established above. 
\end{proof}

\begin{lemma}\label{lem:long_matrix}
    Suppose Assumptions~\ref{assum:uniqueness} and \ref{assum:mu-delta} hold, suppose $\mu<\beta \delta$, and let $M:=(I-\delta G)^{-1}$. Then the matrix $\delta G + (\beta^2I +\mu^2 G^2 - 2\beta\mu G)M$ is non-negative.
\end{lemma}
\begin{proof} 
    Observe that the first derivatives of the entries of the matrix $(\mu^2G^2 - 2\beta\mu G)M$ with respect to $\mu$ are given by the entries of $(2\mu G^2 - 2\beta G )M$, which can be bounded by noting that 
    $$ 
        (2\mu G^2 - 2\beta G )M = 2(\mu G - \beta I)M G\le -2 \beta G\le O,
    $$
    where the equality holds because $G$ commutes with $M=\sum_{k=0}^\infty (\delta G)^k$, the first inequality follows from Lemma~\ref{lem:short} and the non-negativity of $G$, and the second inequality is also a consequence of the non-negativity of $G$. Thus, every entry of $(\mu^2 G^2 - 2\beta \mu G)M$ is non-increasing in $\mu$ on $(0,\beta\delta)$. Given that $\mu<\beta\delta$ and that $(\mu^2 G^2 - 2\beta \mu G)M$ is continuous in $\mu$, it follows that
    $        
        (\mu^2 G^2 - 2\beta \mu G)M \ge (\mu_0^2 G^2 - 2\beta \mu_0 G)M,
    $
    where $\mu_0:=\beta\delta$. Thus,
    \begin{align}
        (\mu^2 G^2 - 2\beta \mu G)M \ge (\beta^2\delta^2G^2 - 2\beta^2\delta G)M.
    \end{align}
    On the basis of this, we have
    \begin{align}
        &\delta G + (\beta^2I +\mu^2 G^2 - 2\beta\mu G)M\cr
        &\ge \delta G + (\beta^2 I + \beta^2\delta^2 G^2 - 2\beta^2\delta G )M\cr
        &=\delta G + \beta^2 M(I + \delta^2G^2 - 2\delta G)\cr
        &=\delta G + \beta^2 M(I-\delta G)^2\cr
        &\stackrel{(a)} = \delta G + \beta^2 (I-\delta G)\cr
        &=\beta^2 I + (1-\beta^2) \delta G\cr
        &\stackrel{(b)}\ge O,
    \end{align}
    where $(a)$ holds because $M=(I-\delta G)^{-1}$ by definition, and $(b)$ holds because $\beta\in (0,1)$. 
\end{proof}

\begin{lemma}\label{lem:spectral}
    Let $S\subset [n]$ and $T=[n]\setminus S$ be given, and suppose the matrices $I- (\delta G_S + (\delta^2+\mu^2)G_{ST}G_{TS})$ and $I - (\delta G + (\mu G - \beta I)^2 M)_T$, where $M:=(I-\delta G)^{-1}$ exists as per Lemma~\ref{lem:short}, are invertible for all $\beta$, $\delta$, and $\mu$ satisfying Assumptions~\ref{assum:uniqueness} and~\ref{assum:domination}. Then, the spectral radii of the matrices $\delta G_S + (\delta^2+\mu^2)G_{ST}G_{TS}$ and $(\delta G + (\mu G - \beta I)^2 M)_T$ are less than 1 for all such $\beta$, $\delta$, and $\mu$.
\end{lemma}
\begin{proof}
We first note the following: as the given matrices are both symmetric, their eigenvalues are all real.

Suppose first that the matrix $(\delta G + (\mu G - \beta I)^2 M)_T$ has an eigenvalue $\lambda=1$. Then the corresponding eigenvalue of $I - (\delta G + (\mu G - \beta I)^2 M)_T$ is 0, which contradicts the fact that this matrix is invertible. Hence, it is impossible to have $\lambda=1$. Now, suppose $\lambda>1$.  In this case, note that $(\delta G + (\mu G - \beta I)^2 M)_T=(\delta G + (\mu G - \beta I)^2 (I-\delta G)^{-1})_T$, and hence, we can make the value of $\lambda$ arbitrarily small (i.e., arbitrarily close to zero) simply by making $\beta$, $\delta$, and $\mu$ arbitrarily small (say, by setting $\delta =  \beta$ and $\mu = \alpha\beta\delta=\alpha \beta^2$ for some $\alpha\in(0,1)$ and then scaling $\beta$ by an arbitrarily small factor) without violating either $\mu<\beta\delta$ or Assumption~\ref{assum:uniqueness}. Moreover, as the eigenvalues of any matrix are continuous in its entries, $\lambda$ is a continuous function of $\beta$, $\delta$, and $\mu$. By the Intermediate Value Theorem, therefore, there exist values of $\delta=\beta$ and $\mu=\alpha\beta^2$ for which $\lambda=1$, in which case, the corresponding eigenvalue of $I - (\delta G + (\mu G - \beta I)^2 M)_T$ becomes 0, once again contradicting the fact that this matrix is invertible for all choices of $\beta$, $\mu$, and $\delta$ that satisfy $\mu<\beta\delta$ as well as Assumption~\ref{assum:uniqueness}. This shows that $\lambda<1$. As $\lambda$ is an arbitrary eigenvalue of $(\delta G + (\mu G - \beta I)^2 M)_T$, we have shown that every eigenvalue of this matrix is less than 1. As the matrix $\delta G + (\mu G - \beta I)^2 M = \delta G + (\beta^2 I + \mu^2 G^2 - 2\beta\mu G)M$ is non-negative by Lemma~\ref{lem:long_matrix}, all its submatrices are also non-negative. Hence, the spectral radius of $(\delta G + (\mu G - \beta I)^2 M)_T$ is itself an eigenvalue of the matrix (see~\cite[(8.3.1)] {meyer2000matrix}). Therefore, this spectral radius is less than 1.     

The above arguments can also be applied to the matrix $\delta G_S + (\delta^2 + \mu^2 )G_{ST}G_{TS}$ so as to show that its spectral radius is also less than 1. 
\end{proof}

\begin{lemma}\label{lem:(1)}
Let $\check p^A\in\R^n_{\ge 0}$ be any pricing policy. If there exists an index set $\tilde S\subset [n]$ and a pricing  policy $p^{A(1)}\le \check p^A$ under which $\frac{\partial u_i}{\partial x_i^B}\bigg\lvert_{ ( x^{A(1)},  x^{B(1)} )} \le 0$ for all $i\in [n]$, where $\tilde T := [n]\setminus \tilde S$,
\begin{align}\label{eq:similar_first}
          x_{\tilde T}^{A(1)} &= \frac{1}{2}(M_{\tilde T}^+ + M_{\tilde T}^-)( p_{\tilde T}^{A(1)} + \delta G_{\tilde T \tilde S}   x_{\tilde S}^{A(1)}) \cr
          &\quad +\frac{1}{2} (M_{\tilde T}^+ - M_{\tilde T}^-)(p_{\tilde T}^{B} + \mu G_{\tilde T \tilde S}  x_{\tilde S}^{A(1)}),\cr
          x_{\tilde T}^{B(1)} &= \frac{1}{2}(M_{\tilde T}^+ + M_{\tilde T}^-)(p_{\tilde T}^{B} +  \mu G_{\tilde T \tilde S}   x_{\tilde S}^{A(1)}) \cr
          &\quad +\frac{1}{2} (M_{\tilde T}^+ - M_{\tilde T}^-)( p_{\tilde T}^{A(1)} + \delta G_{\tilde T \tilde S}  x_{\tilde S}^{A(1)}),
    \end{align}
    $x^{B(1)}_{\tilde S}=\allzero$, and    \begin{align}\label{eq:similar-s-a}
        x_{\tilde S}^{A(1)} &= \frac{1}{2} \bigg(I - \delta G_{\tilde S} - (\delta^2+\mu^2) G_{\tilde S \tilde T}G_{\tilde T \tilde S}\bigg)^{-1}\cr
        &\quad \cdot\bigg((\delta+\mu)G_{\tilde S \tilde T}M_{\tilde T}^+( p^{A(1)}_{\tilde T}+p^B_{\tilde T})\cr
        &\quad \quad + (\delta-\mu)G_{\tilde S \tilde T}M_{\tilde T}^-( p_{\tilde T}^{A(1)} - p_{\tilde T}^{B}) \bigg),
    \end{align}
then $(\check x^A,\check x^B)$ as defined by~\eqref{eq:compare_one} and~\eqref{eq:next-s-a} below satisfies $\frac{\partial u_i}{\partial x_i^B}\bigg\lvert_{ ( \check x^A,\check x^B )}\le 0$ for all $i\in [n]$ under the policy $\check p^A$, and $(\check x^A,\check x^B)$ is a Nash equilibrium of the network game $\left([n],\R_{\ge 0}^n\times \R^n_{\ge 0} ,\{u_i\}_{i\in[n]} \right)$ provided $\check x^A\ge \allzero$ and $\check x^B\ge \allzero$.
\begin{align}\label{eq:compare_one}
         \check x_{\tilde T}^A(\check p^A) &= \frac{1}{2} (M_{\tilde T}^+ + M_{\tilde T}^-)(\check p_{\tilde T}^A + \delta G_{\tilde T \tilde S}  \check x_{\tilde S}^A)\cr
         &\quad +\frac{1}{2} (M_{\tilde T}^+ - M_{\tilde T}^-)(p_{\tilde T}^B + \mu G_{\tilde T \tilde S} \check x_{\tilde S}^A),\cr
         \check x_{\tilde T}^B(\check p^A) &=\frac{1}{2} (M_{\tilde T}^+ + M_{\tilde T}^-)(p_{\tilde T}^B + \mu G_{\tilde T \tilde S}  \check x_{\tilde S}^A) \cr
         &\quad +\frac{1}{2} (M_{\tilde T}^+ - M_{\tilde T}^-)(\check p_{\tilde T}^A + \delta G_{\tilde T \tilde S} \check x_{\tilde S}^A),\cr
         \check x^B_{\tilde S}(\check p^A) &= \allzero,
    \end{align}
    and    \begin{align}\label{eq:next-s-a}
        \check x_{\tilde S}^{A}(\check p^A) &= \frac{1}{2} \bigg(I - \delta G_{\tilde S} - (\delta^2+\mu^2) G_{\tilde S \tilde T}G_{\tilde T \tilde S}\bigg)^{-1}\cr
        &\quad \cdot \bigg((\delta+\mu)G_{\tilde S \tilde T}M_{\tilde T}^+( \check p^{A}_{\tilde T}+p^B_{\tilde T}) \cr
        &\quad\quad  + (\delta-\mu)G_{\tilde S \tilde T}M_{\tilde T}^-( \check p_{\tilde T}^{A} - p_{\tilde T}^{B}) \bigg).
    \end{align}
\end{lemma}
\begin{proof}
    Consider two hypothetical network games with agent utilities given by~\eqref{eq:main}. In the first game \\$\left([n],\R^n\times (\R^{n-|\tilde S|}\times \{0\}^{|\tilde S|}),\{u_i\}_{i\in[n]} \right)$, the agents indexed by $\tilde S$ do not engage in activity $B$, whereas in the second game $\left([n],\R^n\times (\R^{n-|\tilde S|}\times \R_{\ge 0}^{|\tilde S|}),\{u_i\}_{i\in[n]} \right)$, the agents indexed by $\tilde S$ are constrained to have non-negative effort levels in activity $B$, but those indexed by $\tilde T:=[n]\setminus \tilde S$ can  have effort levels of any sign and magnitude. For the first of these games, we can use Lemma~\ref{lem:recast} to show that the first-order conditions for Nash equilibria are equivalent to the system of equations given by~\eqref{eq:compare_one} and~\eqref{eq:next-s-a} under the pricing policy $\check p^A$.

We now compute the first partial derivatives of the agents' true utilities in the second game (which are given by~\eqref{eq:main}) with respect to the agents' own sustainable and unsustainable effort levels, and we need to show that these partial derivatives are all non-positive. To this end, observe that the first-order conditions~\eqref{eq:compare_one} of the first  game are the following by definition.

\begin{align}\label{eq:statement_partial_derivative}
    \frac{\partial u_i }{\partial x_i^A}\bigg\lvert_{ (\check x^A, \check x^B)} &= 0 \quad\text{for all }i\in[n],\cr
     \frac{\partial u_i }{\partial x_i^B}\bigg\lvert_{ (\check x^A, \check x^B)} &= 0 \quad\text{for all }i\in \tilde T.
\end{align}
Since agents have the same utility functions in both the hypothetical games,~\eqref{eq:statement_partial_derivative} holds true for the second game as well. As $u_i$ is strongly concave in $x_i^A$ and $x_i^B$ for each $i\in [n]$, to show that $(\check x^A, \check x^B)$ is a Nash equilibrium of the second hypothetical game, it suffices to show that $\frac{\partial u_i}{\partial x_i^B} \le 0$ for all $i\in S$. To this end, we first compute the first derivative of the utilities in~\eqref{eq:main} and note that
\begin{align}\label{eq:partial}
    \frac{\partial u_i}{\partial x_i^B}\bigg\lvert_{ (\check x^A, \check x^B)} &= p_i^B - \check x_i^B - \beta \check x_i^A + (G (\delta \check x^B + \mu \check x^A))_i\cr
    & = p_i^B  - \beta \check x_i^A + (G (\delta \check x^B + \mu \check x^A))_i
\end{align}
for all $i\in \tilde S$, a relation we use later.

Next, by Lemma~\ref{lem:spectral}, the matrix $(I - (\delta G_{\tilde S} + (\delta^2 + \mu^2)G_{\tilde S \tilde T} G_{\tilde T \tilde S})^{-1}$ has a Neumann series expansion $\sum_{k=0}^\infty H^k$, where $H := \delta G_{\tilde S} + (\delta^2+\mu^2)G_{\tilde S \tilde T}G_{\tilde T \tilde S}$. Using this along with~\eqref{eq:similar-s-a}, we  obtain
\begin{align}\label{eq:final-x-s-a}
        &x_{\tilde S}^{A(1)}\cr
        &= \frac{1}{2} \bigg(\sum_{k=0}^\infty H^k \bigg)\bigg((\delta+\mu)G_{\tilde S \tilde T}M_{\tilde T}^+( p^{A(1)}_{\tilde T}+p^B_{\tilde T})\cr
        &\quad\quad\quad\quad\quad\quad+ (\delta-\mu)G_{\tilde S \tilde T}M_{\tilde T}^-( p_{\tilde T}^{A(1)} - p_{\tilde T}^{B}) \bigg).
    \end{align}

 Similarly, we can express $\check x_{\tilde S}$ as 
    \begin{align}\label{eq:final-hat}
        &\check x_{\tilde S}^A\cr
        &= \frac{1}{2} \bigg(\sum_{k=0}^\infty H^k \bigg)\bigg((\delta+\mu)G_{\tilde S \tilde T}M_{\tilde T}^+(\check p^A_{\tilde T}+p^B_{\tilde T})\cr
        &\quad\quad\quad\quad\quad\quad+ (\delta-\mu)G_{\tilde S \tilde T}M_{\tilde T}^-(\check p_{\tilde T}^A - p_{\tilde T}^B) \bigg).
    \end{align}
    It now follows from~\eqref{eq:final-x-s-a} and~\eqref{eq:final-hat} that
    \begin{align}\label{eq:gets_bigger}
        &\check x_{\tilde S}^A - x_{\tilde S}^{A(1)} \cr
        &= \frac{1}{2} \bigg(\sum_{k=0}^\infty H^k \bigg)\bigg((\delta+\mu)G_{\tilde S \tilde T}M_{\tilde T}^+ \bigg(\check p^A_{\tilde T} - p_{\tilde T}^{A(1)} \bigg)\cr
        &\quad\quad\quad\quad\quad\quad+ (\delta-\mu)G_{\tilde S \tilde T}M_{\tilde T}^-\bigg(\check p_{\tilde T}^A - p_{\tilde T}^{A(1)}\bigg)  \bigg)\cr
        &\ge \allzero,
    \end{align}
    where the inequality holds because $\check p^A \ge p^{A(1)}$, $\delta-\mu>0$ by Assumption~\ref{assum:mu-delta} (which is implied by Assumption~\ref{assum:domination}), and all the matrices appearing in the above expression are non-negative.

    On the other hand, we deduce from~\eqref{eq:compare_one} that 
    \begin{align}\label{eq:long_first}
        &\check x_{\tilde T}^B - x_{\tilde T}^{B(1)}\cr 
        &= \frac{1}{2}\bigg( (M_{\tilde T}^+ + M_{\tilde T}^-)\mu G_{\tilde T \tilde S}(\check x_{\tilde S}^A - x_{\tilde S}^{A(1)})\cr
        &\quad\quad  + (M_{\tilde T}^+ - M_{\tilde T}^-)((\check p_{\tilde T}^A - p_{\tilde T}^{A(1)}) \cr
        &\quad\quad  + \delta G_{\tilde T \tilde S} (\check x_{\tilde S}^{A} - x_{\tilde S}^{A(1)} )) \bigg)\cr
        &= \frac{1}{2}\bigg((\delta+\mu) M_{\tilde T}^+ - (\delta-\mu)M_{\tilde T}^-\bigg)\cr
        &\quad\cdot G_{\tilde T \tilde S}(\check x_{\tilde S}^A - x_{\tilde S}^{A(1)}) \cr
        &\quad + \frac{1}{2}(M_{\tilde T}^+ - M_{\tilde T}^-)(\check p_{\tilde T}^A - p_{\tilde T}^{A(1)})\cr
        &\stackrel{(a)}=\frac{1}{2}\bigg(\sum_{k=0}^\infty \bigg(\bigg(\frac{\delta+\mu}{1+\beta} \bigg)^{k+1} - \bigg(\frac{\delta-\mu}{1-\beta} \bigg)^{k+1} \bigg) \bigg)\cr
        &\quad\cdot G_{\tilde T}^k G_{\tilde T \tilde S}(\check x_{\tilde S}^A - x_{\tilde S}^{A(1)}) \cr
        &\quad - \frac{1}{2}(M_{\tilde T}^- - M_{\tilde T}^+)(\check p_{\tilde T}^A - p_{\tilde T}^{A(1)})\cr
        &\le \allzero,
    \end{align}
    where $(a)$ is obtained by using the Neumann series expansions of $M_{\tilde T}^\pm = (1\pm\beta)^{-1}\left(I - \left(\frac{\delta\pm \mu}{1\pm \beta}\right) G_{\tilde T}\right)^{-1}$ and the inequality holds because (i) $(M_{\tilde T}^- - M_{\tilde T}^+)(\check p_{\tilde T}^A - p_{\tilde T}^{A(1)})\ge \allzero$ since $\check p^A \ge p^{A(1)}$ and $M_{\tilde T}^- - M_{\tilde T}^+$ is a non-negative matrix (which can be shown by applying Theorem~\ref{thm:one}-(i)  to the subnetwork of the original network whose adjacency matrix is $\tilde G$), (ii) $\check x_{\tilde S}^A - x_{\tilde S}^{A(1)} \ge \allzero$ as established in~\eqref{eq:gets_bigger}, and (iii) the summation in the first summand can be bounded as
    \begin{align*}
        &\sum_{k=0}^\infty \left(\left(\frac{\delta+\mu}{1+\beta} \right)^{k+1} - \left(\frac{\delta-\mu}{1-\beta} \right)^{k+1} \right)\cr
        &\stackrel{(b)}\le \sum_{k=0}^\infty \left(\left(\frac{\delta+\beta\delta}{1+\beta} \right)^{k+1} - \left(\frac{\delta-\beta\delta}{1-\beta} \right)^{k+1} \right)\cr
        &= \sum_{k=0}^\infty(\delta^{k+1} - \delta^{k+1})\cr
        &= 0,
    \end{align*}
    where $(b)$ is a consequence of Assumption~\ref{assum:mu-delta}.

    Besides, as $(\check x^A, \check x^B)$ satisfies the following first-order conditions according to Lemma~\ref{lem:recast}:
    \begin{align*}
    \check p^{A} -  \check x^{A} -\beta 
        \begin{pmatrix}
             \check x^{B}_{\tilde T}\\
            \allzero
        \end{pmatrix} 
        + \delta G \check  x^{A} +\mu G \begin{pmatrix}
            \check x^{B}_{\tilde T} \\
            \allzero
        \end{pmatrix}
        &=\allzero,\cr
        \left(p^{B} - \begin{pmatrix}
             \check x^{B}_{\tilde T}\\
            \allzero
        \end{pmatrix}  -\beta  \check x^{A}
        + \delta G \begin{pmatrix}
             \check x^{B}_{\tilde T}\\
            \allzero
        \end{pmatrix} +\mu G  \check x^{A}\right)_{\tilde T}
        &=\allzero.
    \end{align*}
    Hence, we have
    $
        \check x^A = \check p^A - \beta \check x^B + \delta G\check x^A + \mu G\check x^B,
    $
    which implies that
    \begin{align}\label{eq:check}
        \check x^A = M( \check p^A  + (\mu G - \beta I)\check x^B).
    \end{align}

    We now use the above relations in conjunction with~\eqref{eq:partial} to bound $\frac{\partial u_i}{\partial x_i^B}\bigg\lvert_{ (\check x^A, \check x^B)}$. Observe from~\eqref{eq:partial} the following: when the prices of sustainable goods are given by $\check p^A$, these partial derivatives are given by the entries of the vector ${p^B - \beta \check x^A + G(\delta \check x^B + \mu \check x^A)}$, which we now characterize as follows for $i\in \tilde S:$
\begin{align}\label{eq:near_end}
    &\frac{\partial u_i}{\partial x_i^B}\bigg\lvert_{ (\check x^A, \check x^B)}\cr
    &= \bigg(p^B - \beta \check x^A + G(\delta \check x^B + \mu \check x^A)\bigg)_i \cr
    &= \bigg(p^B + (\mu G - \beta I)\check x^A + \delta G\check x^B\bigg)_i \cr
    &\stackrel{(a)}= \bigg(p^B + (\mu G - \beta I)M( \check p^A  \cr
    &\quad \quad + (\mu G - \beta I)\check x^B) + \delta G \check x^B\bigg)_i \cr
    &=\bigg(p^B + (\mu G - \beta I)M \check p^A \cr
    &\quad + (\delta G + (\mu G - \beta I)^2 M ) \check x^B\bigg)_i \cr
    &= \bigg(p^B + (\mu G -\beta I )M \check p^A \cr
    &\quad + \bigg(\delta G + (\beta^2 I  + \mu^2 G^2 - 2\beta \mu G)M\bigg)\check x^B\bigg)_i,
\end{align}
where $(a)$ follows from~\eqref{eq:check}.

    Similarly, when the per-unit prices of sustainable goods are given by $ p^{A(1)}$, the partial derivatives $\frac{\partial u_i}{\partial x_i^B}\bigg\lvert_{ ( x^{A(1)},  x^{B(1)} )}$ are given by the entries of 
    \begin{align}\label{eq:also_near_end}
        &\frac{\partial u_i}{\partial x_i^B}\bigg\lvert_{ ( x^{A(1)},  x^{B(1)} )}\cr
        &= \Big(p^B + (\mu G -\beta I )M  p^{A(1)} \cr
        &\quad\,\, + \big(\delta G + (\beta^2 I  + \mu^2 G^2 - 2\beta \mu G)M\big) x^{B(1)} \Big)_i.
    \end{align}

    We  now use~\eqref{eq:near_end} and~\eqref{eq:also_near_end} to establish the following:
    \begin{align}\label{eq:to_clinch}
        &\frac{\partial u_i}{\partial x_i^B}\bigg\lvert_{ ( \check x^A,\check x^B )}- \frac{\partial u_i}{\partial x_i^B}\bigg\lvert_{ (  x^{A(1)},\check x^{B(1)} )}\cr
        & = \bigg(p^B + (\mu G -\beta I )M \check p^A \cr
        &\quad+ \bigg(\delta G + (\beta^2 I  + \mu^2 G^2 - 2\beta \mu G)M\bigg)\check x^B\bigg)_i\cr
        &\quad - \bigg(p^B + (\mu G -\beta I )M  p^{A(1)} \cr
        &\quad+ \bigg(\delta G + (\beta^2 I  + \mu^2 G^2 - 2\beta \mu G)M\bigg) x^{B(1)} \bigg)_i\cr
        &=\bigg((\mu G - \beta I )M (\check p^A - p^{A(1)}) \cr
        &\quad+ \bigg(\delta G + (\beta^2 I  + \mu^2 G^2 - 2\beta \mu G)M\bigg)\cr
        &\quad\,\,\cdot(\check x^B - x^{B(1)})\bigg)_i\cr
        &\le 0,
    \end{align}
    where the last step is a consequence of~\eqref{eq:long_first}, $x^{B(1)}_{\tilde S} = \check x^B_{\tilde S} = \allzero$, Lemma~\ref{lem:short}, Lemma~\ref{lem:long_matrix}, and the inequality $\check p^A\ge p^{A(1)}$. 

    Given that $\frac{\partial u_i}{\partial x_i^B}\bigg\lvert_{ ( x^{A(1)},  x^{B(1)} )} \le 0$ for all $i\in [n]$, the inequality in~\eqref{eq:to_clinch} implies that we have $\frac{\partial u_i}{\partial x_i^B}\bigg\lvert_{ ( \check x^A,\check x^B )}\le 0$ for all $i\in [n]$ under the policy $\check p^A$, as required.  
\end{proof}

\begin{lemma}\label{lem:recast}
Let $S\subset[n]$ and $T=[n]\setminus S$. Then, the relations   
\begin{align*}
    p^{A} -  \hat x^{A*} -\beta 
        \begin{pmatrix}
             \hat x^{B*}_{T}\\
            \allzero
        \end{pmatrix} 
        + \delta G  x^{A*} +\mu G \begin{pmatrix}
            \hat x^{B*}_{T} \\
            \allzero
        \end{pmatrix}
        &=\allzero,\cr
        \left(p^{B} - \begin{pmatrix}
             \hat x^{B*}_{T}\\
            \allzero
        \end{pmatrix}  -\beta  \hat x^{A*}
        + \delta G \begin{pmatrix}
             \hat x^{B*}_{T}\\
            \allzero
        \end{pmatrix} +\mu G  \hat x^{A*}\right)_{T}
        &=\allzero,
\end{align*}
which are the first-order-conditions for the Nash equilibrium of the game that enforces $x_S = \allzero$ while retaining the utility structure~\eqref{eq:main}, are equivalent to 
\begin{align}\label{eq:unnamed}
         \hat x_T^{A*} &= \frac{1}{2}(M_T^+ + M_T^-)( p_T^A + \delta G_{TS}  \hat x_S^{A*})\cr
         &\quad+\frac{1}{2} (M_T^+ - M_T^-)(p_T^B + \mu G_{TS} \hat x_S^{A*}),\cr
         \hat x_T^{B*} &=\frac{1}{2} (M_T^+ + M_T^-)(p_T^B + \mu G_{TS}  \hat x_S^{A*})\cr
         &\quad +\frac{1}{2} (M_T^+ - M_T^-)( p_T^A + \delta G_{TS} \hat x_S^{A*}),\nonumber\\
       &\left(I - \delta G_S - (\delta^2+\mu^2) G_{ST}G_{TS}\right)\hat x_S^{A*}\cr
       &=\frac{1}{2}\Big((\delta+\mu)G_{ST}M_T^+(p^A_T+p^B_T) \cr
       &\quad + (\delta-\mu)G_{ST}M_T^-(p_T^A - p_T^B) \Big).
\end{align}
\end{lemma}
\begin{proof}
We rewrite the given relations in terms of sub-matrices indexed by $T$ and $S$. This gives
    \begin{align}\label{eq:no-label}
        \hat x^{A*}_T &=  p^A_T + \delta G_T  \hat x_T^{A*} + \mu G_T \hat x_T^{B*} - \beta  \hat x^{B*}_T + \delta G_{TS} \hat x_{S}^{A*},\cr
        \hat x^{B*}_T &= p^B_T + \delta G_T \hat x_T^{B*} + \mu G_T  \hat x_T^{A*} - \beta  \hat x^{A*}_T + \mu G_{TS} \hat x_{S}^{A*},\cr
         \hat x_S^{A*} &= p_S^A + \delta G_S \hat x_S^{A*} + \mu G_{ST}  \hat x_T^{B*} + \delta G_{ST}  \hat x_T^{A*}.
    \end{align}
    Adding the first two equations yields
    \begin{align*}
        (\hat x^{A*}_T + \hat x^{B*}_T )&= p^A_T + p^B_T + (\delta + \mu) G_T(\hat x^{A*}_T+\hat x^{B*}_T)\cr
        &\quad- \beta(\hat x^{A*}_T + \hat x^{B*}_T) + (\delta+\mu)G_{TS}\hat x^{A*}_S,
    \end{align*}
    rearranging which yields
    \begin{align*}
        &((1+\beta)I - (\delta +\mu)G_T)(\hat x^{A*}_T + \hat x^{B*}_T)\cr
        &= (p^A_T + \delta G_{TS}\hat x^{A*}_S) + (p^B_T + \mu G_{TS}\hat x^{A*}_S).
    \end{align*}
    Equivalently, 
    \begin{align}\label{eq:sum}
        &\hat x^{A*}_T + \hat x^{B*}_T\cr
        &= M_T^+((p^A_T + \delta G_{TS}\hat x^{A*}_S) + (p^B_T + \mu G_{TS}\hat x^{A*}_S)),
    \end{align}
    where $M^+_T:= ((1+\beta)I - (\delta+\mu)G_{T})^{-1}$, which exists by virtue of Assumption~\ref{assum:uniqueness} and~\cite[Statement 5.(a), Page 599]{debreu1953nonnegative}. 

    Likewise, subtracting the two equations in~\eqref{eq:no-label} yields
    \begin{align}\label{eq:difference}
        &\hat x^{A*}_T - \hat x^{B*}_T \cr
        &= M_T^-((p^A_T + \delta G_{TS}\hat x^{A*}_S) - (p^B_T + \mu G_{TS}\hat x^{A*}_S)),
    \end{align}
    where $M^-_T:= ((1-\beta)I - (\delta-\mu)G_{T})^{-1}$, which exists by virtue of Assumption~\ref{assum:uniqueness} and~\cite[Statement 5.(a), Page 599]{debreu1953nonnegative}.
    
    Now, adding and subtracting~\eqref{eq:sum} and~\eqref{eq:difference} gives
    \begin{align}
         \hat x_T^{A*} &=\frac{1}{2} (M_T^+ + M_T^-)( p_T^A + \delta G_{TS}  \hat x_S^{A*})\cr
         &+\frac{1}{2} (M_T^+ - M_T^-)(p_T^B + \mu G_{TS} \hat x_S^{A*}),\cr
         \hat x_T^{B*} &=\frac{1}{2} (M_T^+ + M_T^-)(p_T^B + \mu G_{TS}  \hat x_S^{A*})\cr
         &+\frac{1}{2} (M_T^+ - M_T^-)( p_T^A + \delta G_{TS} \hat x_S^{A*}),
    \end{align}
    Plugging the above expressions for $\hat x^{A*}_T$ and $\hat x^{B*}_T$ into the third equation of~\eqref{eq:no-label} gives
    \begin{align}
        &\left(I - \delta G_S - (\delta^2+\mu^2) G_{ST}G_{TS}\right)\hat x_S^{A*}\cr
        &=\frac{1}{2}\Big((\delta+\mu)G_{ST}M_T^+(p^A_T+p^B_T) \cr
        &\quad + (\delta-\mu)G_{ST}M_T^-(p_T^A - p_T^B) \Big)
    \end{align}
    as required. We have thus shown that \eqref{eq:no-label} implies \eqref{eq:unnamed}. By using \eqref{eq:unnamed} to express both left-hand and right-hand sides of \eqref{eq:no-label} in terms of $\hat x^{A*}_S$, we can show that \eqref{eq:unnamed} implies \eqref{eq:no-label}, thereby establishing the required equivalence. 
\end{proof}

\begin{proof}{Proof of Theorem \ref{thm:non-negative}.} We prove the statements in the order (i)-(iii)-(ii).
\begin{enumerate}
    \item [(i)] We first make two observations. First, since  Assumptions~\ref{assum:uniqueness} and~\ref{assum:domination} hold, we know from Theorem~\ref{thm:one}-(i)  that $s^+$ holds true. Therefore, $M_\Delta = M^- - M^+$ is positive. Next, we can use the definition of $\plim$ to show that the right-hand side of~\eqref{subeq:x-B} for $p^A=\plim$ reduces to
    \begin{align}\label{eq:p_lim_makes_0}
        (M^+ + M^-)p^B + (M^+ - M^-)\plim=\allzero.
    \end{align}
    Now, since $\mu<\beta\delta$, the non-negativity of $M^+ + M^-$ (respectively, $M^- - M^+$) implies that ${(M^+ + M^-)p^A}$ (respectively,  $(M^+ - M^-)p^A$) is  monotonically non-decreasing (respectively, non-increasing) in every entry of $p^A$. It now follows from $p^{A0}\le p^A$ (respectively, $p^A\le \plim$ and~\eqref{eq:p_lim_makes_0}) that the right-hand-side of~\eqref{subeq:x-A} (respectively,~\eqref{subeq:x-B}) is non-negative. Therefore, the effort levels given by~\eqref{eq:equilibrium} constitute a Nash equilibrium not only of the game $\left([n],\R^n\times \R^n ,\{u_i\}_{i\in[n]} \right)$ but also  of the game $\left([n],\R_{\ge 0}^n\times \R^n_{\ge 0} ,\{u_i\}_{i\in[n]} \right)$. 

    To prove uniqueness of the Nash equilibrium with positive sustainable effort levels, suppose there exists a  Nash equilibrium $(x^{A'}, x^{B'})\ne (x^{A*},x^{B*})$ such that $x^{A'}>\allzero$. Since the first order conditions $\frac{\partial u_i}{\partial x_i^{B}}=\frac{\partial u_i}{\partial x_i^{A}}=0$ have a unique solution, which is characterized by~\eqref{eq:equilibrium} (as established in the proof of~\cite[Theorem 4]{chen2018multiple}), it follows that there exists an $i\in [n]$ such that either $\frac{\partial u_i}{\partial x_i^B}\Big\lvert_{(x^A,x^B) = (x^{A'},x^{B'})}\ne 0$ or $\frac{\partial u_i}{\partial x_i^{A}}\Big\lvert_{(x^A,x^B) = (x^{A'},x^{B'})}\ne 0$. Suppose w.l.o.g. that $\frac{\partial u_i}{\partial x_i^B}\Big\lvert_{(x^A,x^B) = (x^{A'},x^{B'})}\ne 0$. Due to the strong concavity of the utility function, if  $\frac{\partial u_i}{\partial x_i^B}\Big\lvert_{(x^A,x^B) = (x^{A'},x^{B'})}> 0$, the marginal utility gain obtained by agent $i$ on incrementing $x_i^{B}$ will be positive. On the other hand, if  $\frac{\partial u_i}{\partial x_i^B}\Big\lvert_{(x^A,x^B) = (x^{A'},x^{B'})}< 0$ with $x_i^{B'}>0$, then the marginal utility gain obtained by agent $i$ on decrementing $x_i^{B}$ from $x_i^{B'}>0$ will be positive. Marginal utility gains being positive implies that $(x^{A'},x^{B'})$ is not a Nash equilibrium, thereby contradicting our definition of $(x^{A'},x^{B'})$. 

    Therefore, it is impossible for the  inequality $\frac{\partial u_i}{\partial x_i^B}\Big\lvert_{(x^A,x^B) = (x^{A'},x^{B'})}\ne 0$ to hold for any $i\in [n]$ unless $\frac{\partial u_i}{\partial x_i^B}\Big\lvert_{(x^A,x^B) = (x^{A'},x^{B'})}< 0$ and $x_i^{B'}=0$. It follows that $\frac{\partial u_i}{\partial x_i^B}\Big\lvert_{(x^A,x^B) = (x^{A'},x^{B'})} \le 0$ for all $i\in [n]$. By using~\eqref{eq:main} to compute these partial derivatives, this set of inequalities can be shown to be equivalent to
    \begin{align}
        p^B - (I-\delta G)x^{B'}  + (\mu G -\beta I)x^{A'} \le \allzero.
    \end{align}
    Multiplying both sides by $M$ yields
    $$
        M p^B - M(I-\delta G)x^{B'}  + (\mu G -\beta I)M x^{A'} \le \allzero,
    $$
    where we have used the non-negativity of $M$ (established in Lemma~\ref{lem:short}) and the fact that $M$ and $G$ commute because $M=(I-\delta G)^{-1}$. We now use $M(I-\delta G) = I$ to simplify the above inequality as follows:
    \begin{align}\label{eq:headache_1}
        x^{B'} \ge M(p^B + (\mu G - \beta I)x^{A'}).
    \end{align}
    Since $\frac{\partial u_i}{\partial x_i^B}\Big\lvert_{(x^A,x^B) = (x^{A*},x^{B*})}= 0$, similar arguments can be used to show that
    \begin{align}\label{eq:headache_5}
        x^{B*} \ge M(p^B + (\mu G - \beta I)x^{A*}).
    \end{align}
    Similarly, we can argue that it is impossible for $\frac{\partial u_i}{\partial x_i^A}\Big\lvert_{(x^A,x^B) = (x^{A'},x^{B'})}\ne 0$ to hold  for any $i\in[n]$ unless $\frac{\partial u_i}{\partial x_i^B}\Big\lvert_{(x^A,x^B) = (x^{A'},x^{B'})}< 0$ and $x_i^{A'}=0$. However, we know that $x_i^{A'}>0$ for all $i\in[n]$ by definition. It follows that $\frac{\partial u_i}{\partial x_i^B}\Big\lvert_{(x^A,x^B) = (x^{A'},x^{B'})}= 0$ for all $i\in [n]$. Equivalently,
    $$
        p^A - (I-\delta G)x^{A'} + (\mu G -\beta I)x^{A'} = \allzero.
    $$
    Multiplying both sides by $M$ and simplifying the second term by noting that $M(I-\delta G) = I$ yields
    \begin{align}\label{eq:headache_2}
        x^{A'} = M p^A + (\mu G - \beta I)M x^{B'}.
    \end{align}
    Combining~\eqref{eq:headache_1} and~\eqref{eq:headache_2} in light of the negativity of $(\mu G - \beta I)M$ (which is established in Lemma~\eqref{lem:short}) now yields
    $$
        x^{A'}
        \le M p^A + (\mu G - \beta I)M^2 (p^B + (\mu G -\beta I)x^{A'})
        = M p^A + (\mu G - \beta I)M^2 p^B + ((\mu G - \beta I)M)^2 x^{A'},
    $$
    i.e.,
    $$
        (I - ((\mu G -\beta I)M)^2) x^{A'} \le M p^A + (\mu G - \beta I) M^2 p^B.
    $$
    Now, Lemma \ref{lem:very_short} implies that $I - ((\mu G -\beta I)M)^2$ has an inverse with the following Neumann series expansion: $(I - ((\mu G -\beta I)M)^2)^{-1} = \sum_{k=0}^\infty(((\mu G -\beta I)M)^2)^k$. Moreover, this series expansion is non-negative because $(\mu G - \beta I)M$ is a non-positive matrix (and hence its square is a non-negative matrix) by Lemma~\ref{lem:short}.  Therefore, pre-multiplying both sides of the above inequality by $(I - ((\mu G -\beta I)M)^2)^{-1}$ results in
    \begin{align}\label{eq:headache_3}
        x^{A'} \le \sum_{k=0}^\infty L^k(M p^A +  (\mu G - \beta I) M^2 p^B),
    \end{align}
    where $L:=((\mu G -\beta I)M)^2$. On the other hand, since $\frac{\partial u_i}{\partial x_i^B}\Big\lvert_{(x^A,x^B) = (x^{A*},x^{B*})}= 0$, similar arguments can be used to show that 
    \begin{align}\label{eq:headache_4}
        x^{A*} = \sum_{k=0}^\infty L^k(M p^A +  (\mu G - \beta I) M^2 p^B).
    \end{align}
    \eqref{eq:headache_3} and \eqref{eq:headache_4} together imply $x^{A'}\le x^{A*}$. We use this inequality along with~\eqref{eq:headache_1} to bound $x^{B'}$ as follows:
    \begin{align}
        x^{B'} &\ge M p^B + (\mu G -\beta I)M x^{A'}\cr
        &\stackrel{(a)}\ge M p^B + (\mu G -\beta I)M x^{A*}\cr
        &\stackrel{(b)}= x^{B*},
    \end{align}
    where $(a)$ follows from the negativity of $(\mu G - \beta I)M$ (established in Lemma~\ref{lem:short}), and (b) follows from~\eqref{eq:headache_5}. Thus, we have shown that $x^{B*}\le x^{B'}$. Since $x_i^{B'}=0$, it follows that $x^{B*}_i=0$. However, this contradicts the positivity of $x^{B*}$  as implied by the positivity of $M^+ $ and $ M^-$ (which can be shown by invoking Lemma \ref{lem:short_1} with $T=[n]$)  and  the inequality $p^A<\plim$. Therefore, our supposition that there exists an equilibrium distinct from $(x^{A*},x^{B*})$ with positive sustainable effort levels is incorrect. This completes the proof of (i).
    \item [(iii)-(a)] The given inequality implies that the right-hand-side of~\eqref{subeq:x-B} is non-negative, and $p^A\ge p^{A0}$ again implies that the right-hand-side of~\eqref{subeq:x-A} is non-negative. Repeating the arguments used to establish (i) now completes the proof.
    \item [(iii)-(b)] For notational ease, we will drop the indexing $(S)$ from the notation $\hat x^{A*}(S)$ and $\hat x^{B*}(S)$ when the indexing set $S$ is clear from the context.
    
    Observe from~\eqref{eq:aliter} and~\eqref{eq:equilibrium} that $(\hat x^{A*},\hat x^{B*})$ is the unique interior Nash equilibrium of the game $\left([n],\R^n\times \R^n,\{\hat u_i\}_{i\in[n]} \right)$, where $\hat u_i(x^{A*},x^{B*}):=u_i(x^{A*}, x^{B*}; p^A, \hat p^B)$, i.e., $\hat u_i$ is the modified utility of agent $i$ that results from replacing $p^B$ with $\hat p^B$. 
 Therefore, $(\hat x^{A*},\hat x^{B*})$  is the unique solution of the following system of equations (which are first-order equilibrium conditions, i.e., they set the partial first derivative of $u_i$ w.r.t. each of $x_i^{A*}$ and $x_i^{B*}$ to 0 for all $i\in[n]$):
    \begin{align}\label{eq:run_out_of_labels}
        p^A - \hat x^{A*} - \beta \hat x^{B*} + \delta G \hat x^{A*} + \mu G \hat x^{B*} &= \allzero,\cr
        \hat p^B - \hat x^{B*} - \beta \hat x^{A*} + \delta G \hat x^{B*} + \mu G \hat x^{A*} &= \allzero.
    \end{align}
    On the other hand,~\eqref{eq:aliter} also implies 
    \begin{align}
        2\hat x^{B*}_S &= ((M^+ + M^-)\hat p^B + (M^+ - M^-)p^A)_S\cr
        &= (M^+ + M^-)_S \hat p^B_S + (M^+ + M^-)_{ST} \hat p^B_T\cr
        &\quad+ (M^+ - M^-)_{S}p^A_S + (M^+ - M^-)_{ST}p^A_T\cr
        &\stackrel{(a)}= \bigg((M^-  - M^+)_{ST} p_T^A + (M^- - M^+)_S p_S^A \cr
        &\quad - (M^+ + M^-)_{ST} p_T^B \bigg)+ (M^+ + M^-)_{ST}  p^B_T\cr
        &\quad+ (M^+ - M^-)_{S}p^A_S + (M^+ - M^-)_{ST}p^A_T\cr
        &=\allzero,
    \end{align}        
    where $(a)$ follows from the definition of $\hat p^B$. Thus, $\hat x^{B*}_S=\allzero$, which means that~\eqref{eq:run_out_of_labels} implies
    \begin{align*}
    p^{A} - \hat{x}^{A*} 
    - \beta 
    \begin{pmatrix}
        \hat{x}^{B*}_{T} \\
        \allzero
    \end{pmatrix} 
    + \delta G \hat{x}^{A*} 
    + \mu G 
    \begin{pmatrix}
        \hat{x}^{B*}_{T} \\
        \allzero
    \end{pmatrix} 
    &= \allzero, \\
        \hat p^{B} 
        - \begin{pmatrix}
            \hat{x}^{B*}_{T} \\
            \allzero
        \end{pmatrix} 
        - \beta \hat{x}^{A*}
        + \delta G 
        \begin{pmatrix}
            \hat{x}^{B*}_{T} \\
            \allzero
        \end{pmatrix} 
        + \mu G \hat{x}^{A*}
    &= \allzero,\\
    \hat x^{B*}_S &= \allzero,
\end{align*}
which can also be expressed as 
\begin{align}\label{eq:big1}
    p^{A} - \hat{x}^{A*} 
    - \beta 
    \begin{pmatrix}
        \hat{x}^{B}_{T} \\
        \allzero
    \end{pmatrix} 
    + \delta G \hat{x}^{A*} 
    + \mu G 
    \begin{pmatrix}
        \hat{x}^{B}_{T} \\
        \allzero
    \end{pmatrix} 
    &= \allzero, \\
    \label{eq:big2}\left(
        p^{B} 
        - \begin{pmatrix}
            \hat{x}^{B}_{T} \\
            \allzero
        \end{pmatrix} 
        - \beta \hat{x}^{A*}
        + \delta G 
        \begin{pmatrix}
            \hat{x}^{B}_{T} \\
            \allzero
        \end{pmatrix} 
        + \mu G \hat{x}^{A*}
    \right)_{T} 
    &= \allzero,\\
    \hat p^B_S - \beta \hat x^{A*}_S + \delta G_{ST} \hat x^{B*}_T + \mu G_S \hat x^{A*}_S + \mu G_{ST}\hat x^{A*}_T&=\allzero,\label{eq:trouble} \\
    \hat x^{B*}_S &= \allzero\label{eq:just_once},
\end{align}
where~\eqref{eq:big2} holds because $\hat p^B_T = p^B_T$. 

In particular, $(x^A, x^B)= (\hat x^{A*},\hat x^{B*})$ solves the following system of equations.
\begin{align}\label{eq:big3}
    p^{A} - {x}^{A} 
    - \beta 
    \begin{pmatrix}
        {x}^{B}_{T} \\
        \allzero
    \end{pmatrix} 
    + \delta G {x}^{A} 
    + \mu G 
    \begin{pmatrix}
        {x}^{B}_{T} \\
        \allzero
    \end{pmatrix} 
    &= \allzero, \\
    \label{eq:big4}\left(
        p^{B} 
        - \begin{pmatrix}
            {x}^{B}_{T} \\
            \allzero
        \end{pmatrix} 
        - \beta {x}^{A}
        + \delta G 
        \begin{pmatrix}
            {x}^{B}_{T} \\
            \allzero
        \end{pmatrix} 
        + \mu G {x}^{A}
    \right)_{T} 
    &= \allzero,\\
     x^{B}_S &= \allzero\label{eq:just_second}.
\end{align}
We now show that the above system of equations is equivalent to \eqref{eq:aliter}. For this purpose, consider any solution $(\hat x^A, \hat x^B)$ of the  system \eqref{eq:big3} - \eqref{eq:just_second}. Then, we can define $\check p^B_S := \beta \hat x^A_S - \delta G_{ST} x_T^B - \mu G_S x_S^A - \mu G_{ST} x_T^A $ so that we have $\check p^B_S - \beta \hat x^A_S + \delta G_{ST} x_T^B + \mu G_S x_S^A + \mu G_{ST} x_T^A = \allzero$, which is identical to \eqref{eq:trouble}, except that $\hat p^B$ is replaced by $\check p^B$. Therefore, the system  \eqref{eq:big3} - \eqref{eq:just_second} is equivalent to 

\begin{align}    p^{A} - {x}^{A} 
    - \beta 
    \begin{pmatrix}
        {x}^{B}_{T} \\
        \allzero
    \end{pmatrix} 
    + \delta G {x}^{A} 
    + \mu G 
    \begin{pmatrix}
        {x}^{B}_{T} \\
        \allzero
    \end{pmatrix} 
    &= \allzero, \\
    \left(
        p^{B} 
        - \begin{pmatrix}
            {x}^{B}_{T} \\
            \allzero
        \end{pmatrix} 
        - \beta {x}^{A}
        + \delta G 
        \begin{pmatrix}
            {x}^{B}_{T} \\
            \allzero
        \end{pmatrix} 
        + \mu G {x}^{A}
    \right)_{T} 
    &= \allzero,\\
     \check p^B_S - \beta  x^{A}_S + \delta G_{ST}  x^{B}_T + \mu G_S  x^{A}_S + \mu G_{ST} x^{A}_T&=\allzero, \\
     x^{B}_S &= \allzero,
\end{align}
which is identical to
    \begin{align*}
    p^{A} - {x}^{A} 
    - \beta 
    \begin{pmatrix}
        {x}^{B}_{T} \\
        \allzero
    \end{pmatrix} 
    + \delta G {x}^{A} 
    + \mu G 
    \begin{pmatrix}
        {x}^{B}_{T} \\
        \allzero
    \end{pmatrix} 
    &= \allzero, \\
         \check p^{B} 
        - \begin{pmatrix}
            {x}^{B}_{T} \\
            \allzero
        \end{pmatrix} 
        - \beta {x}^{A}
        + \delta G 
        \begin{pmatrix}
            {x}^{B}_{T} \\
            \allzero
        \end{pmatrix} 
        + \mu G {x}^{A}
    &= \allzero,\\
     x^{B}_S &= \allzero,
\end{align*}
where $\check p^B_T:= p^B_T$. We have just shown that $(\hat x^A, \hat x^B)$ (which we defined above as a solution of \eqref{eq:big3}-\eqref{eq:just_second})
satisfies the first-derivative conditions for the Nash equilibrium of the game $\left([n],\R^n\times \R^n ,\{\check u_i\}_{i\in[n]} \right)$, implying that 
\begin{align}\label{eq:semi-semi}
    \hat x^A &= \frac{1}{2}\left( (M^+ + M^-)p^A + (M^+ - M^-)\check p^B\right),\cr
    \hat x^B &= \frac{1}{2}\left((M^+ + M^-)\check p^B + (M^+ - M^-)p^A\right).
\end{align}
By the second vector equation above, we have
\begin{align*}
    \allzero &= 2\hat x^B_S\cr
    &= (M^+ + M^-)_S p_S^A + (M^+ + M^-)_{ST} p_T^A \cr
    &\quad + (M^+ + M^-)_{ST}\check p^B_T + (M^+ + M^-)_S \check p_S^B\cr
    &=  (M^+ - M^-)_S p_S^A \cr
    &\quad + (M^+ - M^-)_{ST} p_T^A + (M^+ + M^-)_{ST} p^B_T\cr
    &\quad + (M^+ + M^-)_S \check p_S^B,
\end{align*}
which implies that
$$
    (M^+ + M^-)_S\check p_S^B = (M^- - M^+)_S p_S^A + (M^- - M^+)_{ST} p_T^A - (M^+ + M^-)_{ST} p^B_T = (M^+ + M^-)_S \hat p_S^B,
$$
where the second equality is a straightforward consequence of the definition of $\hat p^B$. Thus, $(M^+ + M^-)_S\check p_S^B = (M^+ + M^-)_S\hat p_S^B$. By the invertibility of $(M^+ + M^-)_S$ (established in Lemma \ref{lem:submatrix}), we have $\check p^B = \hat p^B$. Thus, \eqref{eq:semi-semi} is identical to \eqref{eq:aliter}. Therefore, we have shown that the system \eqref{eq:big3} - \eqref{eq:just_second}  implies \eqref{eq:aliter}. Since we have already shown that~\eqref{eq:aliter} implies \eqref{eq:big3} - \eqref{eq:just_second}, we have shown that these two systems of equations are equivalent to each other.

We prove in Lemma \ref{lem:recast} that each of the above systems is equivalent to
    \begin{subequations}\label{eq:top}
    \begin{align}
        \begin{split}
            &\left(I-\delta G_S-\left(\delta^2+\mu^2\right) G_{S T} G_{T S}\right)\hat{x}_S^{A*}  \\
            &=\frac{1}{2} \bigg((\delta+\mu) G_{S T} M_T^{+}\left(p_T^A+p_T^B\right)\cr 
            &\quad\quad + (\delta-\mu) G_{S T} M_T^{-}\left(p_T^A-p_T^B\right)\bigg),
        \end{split} \label{eq:subeq-x-S-A} \\
        \hat{x}_S^{B*} &= \mathbf{0}, \label{eq:subeq2} \\
        \begin{split}
            \hat{x}_T^{A*} &= \frac{1}{2}\left(M_T^{+}+M_T^{-}\right)\left(p_T^A+\delta G_{T S} \hat{x}_S^{A*}\right) \\
            &\quad + \frac{1}{2} \left(M_T^{+}-M_T^{-}\right)\left(p_T^B+\mu G_{T S} \hat{x}_S^{A*}\right),
        \end{split} \label{eq:subeq3} \\
        \begin{split}
            \hat{x}_T^{B*} &= \frac{1}{2}\left(M_T^{+}+M_T^{-}\right)\left(p_T^B+\mu G_{T S} \hat{x}_S^{A*}\right) \\
            &\quad + \frac{1}{2}\left(M_T^{+}-M_T^{-}\right)\left(p_T^A+\delta G_{T S} \hat{x}_S^{A*}\right).
        \end{split} \label{eq:subeq4}
    \end{align}\label{eq:everything}
\end{subequations}   

By the uniqueness of the interior Nash equilibrium $(\hat x^{A*},\hat x^{B*})$, the above equivalence implies that $\hat x^{A*}_S$ is unique, i.e.,~\eqref{eq:subeq-x-S-A} has a unique solution. This implies that the matrix $I-\delta G_S - (\delta^2 +\mu^2)G_{ST}G_{TS}$ is invertible (otherwise, given any solution $y$ of~\eqref{eq:subeq-x-S-A}, we could have picked a vector $v\ne \allzero$ from its null-space and observed that $y+v$ would be another solution of~\eqref{eq:subeq-x-S-A}, in which case replacing $\hat x_S^{A*}$ with $y+v$ in~\eqref{eq:top} would then result in a complete solution set for $(\hat x^{A*}_T, \hat x^{A*}_S, \hat x^{B}_T)$, which would be distinct from the original solution set, thereby contradicting the uniqueness of the Nash equilibrium described above). Therefore, 
    \begin{align}\label{eq:yes-s-a}
        \hat x_S^{A*} &= \frac{1}{2} \Big(I - \delta G_S - (\delta^2+\mu^2) G_{ST}G_{TS}\Big)^{-1}\cr
        &\quad \cdot\Big((\delta+\mu)G_{ST}M_T^+( p^{A}_T+p^B_T)\cr
        &\quad\quad + (\delta-\mu)G_{ST}M_T^-( p_T^{A} - p_T^{B}) \Big).
    \end{align}

The next step is to show that $\hat x^{B*} \ge \tilde x^{B*}:= \frac{1}{2} ((M^+ + M^-) p^B + (M^+ - M^-)p^A)$. Since we already know from the definition of $S$ that $\hat x^{B*}_S = \allzero \ge \frac{1}{2} ( (M^+ + M^-) p^B + (M^+ - M^-)p^A)_S = \tilde x^{B*}_S$, it suffices to show that  ${\hat x^{B*}_T  \ge \tilde x^{B*}_T}$. 

Now, $\tilde x^{B*}$ and $\tilde x^{A*}:=\frac{1}{2}((M^+ + M^-)p^A + (M^+ - M^-)p^B)$ satisfy~\eqref{eq:equilibrium}. According to~\cite[Theorem 4] {chen2018multiple}, this means that $(\tilde x^{A*},\tilde x^{B*})$ is  the unique (interior) Nash equilibrium of the game $\left([n],\R^n\times \R^n ,\{u_i\}_{i\in[n]} \right)$. Hence, it satisfies  the first derivative conditions $\frac{\partial u_i }{\partial x_i^A} \big\lvert_{(x^A,x^B)=(\tilde x^{A*},\tilde x^{B*})}=\frac{\partial u_i }{\partial x_i^B} \big\lvert_{(x^A,x^B)=(\tilde x^{A*},\tilde x^{B*})}=0$ for all $i\in [n]$. These conditions are equivalent to the following:
\begin{align*}
    p^{A} -  \tilde x^{A*} -\beta 
        \tilde x^{B*}
        + \delta G \tilde x^{A*} +\mu G \tilde x^{B*}
        &=\allzero,\cr
        p^{B} - \tilde x^{B*} -\beta  \tilde x^{A*}
        + \delta G \tilde x^{B*} +\mu G  \tilde x^{A*}
        &=\allzero.
\end{align*}
Rearranging the terms in these equations results in
\begin{align}
    \tilde x^{A*} = M(p^A + (\mu G - \beta I)\tilde x^{B*}),
\end{align}
where $M:=(I-\delta G)^{-1}$ exists and is non-negative by Lemma~\ref{lem:short}, and 
\begin{align}
    \tilde x^{B*} = p^B -\beta \tilde x^{A*} + \delta G \tilde x^{B*} + \mu G\tilde x^{A*}.
\end{align}

Similarly,  
we can show that the  conditions~\eqref{eq:big3} and ~\eqref{eq:big4} are equivalent to the following when put together with~\eqref{eq:just_once}:
\begin{align}\label{eq:notbig1}
    \hat x^{A*} &= M(p^A + (\mu G - \beta I)\hat x^{B*}),\\
    \label{eq:notbig2} \hat x^{B*}_T &= (\hat p^B -\beta \hat x^{A*} + \delta G \hat x^{B*} + \mu G\hat x^{A*})_T,\\
    \hat x^{B*}_S & = \allzero.
\end{align}

We now manipulate \eqref{eq:notbig2} as follows:
\begin{align*}
    &\hat x^{B*}_T\cr
    &= \left(\hat p^B + (\mu G - \beta I)\hat x^{A*} + \delta G\hat x^{B*}\right)_T \\
    &= \big(\hat p^B + (\mu G - \beta I)M(  p^A  + (\mu G - \beta I)\hat x^{B*}) \cr
    &\quad + \delta G \hat x^{B*}\big)_T \\
    &\stackrel{(a)}{=} \Big(\hat p^B + (\mu G - \beta I)M  p^A + \cr
    &\quad (\delta G + (\mu G - \beta I)^2 M ) \hat x^{B*}\Big)_T \\
    &= \left(\hat p^B + (\mu G -\beta I )M  p^A\right)_T \cr
    &\quad + (\delta G + (\mu G - \beta I)^2 M )_{T}\hat x^{B*}_T \cr
    &\quad + (\delta G + (\mu G - \beta I)^2 M )_{TS}\hat x^{B*}_S.  
\end{align*}
where $(a)$ follows from the commutativity of $M$ and $\mu G - \beta I$, both of which are either  polynomial functions of $G$ or linear combinations thereof. Hence,
\begin{align*}
   &(I - (\delta G + (\mu G - \beta I)^2 M )_T  ) \hat x^{B*}_T\cr
   &= \left(\hat p^B + (\mu G -\beta I )M  p^A\right)_T \cr
   &\quad + (\delta G + (\mu G - \beta I)^2 M )_{TS}\hat x^{B*}_S.  
\end{align*}
Here, we observe that any value of $\hat x^{B*}_T$ that satisfies the above equality (where $\hat x^{B*}_S=\allzero$) also satisfies~\eqref{eq:notbig2}, which, in combination with~\eqref{eq:notbig1} and $\hat x_S^{B*}=\allzero$, is  equivalent to~\eqref{eq:big3} - \eqref{eq:just_second}, which is in turn equivalent to \eqref{eq:aliter}.  In light of the uniqueness 
of $\hat x^{B*}$ (which follows from the uniqueness of the interior Nash equilibrium $(\hat x^{A*},\hat x^{B*})$ of the game $\left([n],\R^n\times \R^n ,\{\hat u_i\}_{i\in[n]} \right)$), this equivalence implies that $(I - (\delta G + (\mu G - \beta I)^2 M )_T  )$ is invertible and
$$
    \hat x^{B*}_T = (I - (\delta G + (\mu G - \beta I)^2 M )  )^{-1} \left(\left(\hat p^B + (\mu G -\beta I )M  p^A\right)_T + (\delta G + (\mu G - \beta I)^2 M )_{TS}\hat x^{B*}_S\right).
$$
Now, by Lemma~\ref{lem:spectral}, the matrix $(I - (\delta G + (\mu G - \beta I)^2 M )_T  )^{-1}$ has a Neumann series expansion $\sum_{k=0}^\infty (J_T)^k$, where $J:=\delta G + (\mu G - \beta I)^2 M$. Using this along with the above relation, we  obtain 
\begin{align} \label{eq:last_to_be_labeled}
    \hat x^{B*}_T = \sum_{k=0}^\infty J_T^k \left(\left(\hat p^B + (\mu G -\beta I )M  p^A\right)_T + J_{TS}\hat x^{B*}_S\right).
\end{align}
Similarly, we can show that
\begin{align*}
    \tilde x^{B*}_T = \sum_{k=0}^\infty J_T^k \left(\left( p^B + (\mu G -\beta I )M  p^A\right)_T + J_{TS}\tilde x^{B*}_S\right).
\end{align*}
Since $\hat p^B_T = p^B_T$, the above pair of equalities implies 
\begin{align}
    &\hat x^{B*}_T - \tilde x^{B*}_T \cr
    &= \sum_{k=0}^\infty J_T^k (\delta G + (\mu G - \beta I)^2 M)_{TS}(\hat x^{B*}_S -\tilde x^{B*}_S)\cr
    &=\sum_{k=0}^\infty J_T^k (\delta G + (\mu G - \beta I)^2 M)_{TS}|\tilde x^{B*}_S|\cr
    &\ge \allzero,
\end{align}
where the second equality holds because $\hat x_s^{B*} = \allzero > \tilde x_S^{B*}$.

    Now, we return to the setup described in the theorem statement. Suppose $S\subset\{i\in[n]: ((M^+ + M^-)p^B + (M^+ - M^-)p^A)_i<0 \}$ is a minimal set with the property that $\hat x^{B*}>\allzero$, where $\hat x^{A*}$ and $\hat x^{B*}$ are as defined by~\eqref{eq:aliter}. By the definition of $S$, we then have $((M^+ + M^-)p^B + (M^+ - M^-)p^A)_{s}<0$ for all $s\in S$, and hence, 
    \begin{align}\label{eq:long_min}
        &\min_{s\in S}\{((M^+ + M^-)p^B\cr
        &\quad \quad + (M^+ - M^-)(p^{A0} + 1\cdot(p^A - p^{A0})))_{s}\}<0.
    \end{align}
    On the other hand, since $p^{A0}\le \plim$, we have
    \begin{align}\label{eq:plim}
        &(M^+ + M^-)p^B + (M^+ - M^-)p^{A0} \cr
        &\stackrel{(a)}\ge (M^+ + M^-)p^B + (M^+ - M^-)\plim\cr
        &\stackrel{(b)}= (M^+ + M^-)p^B \cr
        &\quad + (M^+ - M^-)(M^- - M^+)^{-1}(M^+ + M^-)p^B\cr
        & = (M^+ + M^-)p^B - I(M^+ + M^-)p^B\cr
        &=\allzero,
    \end{align}
    for all $i\in [n]$, where $(a)$ holds because $M^+ - M^-=-M_\Delta$ is a negative matrix (as established in Theorem~\ref{thm:one}), and $(b)$ follows from the definition of $\plim$. This implies
    \begin{align}\label{eq:2nd_long_min}
        &\min_{s\in S}\{((M^+ + M^-)p^B\cr
        &\quad\quad+ (M^+ - M^-)(p^{A0} + 0\cdot(p^A - p^{A0})))_{s}\}\ge 0.
    \end{align}
    By the Intermediate Value Theorem,~\eqref{eq:long_min} and~\eqref{eq:2nd_long_min} imply the existence of an $\alpha_1\in [0,1]$ for which $\min_{s\in S}\{((M^+ + M^-)p^B + (M^+ - M^-)(p^{A0} + \alpha_1(p^A - p^{A0})))_{s}\}= 0$. That is, under the pricing policy $p^{A1}:=p^{A0} + \alpha_1(p^A - p^{A0})$ (which satisfies $p^{A0}\le p^{A1}\le p^A$ due to $\alpha_1\in[0,1]$), the equilibrium $(x^{A1}, x^{B1})$ of the game $\left([n],\R^n\times \R^n ,\{ u_i^1 \}_{i\in[n]} \right)$ satisfies $x^{B1}_{s_1}=0$, where $u^1_i(x^A,x^B):= u_1(x^A,x^B; p^{A1},p^B)$ and $s_1:=\argmin_{s\in S}\{((M^+ + M^-)p^B + (M^+ - M^-)(p^{A0} + \alpha_1(p^A - p^{A0})))_{s}\}$. Thus, $(x^{A1},x^{B1})$ is also a Nash equilibrium of the hypothetical game that constrains the effort level of the $s_1$-th agent in activity $B$ to be equal to 0 (while retaining the  utility structure~\eqref{eq:main}). It now follows from Lemma~\ref{lem:recast} that $(x^{A1},x^{B1})$ satisfies~\eqref{eq:unnamed} with $p^{A1}$ replacing $p^A$ and $\tilde S = \{s_1\}$. Moreover, since it is the Nash equilibrium of the game $\left([n],\R^n\times \R^n ,\{ u_i^1\}_{i\in[n]} \right)$, it is also true that $\frac{\partial u_i^1}{\partial x_i^B}\bigg\lvert_{(x^A,x^B)= (x^{A1},x^{B1})}= 0$ for all $i\in [n]$. 
    
    Now, the definitions of $p^{A1}$ and $s_1$ imply that $x^{B1}_s = ((M^+ + M^-)p^B + (M^+ - M^-)p^{A1} )_{s}\ge 0$ for all $s\in S\setminus\{s_1\}$. By  Lemma~\ref{lem:recast}, this is equivalent to
    \begin{align}\label{eq:recently_labeled}
        &\min_{s\in S\setminus \{s_1\}}\bigg\{\bigg(\bigg(M_{\tilde{T}}^{+}+M_{\tilde{T}}^{-}\bigg)\bigg(p_{\tilde{T}}^B+\mu G_{\tilde{T}} {x}_{\tilde{S}}^{A}(p^{A1}) \bigg)\cr
        &\quad \quad \quad +\bigg(M_{\tilde{T}}^{+}-M_{\tilde{T}}^{-}\bigg)\bigg( p_{\tilde{T}}^{A1}+\delta G_{\tilde{T} \tilde{S}} {x}_{\tilde{S}}^{A}(p^{A1})\bigg)\bigg)_s\bigg\} \cr
        &\ge 0,
    \end{align}
    where $x_{\tilde S}^{A}(p^{A1})$ is given by 
    $$
        x_{\tilde S}^{A1}  = \frac{1}{2} \left(I - \delta G_{\tilde S} - (\delta^2+\mu^2) G_{\tilde S \tilde T}G_{\tilde T \tilde S}\right)^{-1}\left((\delta+\mu)G_{\tilde S \tilde T}M_{\tilde T}^+( p^{A1}_{\tilde T}+p^B_{\tilde T}) + (\delta-\mu)G_{\tilde S \tilde T}M_{\tilde T}^-( p_{\tilde T}^{A1} - p_{\tilde T}^{B}) \right),
    $$
   where $\tilde T:=[n]\setminus\tilde S = [n]\setminus \{s_1\}$. On the other hand, the minimality of $S$ implies that there exists an $s\in \tilde T$ such that $\hat x^{B}_s(\tilde S)<0$. 

   Here, two cases arise: either $s\in S$ or $s\in T$. In the first case, we have the following as a consequence of Lemma~\ref{lem:recast}:
\begin{align}\label{eq:november_21}
    &\min_{s'\in S\setminus \{s_1\}}\bigg\{\bigg(\bigg(M_{\tilde{T}}^{+}+M_{\tilde{T}}^{-}\bigg)\bigg(p_{\tilde{T}}^B+\mu G_{\tilde{T}} {x}_{\tilde{S}}^{A}(p^{A}) \bigg)\cr
    &+\bigg(M_{\tilde{T}}^{+}-M_{\tilde{T}}^{-}\bigg)\bigg( p_{\tilde{T}}^{A}+\delta G_{\tilde{T} \tilde{S}} {x}_{\tilde{S}}^{A}(p^{A})\bigg)\bigg)_{s'} \bigg\}< 0.
\end{align}
    In the second case, we have $\hat x^{B*}_T(\tilde S)\ngeq \allzero$ by the definition of $s$. We also have  the following relation, whose proof  is identical to the derivation of~\eqref{eq:last_to_be_labeled}, except that it replaces $(\hat x^{A*}(S),\hat x^{B*}(S))$ with $(\hat x^{A*}(\tilde S),  \hat x^{B*}(\tilde S) )$ in the derivation:
    $$
        \hat x^{B}_T(\tilde S) = \sum_{k=0}^\infty J_T^k \left(\left(\hat p^B + (\mu G -\beta I )M  p^A\right)_T + J_{TS}\hat x^{B}_S (\tilde S) \right).
    $$
    Recall that the matrix $J=\delta G + (\mu G -\beta I)^2M$ is non-negative according to Lemma~\ref{lem:long_matrix}. Now, suppose that $\hat x_s^{B*}(\tilde S)\ge \hat x_s^{B*}(S)$, then,~\eqref{eq:last_to_be_labeled} and the above relation would together imply that $\hat x^{B*}_T(\tilde S)\ge \hat x^{B*}_T(S) \ge \allzero$  (where the second inequality follows from the definition of $S$ provided in the theorem statement). However, this would contradict the preceding observation $\hat x^{B*}_T(\tilde S) \ngeq \allzero$.  Thus, our supposition is wrong, i.e., $\hat x_s^{B*}(\tilde S)\ngeq \hat x_s^{B*}(S)$. In other words, there exists an $s''\in S$ such that $\hat x^{B*}_{s''}(\tilde S)< \hat x^{B*}_{s''}(S)=0$ ( where the equality follows from~\eqref{eq:just_once}). Invoking Lemma~\ref{lem:recast} now shows that~\eqref{eq:november_21} holds true in this (second) case as well.
    
    By the Intermediate Value Theorem, the inequalities~\eqref{eq:recently_labeled} and~\eqref{eq:november_21} imply the existence of an $\alpha_2\in[0,1]$ such that
    \begin{align}\label{eq:pA2}
        &\min_{s\in S\setminus \{s_1\}}\bigg\{\bigg(\bigg(M_{\tilde{T}}^{+}+M_{\tilde{T}}^{-}\bigg)\bigg(p_{\tilde{T}}^B+\mu G_{\tilde{T}} {x}_{\tilde{S}}^{A}(p^{A2}) \bigg)\cr
        &\quad\quad+\bigg(M_{\tilde{T}}^{+}-M_{\tilde{T}}^{-}\bigg)\bigg( p_{\tilde{T}}^{A2}+\delta G_{\tilde{T} \tilde{S}} {x}_{\tilde{S}}^{A}(p^{A2})\bigg)\bigg)_s\bigg\} = 0,
    \end{align}    
    where $p^{A2}= p^{A1} + \alpha_2(p^A - p^{A1})$. Note that $p^{A2}$ satisfies $p^{A1}\le p^{A2}\le p^A$ because $\alpha_2\in[0,1]$ and $p^A\ge p^{A1}$. Let us now define $(x^{A2},x^{B2})$ as the Nash equilibrium of the game that enforces $x_{\tilde S}^{B2} = x_{s_1}^{B2}=0$ while retaining the utility structure~\eqref{eq:main}. Since $p^{A2}\ge p^{A1}$, using Lemma~\ref{lem:(1)} after setting $\check p^A = p^{A2}$, $\tilde S = \{s_1\}$, and $p^{A(1)}= p^{A1}$ implies that $\frac{\partial u_i}{\partial x_i^B}\bigg\lvert_{(x^{A2},x^{B2})}\le 0$ for all $i\in [n]$ under the policy $p^{A2}$.

    Besides, we know from Lemma~\ref{lem:recast} that $(x^{A2},x^{B2})$ satisfy~\eqref{eq:similar_first} and~\eqref{eq:similar-s-a} with $p^{A(1)}=p^{A2}$. Also, we know from~\eqref{eq:pA2} that $x^{B2}_{s_2} = 0$, where $$s_2:=\argmin_{s\in S\setminus \{s_1\}}\left\{\left(\left(M_{\tilde{T}}^{+}+M_{\tilde{T}}^{-}\right)\left(p_{\tilde{T}}^B+\mu G_{\tilde{T}} {x}_{\tilde{S}}^{A}(p^{A2}) \right)+\left(M_{\tilde{T}}^{+}-M_{\tilde{T}}^{-}\right)\left( p_{\tilde{T}}^{A2}+\delta G_{\tilde{T} \tilde{S}} {x}_{\tilde{S}}^{A}(p^{A2})\right)\right)_s\right\}$$. Thus, $x^{B2}_{\{s_1,s_2\}}=\allzero$, which means that $(x^{A2},x^{B2})$ is also a Nash equilibrium of the hypothetical game that enforces the effort levels of both $s_1$-th and $s_2$-th agents  to be zero in activity $B$ (while retaining the  utility structure~\eqref{eq:main}). Hence, it satisfies~\eqref{eq:similar_first} and~\eqref{eq:similar-s-a} with $\tilde S = \{s_1,s_2\}$.

We can repeat the arguments of the preceding paragraphs $r-2$ times to show that there exists a pricing policy $p^{Ar}$ such that $p^{A0}\le p^{Ar}\le p^A$, and $(x^{Ar}, x^{Br})$, as defined by~\eqref{eq:similar_first} and~\eqref{eq:similar-s-a} for $\tilde S = \{s_1,s_2,\ldots, s_r\}=S$, satisfies $\frac{\partial u_i}{\partial x_i^B}\bigg\lvert_{ ( x^{Ar},  x^{Br} )} \le 0$ for all $i\in [n]$.  It now follows from Lemma~\ref{lem:(1)} that under $\check p^A = p^A$,  $(\check x^A, \check x^B)$ is a Nash equilibrium of the hypothetical game with utility structure~\eqref{eq:main} and non-negativity constraints on the effort level of the agents indexed by $\tilde S = \{s_1,s_2,\ldots, s_r\}=S$. As $\check p^A = p^A$, this means that $(\hat x^{A*}, \hat x^{B*}) = (\check x^A, \check x^B)$ is a Nash equilibrium of the game that imposes non-negativity constraints on the effort levels of the agents indexed by $S$. As $\hat x^{B*} \ge \allzero$ by the definition of $S$, this also means that $(\hat x^{A*},\hat x^{B*})$ is a Nash equilibrium of the game with utility structure~\eqref{eq:main} that imposes non-negativity constraints on the effort levels of all the agents in activity $B$. 
    
    Finally, we note that the expressions for $x^{A(1)}$ as given by~\eqref{eq:similar_first} and~\eqref{eq:similar-s-a} are increasing functions of $p^{A(1)}$ for all $\tilde S\subset [n]$. Since $p^{A0}\le p^{A1}\le p^{A2}\le\cdots\le p^{Ar}\le p^A$, this means that $\hat x^{A*}\ge x^{Ar}\ge \cdots \ge x^{A0} > \allzero$, where the last inequality holds by assumption. Therefore, $\hat x^{A*}$, like $\hat x^{B*}$, is non-negative. This shows that  $(\hat x^{A*},\hat x^{B*})$ is a Nash equilibrium of the game with utility structure~\eqref{eq:main} that imposes non-negativity constraints on all the effort levels. Furthermore, if $\left(\frac{\delta+\mu}{1-\beta}\right)\rho(G)<1$, then this equilibrium is unique, as implied by~\cite[Equation (21)]{parise2019variational}.
    \item [(ii)] Observe that under the pricing policy $\plim$, the equilibrium $(x^{A\lim}, x^{B\lim} )$ satisfies
    $$
        {x^{B\lim} = \frac{1}{2}((M^+ + M^-)p^B + (M^+ - M^-)\plim )= \allzero}
    $$
    by~\eqref{eq:equilibrium} and~\eqref{eq:plim}. Therefore, we have
    \begin{align}\label{eq:tomake0}
        0 &= \frac{\partial u_i }{\partial x_i^B}\bigg\lvert_{(x^{A\lim},x^{B\lim})}\cr
        &= (p^B + (\mu G - \beta I)M\plim)_i + 0,
    \end{align}
    where the second equality follows from~\eqref{eq:near_end}. On the other hand, under the pricing policy $p^A\ge \plim$, we can verify that $x^A = (I-\delta G)^{-1}p^{A}$ and $x^B = \allzero$ satisfy $\frac{\partial u_i }{\partial x_i^A}\bigg\lvert_{(x^{A},x^{B})}=0$ as the utility structure effectively reduces to that of a single-activity game. Moreover, in light of~\eqref{eq:near_end},  $x^B=\allzero$ also implies that $\frac{\partial u_i }{\partial x_i^B}\bigg\lvert_{(x^{A},x^{B})} = (p^B + (\mu G - \beta I)Mp^A)_i$ for all $i\in [n]$. Consequently, we have
    \begin{align*}
        &\frac{\partial u_i }{\partial x_i^B}\bigg\lvert_{(x^{A},x^{B})} \cr
        &= (p^B + (\mu G - \beta I)Mp^A)_i - 0\cr
        &\stackrel{(a)}= (p^B + (\mu G - \beta I)Mp^A)_i \cr
        &\quad - (p^B + (\mu G - \beta I)M\plim )_i\cr
        &= ((\mu G - \beta I)M (p^A - \plim))_i\cr
        &\stackrel{(b)}\le 0,
    \end{align*}
    where $(a)$ holds because of~\eqref{eq:tomake0} and $(b)$ holds because $p^A\ge \plim$ and because  $(\mu G-\beta I) M = (\mu G - \beta I )(I-\delta G)^{-1}$ is a non-positive matrix (as established in Lemma~\ref{lem:short}). As the utility functions are strongly concave in all the $2n$ effort levels, $\frac{\partial u_i }{\partial x_i^B}\bigg\lvert_{(x^{A},x^{B})}\le 0$ along with $\frac{\partial u_i }{\partial x_i^A}\bigg\lvert_{(x^{A},x^{B})}=0$ together imply that $(x^A,x^B)$ is the Nash equilibrium of the game with utility structure~\eqref{eq:main} and non-negative effort levels in activity $B$. Moreover, since $(I-\delta G)^{-1}$ is non-negative according to Lemma~\ref{lem:short}, $x^A$ is a non-negative vector. Therefore, $(x^A,x^B)$ is also the Nash equilibrium of the game with utility structure~\eqref{eq:main} and non-negative effort levels in both the activities.

    To prove the uniqueness of this equilibrium, suppose there exists a Nash equilibrium $(x^{A'},x^{B'})\ne (x^A, x^B)$ under the given policy $p^A\ge \plim$. If $x^{B'}= x^B= \allzero$, then the game reduces to a single-activity game in which only activity $A$ is performed, in which case we know from the first-order conditions for equilibrium that $(I-\delta G)^{-1} p^A=x^A$ is the unique equilibrium effort level, thereby contradicting the existence of a distinct equilibrium $x^{A'}$. On the other hand, if $x^{B'}\ne \allzero$, then there exists a proper subset $\tilde S$ of $[n]$ for which $x^{B'}_{\tilde S}=\allzero$, in which case it follows from Lemma~\eqref{lem:recast} and first-order equilibrium conditions that $(x^{A'},x^{B'})$ satisfies~\eqref{eq:unnamed} with $(\hat x^{A*},\hat x^{B*})= (x^{A'},x^{B'})$. However, \eqref{eq:unnamed} is also satisfied by $(x^A,x^B)$, since $(x^A,x^B)$ is also a Nash equilibrium of the game that has utility structure~\eqref{eq:main} and imposes non-negativity constraints only on the effort levels of the agents indexed by $\tilde S$ in activity $B$. We have shown earlier that the system of equations~\eqref{eq:unnamed} is equivalent to the system~\eqref{eq:aliter}, which has a unique solution. It follows that $(x^{A'},x^{B'})=(x^A,x^B)$, which again contradicts our assumption that the two equilibria are distinct. This proves that the concerned equilibrium is unique. 
\end{enumerate} 
\end{proof}

\section{Proof of Proposition~\ref{prop:one}}
\begin{proof} As $\max_{i\in[n]}\rho_i>0$, there exists $i\in [n]$ such that $\rho_i>0$. We can verify that the policy $(p^A,p^B):= (p^{A0} + \rho_i e^{(i)}, p^{B0} - \rho_i e^{(i)})$ satisfies the constraints in~\eqref{eq:redistribution} for a large enough $\tau^B$ (precisely, for any $\tau^B$ that satisfies $\tau^B \ge \max \{\allone^\top\left( (M^+ + M^-)p^B + (M^+ - M^-)p^A \right)  : p^{A0} \le p^A \le p^{A0}+b+\rhomax, \allzero\le p^B\le p^{B0}\}$, where the maximum is finite as it is computed over a compact set). Hence, $(p^A,p^B)$ is a feasible solution of $\pbr$ for large enough values of $\tau^B$.  

 Now, let $(x^{A*},x^{B*})$ denote the post-intervention equilibrium corresponding to the policy $(p^A,p^B)$, and observe that
    \begin{align*}
        \sum_{i=1}^n x_i^{B^*} &= \allone^\top\left( (M^+ + M^-)p^B + (M^+ - M^-)p^A \right) \cr
        &\stackrel{(a)}=  \allone^\top\left( (M^+ + M^-)p^{B0} + (M^+ - M^-)p^{A0} \right)\cr
        &\quad- 2\rho_i \allone^\top M^-e^{(i)}\cr
        &\stackrel{(b)}< \allone^\top\left( (M^+ + M^-)p^{B0} + (M^+ - M^-)p^{A0} \right)\cr
        &= \sum_{i=1}^n x_i^{B0}.
    \end{align*}
    where $(a)$ follows from our choice $(p^A,p^B)= (p^{A0} + \rho_i e^{(i)}, p^{B0} - \rho_i e^{(i)})$ and the statement $(b)$ holds for two reasons. First, $\rho_i > 0$. Second, the matrix 
\[
M^- = \sum_{k=0}^\infty \frac{(\delta - \mu)^k}{(1 - \beta)^{k+1}} G^k
\] 
is positive. This positivity follows from Assumption~\ref{assum:mu-delta} and the irreducibility of $G$, which guarantees that for every $(i, j) \in [n] \times [n]$, there exists a power $k$ such that $(G^k)_{ij} > 0$ (for further details, see the proof of Theorem~\ref{thm:one}).
We have thus shown that $\sum_{i=1}^n x_i^{B*}< \sum_{i=1}^n x_i^{B0}$, which completes the proof.
\end{proof}

\section{Proof of Lemma~\ref{lem:positive-definite}}

\begin{proof} We prove the assertions one by one.
\begin{enumerate}
    \item [(i)] Recall from the proof of Proposition~\ref{prop:one} that $M^-$ is a positive matrix. Since we can also express $M^+$ using a Neumann series expansion ($M^+=\sum_{k=1}^\infty \frac{(\delta+\mu)^k}{(1+\beta)^{k+1}}G^k$), we can use similar arguments to show that $M^+$ is also a positive matrix. It now follows from $\beta\in(0,1)$ that $Q = (1-\beta)(M^-)^2 + (1+\beta)(M^+)^2$ is positive.

As for positive-definiteness, note that $M^+$ and $M^-$ are both positive-definite because they are Leontief matrices. Therefore, squaring these matrices (which results in squaring their eigenvalues)  also yields positive-definite matrices. This implies that $Q$ is a weighted sum of two positive-definite matrices. As these weights are non-negative (since $\beta\in (0,1)$), it follows that $Q$ is positive-definite.
\item [(ii)] We first observe that
$$
    (1-\beta)M^- \stackrel{(a)}= \sum_{k=0}^\infty \left(\frac{\delta - \mu}{1-\beta} \right)^k G^k \stackrel{(b)}\ge \sum_{k=0}^\infty \left(\frac{\delta + \mu }{1+\beta} \right)^k G^k\stackrel{(c)}=(1+\beta)M^+,
$$
where $(a)$ and $(c)$ follow from the Neumann series expansions of $M^-$ and $M^+$, respectively, and $(b)$ follows from~\eqref{eq:handy} and the non-negativity of $G$. We have thus shown that 
\begin{align}\label{eq:convenient}
    (1-\beta) M^- - (1+\beta)M^+ \ge O,
\end{align}
where $O\in \R^{n\times n}$ is the matrix in which all entries equal  zero. Using this, we observe the following:
\begin{align*}
    R&= (1-\beta)(M^-)^2 - (1+\beta)(M^+)^2\cr
    &\stackrel{(a)}= (1-\beta)(M_\Delta + M^+)(M^-) - (1+\beta)(M^+)^2\cr
    &= (1-\beta)M_\Delta M^- \cr
    &\quad + (1-\beta)M^+ M^- - (1+\beta)(M^+)^2\cr
    &= (1-\beta) M_\Delta M^- \cr
    &\quad + M^+\left((1-\beta)M^- - (1-\beta)M^+ \right)\cr
    &\stackrel{(b)}> O,
\end{align*}
where $(a)$ follows from the definition of $M_{\Delta}$, and $(b)$ holds because of~\eqref{eq:convenient} and because $M^+$ is positive (as established in the proof of Lemma~\ref{lem:positive-definite}-(i)), $M^-$ is positive (as established in the proof of Proposition~\ref{prop:one}), and $M_\Delta$ is positive   under Assumption~\ref{assum:domination} (as established in Theorem~\ref{thm:one}-(i)). 

We have thus shown that $R$ is positive. It now follows that $v = 2Rp^{B0}$ is a positive vector.
\item [(iii)] Note that $Q-R = 2(1+\beta)(M^+)^2$, the right-hand side of which  is a positive matrix because $M^+$ is positive (as established in the proof of Lemma~\ref{lem:positive-definite}-(i)). It follows that $Q\ge R$. 
\end{enumerate}
\end{proof}

\section{Proof of Theorem~\ref{thm:general_case}}

\begin{proof} 
Let $p^{A(1)}$ and $p^{A(2)}$ be given to satisfy $p^{A0}\le p^{A(2)}\le p^{A(1)}\le \plim$. 

Now, suppose $Q(p^{A(1)}+ p^{A(2)})\ge v$. We then have
\begin{align*}
(p^{A(1)}- p^{A(2)})^\top Q(p^{A(1)}+ p^{A(2)}- v)\ge 0
\end{align*}
as the left-hand side is the inner product of  two non-negative vectors. Equivalently, 
\begin{align}\label{eq:simple}
&(p^{A(1)}- p^{A(2)})^\top Q(p^{A(1)}+ p^{A(2)})\cr
&\ge  v^\top (p^{A(1)}-p^{A(2)}).
\end{align}
Expanding the left-hand side of this inequality yields
\begin{align}\label{eq:over_explained}
    &(p^{A(1)}- p^{A(2)})^\top(Q(p^{A(1)}+ p^{A(2)})\cr
    &= \left(p^{A(1)}\right)^\top Q p^{A(1)} \cr
    &\quad + \left(p^{A(1)}\right)^\top Q p^{A(2)} - \left(p^{A(2)}\right)^\top Q p^{A(1)}\cr
    &\quad + \left(p^{A(2)}\right)^\top Q p^{A(2)}\cr
    &= \left(p^{A(1)}\right)^\top Q p^{A(1)} + \left(p^{A(2}\right)^\top Q p^{A(2)},
\end{align}
where the last step holds because of the following scalar subtraction:
\begin{align}
    &\left(p^{A(1)}\right)^\top Q p^{A(2)} - \left(p^{A(2)}\right)^\top Q p^{A(1)} \cr
    &= \left(\left(p^{A(1)}\right)^\top Q p^{A(2)}\right)^\top  - \left(p^{A(2)}\right)^\top Q p^{A(1)}\cr
    &= \left(p^{A(2)}\right)^\top Q^\top p^{A(1)} - \left(p^{A(2)}\right)^\top Q p^{A(1)}\cr
    &\stackrel{(a)}=0,
\end{align}
where $(a)$ holds because  $Q$ is symmetric (which is in turn true because $G$ is symmetric). Combining~\eqref{eq:simple} and~\eqref{eq:over_explained} now gives
$
    \left(p^{A(1)}\right)^\top Q p^{A(1)} + \left(p^{A(2}\right)^\top Q p^{A(2)} \ge v^\top \left(p^{A(1)} - p^{A(2)}\right),
$
rearranging which yields
$$
    \left(p^{A(1)}\right)^\top Q p^{A(1)} - v^\top p^{A(1) } \ge \left(p^{A(2)}\right)^\top Q p^{A(2)} - v^\top p^{A(2) },
$$
or equivalently, $\phi(p^{A(1)})\ge \phi(p^{A(2)})$. 

By reversing the signs of all the inequalities above, we can repeat the above  sequence of arguments to show that $Q(p^{A(1)} + p^{A(2)})<v$ implies $\phi(p^{A(1)})<\phi(p^{A(2)})$.

To show that $\pmax$ leads to minimal aggregate unsustainable effort, we note that the following holds for all policies $p^A$ satisfying $p^{A0}\le p^A \le \pmax$:
\begin{align*}
    &\allone^\top \left((M^+ + M^-)p^{B0} + (M^+ - M^-)\pmax \right)\cr
    &=\allone^\top \left((M^+ + M^-)p^{B0} -M_\Delta \pmax \right)\cr
    &\stackrel{(a)}\le \allone^\top \left((M^+ + M^-)p^{B0} -M_\Delta p^A \right)\cr
    &= \allone^\top \left((M^+ + M^-)p^{B0} + (M^+ - M^-)p^A \right),
\end{align*}
where $(a)$ holds because $M_\Delta$ is positive under Assumption~\ref{assum:domination} (as established in Theorem~\ref{thm:one}-(i)). According to \eqref{eq:equilibrium} and Theorem~\ref{thm:non-negative}-(i), this implies that the aggregate unsustainable effort in the post-intervention equilibrium corresponding to $\pmax$ is no greater than its value in  the post-intervention equilibrium corresponding to $p^A$. In other words, $\pmax$ minimizes the post-intervention value of aggregate unsustainable effort over all feasible policies. Similarly, we can use $p^A\ge p^{A0}$ to show that $p^{A0}$ maximizes the post-intervention value of aggregate unsustainable effort over all feasible policies. 
\end{proof}

\section{Proof of Lemma~\ref{lem:interpret}}
\begin{proof} By expressing $Q$ and $R$ in terms of $M^+$ and $M^-$ (using the definitions of $Q$ and $R$), we can show that
$
    |Qp^A - R p^{B0}| = \left|(1+\beta)(M^+)^2 (p^A+p^{B0}) + (1-\beta)(M^-)^2(p^A - p^{B0}) \right|= \psi(p^A)
$ for all $p^A\in \R^n_{\ge 0}$. Therefore, it suffices to show that the condition $Q(p^{A(1)} + p^{A(2)})\ge v$ (respectively, $Q(p^{A(1)} + p^{A(2)})< v$) is equivalent to $|Qp^{A(1)}- Rp^{B0}|\ge|Qp^{A(2)}- Rp^{B0}|$ (respectively, $|Qp^{A(1)}- Rp^{B0}|<|Qp^{A(2)}- Rp^{B0}|$).

To this end, observe that $Q(p^{A(1)} + p^{A(2)})\ge v$ is equivalent to the following inequality, which we obtain on substituting $v=2Rp^{B0}$:
\begin{align}\label{eq:skewed_ineq}
    Qp^{A(1)} - Rp^{B0} \ge Rp^{B0} - Qp^{A(2)}.
\end{align}
Hence, it suffices to show that~\eqref{eq:skewed_ineq} is equivalent to 
\begin{align}\label{eq:mod_skewed}
    |Qp^{A(1)}- Rp^{B0}|\ge|Qp^{A(2)}- Rp^{B0}|.
\end{align}
We first show that the former implies the latter. For this purpose, note that the non-negativity of $Q$ (established in Lemma~\ref{lem:positive-definite}-(i)) and our assumption  $p^{A(1)}\ge p^{A(2)}$ together imply that 
$$
    Qp^{A(1)} - Rp^{B0} \ge Qp^{A(2)} - Rp^{B0}.
$$
Combining this with \eqref{eq:skewed_ineq} gives
$$
    Qp^{A(1)} - Rp^{B0} \ge |Rp^{B0} - Qp^{A(2)}|,
$$
which implies~\eqref{eq:mod_skewed}. Thus,~\eqref{eq:skewed_ineq} implies~\eqref{eq:mod_skewed}.

Now, suppose we are given that~\eqref{eq:mod_skewed} holds. To this end, suppose $(Qp^{A(1)} - Rp^{B0})_i<0$ for some $i\in [n]$. Then 
$$
    (Rp^{B0})_i - (Qp^{A(1)})_i =|(Qp^{A(1)} - Rp^{B0})_i| \ge |(Rp^{B0})_i - (Qp^{A(2)})_i| \ge  (Rp^{B0})_i - Qp^{A(2)})_i,
$$
where the first inequality is due to~\eqref{eq:mod_skewed}. On subtracting $(Rp^{B0})_i$ from all sides, this implies $(Qp^{A(1)})_i\le Q(p^{A(2)})_i$. On the other hand, the non-negativity of $Q$ along with $p^{A(1)}\ge p^{A(2)}$ implies that $Qp^{A(1)}\ge Qp^{A(2)}$. This is possible only if $Qp^{A(1)}=Qp^{A(2)}$, which implies $p^{A(1)}=p^{A(2)}$ by the invertibility of $Q$ (which is implied by the positive-definiteness of $Q$). This contradicts our assumption that $p^{A(1)} $ and $p^{A(2)}$ are distinct, thereby proving that our supposition that $(Qp^{A(1)} - Rp^{B0})_i<0$ was wrong. This shows that $Qp^{A(1)}- Rp^{B0}\ge \allzero$. Therefore, we have
$$
    Qp^{A(1)}- Rp^{B0} = |Qp^{A(1)}- Rp^{B0}| \ge |Rp^{B0}-Qp^{A(2)}|\ge Rp^{B0}-Qp^{A(2)}.
$$
We have thus shown that \eqref{eq:mod_skewed} implies \eqref{eq:skewed_ineq}. 

In sum, we have so far shown that \eqref{eq:mod_skewed} and \eqref{eq:skewed_ineq} are equivalent. Since we have proved the equivalence of \eqref{eq:skewed_ineq} and $Q(p^{A(1)} + p^{A(2)})\ge v$ earlier in the proof, we have effectively shown that \eqref{eq:mod_skewed}  and  $Q(p^{A(1)} + p^{A(2)})\ge v$ are equivalent, as required. 

Next, we show that the conditions $Q(p^{A(1)}+p^{A(2)})>v$ and $\psi(p^{A(1)})>\psi(p^{A(2)})$ are equivalent. To this end, observe that $Q(p^{A(1)} + p^{A(2)})< v$ is equivalent to the following inequality, which we obtain on substituting $v=2Rp^{B0}$:
\begin{align}\label{eq:skewed_ineq}
    Qp^{A(1)} - Rp^{B0}< Rp^{B0} - Qp^{A(2)}.
\end{align} 

We now show that $|Qp^{A(1)} - Rp^{B0}|<|Rp^{B0} - Qp^{A(2)}|$ and $Qp^{A(1)} - Rp^{B0}<Rp^{B0} - Qp^{A(2)}$ are equivalent. 

We first show that
\begin{align}\label{eq:2nd_mod_skewed}
    |Qp^{A(1)} - Rp^{B0}|<|Rp^{B0} - Qp^{A(2)}|
\end{align}
implies
\begin{align}\label{eq:2nd_skewed_inew}
    Qp^{A(1)} - Rp^{B0}<Rp^{B0} - Qp^{A(2)}.
\end{align}
Suppose the claim~\eqref{eq:indeed} below is false.
\begin{align}\label{eq:indeed}
    Rp^{B0} - Qp^{A(2)}\ge \allzero.
\end{align}
 Then there exists an $i\in [n]$ such that $(Qp^{A(2)} - Rp^{B0})_i >0$. As a result, we have
 $$
    (Qp^{A(2)})_i - (Rp^{B0})_i = |(Rp^{B0} - Qp^{A(2)})_i |> |Qp^{A(1)} - Rp^{B0}| \ge (Qp^{A(1)})_i - (Rp^{B0})_i,
 $$
 where the first inequality above follows from~\eqref{eq:2nd_mod_skewed}. However, this implies that $(Qp^{A(2})_i> (Qp^{A(1)})_i$, which contradicts our preceding observation that $Qp^{A(1)}\ge Qp^{A(2)}$. Hence, our supposition was false, thereby proving that~\eqref{eq:indeed} holds true given \eqref{eq:2nd_mod_skewed}. Consequently, if \eqref{eq:2nd_mod_skewed} holds, then we have
 $$
    Rp^{B0} - Qp^{A(2)} \stackrel{(a)}= |Rp^{B0} - Qp^{A(2)}| > | Qp^{A(1)} - Rp^{B0}| \ge  Qp^{A(1)} - Rp^{B0},
 $$
 where $(a)$ follows from~\eqref{eq:indeed} and the inequality follows from \eqref{eq:2nd_mod_skewed}. Thus,~\eqref{eq:2nd_mod_skewed} implies \eqref{eq:2nd_skewed_inew}. 

 To show the reverse, suppose \eqref{eq:2nd_skewed_inew} is given to be true. Note that $Qp^{A(1)}\ge Qp^{A(2)}$ implies
 \begin{align}\label{eq:helpful}
     Rp^{B0} - Qp^{A(1)}\le Rp^{B0} - Qp^{A(2)}.
 \end{align}
 Combining \eqref{eq:2nd_skewed_inew} and \eqref{eq:helpful} gives
 \begin{align}\label{eq:help}
     |Qp^{A(1)} - Rp^{B0}|\le Rp^{B0} - Qp^{A(2)}.
 \end{align}
 Therefore, we have
 $$
    |Rp^{B0} - Qp^{A(2)}|\ge Rp^{B0} - Qp^{A(1)} \stackrel{(a)}\ge |Qp^{A(1)} - Rp^{B0}|,
 $$
 where $(a)$ is a consequence of \eqref{eq:help}. Thus, \eqref{eq:2nd_skewed_inew} implies \eqref{eq:2nd_mod_skewed}.
 
 In sum, we have so far shown that \eqref{eq:2nd_mod_skewed} and \eqref{eq:2nd_skewed_inew} are equivalent. Since we have proved the equivalence of \eqref{eq:2nd_skewed_inew} and $Q(p^{A(1)} + p^{A(2)})< v$ earlier in the proof, we have effectively shown that \eqref{eq:2nd_mod_skewed}  and  $Q(p^{A(1)} + p^{A(2)})< v$ are equivalent, as required. This completes the proof. 
\end{proof}

\section{Proof of Corollary~\ref{cor:first}}
We prove the assertions one by one.
\begin{enumerate}
    \item [(i)] Since all feasible policies $p^A$ satisfy $p^A\le \pmax$, it is sufficient to show that $\pmax$ is no less welfare-generating than any policy $p$ that satisfies $p^{A0}\le p\le \pmax$. By Theorem~\ref{thm:general_case}, this means that it suffices to prove that each of the conditions in Corollary \ref{cor:first}-(i)-(a) -  Corollary \ref{cor:first}-(i)-(c) implies \eqref{eq:thm_condition_1} for the choice $p^{A(1)}=\pmax$ and $p^{A(2)}=p$. We proceed as follows.
    \begin{enumerate}
        \item [(a)] We have 
        $$
            Q(\pmax+ p)\stackrel{(a)}\ge  2Q  p^{A0} \stackrel{(b)}\ge 2Q p^{B0} \stackrel{(c)} \ge  2Rp^{B0}=v,
        $$
        where $(a)$ and $(b)$ hold because $\pmax\ge p^{A0}$, $p\ge p^{A0}$, and because $Q$ is non-negative (as established in Lemma~\ref{lem:positive-definite}-(i)); $(b)$ also uses the given inequality $p^{A0}\ge p^{B0}$; and $(c)$ holds because $Q\ge R$ (as shown in Lemma \ref{lem:positive-definite}-(iii)). Thus, \eqref{eq:thm_condition_1} holds for $p^{A(1)}=\pmax$ and $p^{A(2)}=p$.
        \item [(b)] We have 
        $$
            Q(\pmax + p) \stackrel{(a)}\ge 2Qp^{A0}>2Rp^{B0} = v,
        $$
        where the inequality follows from the given condition $Qp^{A0}> Rp^{B0}.$
        \item[(c)] Pre-multiplying both sides of the given inequality by the non-negative matrix $Q$ yields $Q\pmax - Qp^{B0} \ge Qp^{B0} - Qp^{A0}$, or equivalently, $Q\pmax + Qp^{A0}\ge 2Qp^{B0}$. Since $Q\ge R$ by Lemma~\ref{lem:positive-definite}, it follows that $Q(\pmax + p^{A0})\ge 2Rp^{B0} = v$. Thus,~\eqref{eq:thm_condition_1} is satisfied for  $p^{A(1)}=\pmax$ and $p^{A(2)}=p^{A0}$. Since $Q$ is non-negative and $p\ge p^{A0}$, we have $Q(\pmax+p)\ge Q(\pmax + p^{A0})\ge v$, which means that \eqref{eq:thm_condition_1} is also satisfied for $p^{A(1)}=\pmax$ and $p^{A(2)}=p$.
    \end{enumerate}
    \item [(ii)] In this case, we have $\pmax\le  \alpha\allone$ for some $\alpha< \min_{j\in [n]}\left\{\frac{(Rp^{B0})_j}{(Q\allone)_j} \right\}.$ Note that $\alpha (Q\allone)_j < (Rp^{B0})_j$ for all $j\in [n]$, i.e., $Q(\alpha\allone) < Rp^{B0}$. The non-negativity of $Q$ now implies that $Q\pmax\le Q(\alpha\allone)<Rp^{B0}$. Therefore,
    $$
        Qp + Qp^{A0}\stackrel{(a)} \le 2Q\pmax <2 Rp^{B0} = v,
    $$
    where $(a)$ follows from the non-negativity of $Q$ and the inequalities $p^{A0}\le \pmax$ and $p\le \pmax$. We have thus shown that \eqref{eq:thm_condition_2} holds with $p^{A(1)}=p$ and $p^{A(2)} =p^{A0}$. Since we also have $p^{A(1)}\ge p^{A(2)}$ for this choice of  $p^{A(1)}$ and $p^{A(2)}$, it now follows from Theorem~\ref{thm:general_case} that $\phi(p)<\phi(p^{A0})$. Since this applies to all $p\ne p^{A0}$ that satisfy $p^{A0}\le p\le \pmax$, we have shown that $p^{A0}$ is the most welfare-generating policy, as required. The maximality of aggregate unsustainable effort is now an immediate consequence of Theorem~\ref{thm:general_case}. 
\end{enumerate}

\section{Proof of Proposition~\ref{prop:two}}

\begin{proof} We know from the proof of Theorem~\ref{thm:one}-(ii) that there exist $G\in\{0,1\}^{n\times n}$ and $\delta>0$ such that the matrix $M_{\Delta} = M^- - M^+$ (which depends on $\beta$, $\delta$, and $\mu$) is a negative matrix. In other words, $M^- < M^+$. For this combination of $G$ and $\delta$, we have $(M^-)^2 < (M^+)^2$ and hence, $(1-\beta)(M^-)^2< (1+\beta)(M^+)^2$ (because $\beta\in (0,1)$ implies $1-\beta<1+\beta$). It follows that $R$ is a negative matrix, which means~\eqref{eq:welfare_expression} is equivalent to 
$$
    \phi(p^A) = \frac{1}{4} \left((p^A)^\top Q p^A + 2(p^{B0})^\top|R|p^{A} + (p^{B0})^\top Q p^{B0} \right).
$$
Now, $Q$ is non-negative because $M^+ = \sum_{k=0}^\infty \frac{(\delta+\mu)^k}{(1+\beta)^{k+1}} G^k$ and $M^- =\sum_{k=0}^\infty \frac{(\delta-\mu)^k}{(1-\beta)^{k+1}} G^k$ are non-negative because of Assumption~\ref{assum:mu-delta} and because $G$ is non-negative. Therefore, the above expression for $\phi(p^A)$ implies that the welfare is non-decreasing  in every entry of $p^A$, which proves that $\pmax$ (the policy that sets every entry of $p^A$ to its maximum feasible value) is an optimal solution of $\pbold$. 

On the other hand, we know from \eqref{eq:equilibrium} that every entry of $x^{B*}$ is non-decreasing in every entry of $p^A$ given that $M^+ - M^-$ is non-negative. It follows that $\pmax$ results in maximal aggregate unsustainable effort in the post-intervention equilibrium for our choice of $G$ and $\delta$. 
\end{proof}

\section{Proof of Theorem~\ref{thm:redistribution}}

We need the following lemmas to prove Theorem~\ref{thm:redistribution}.

\begin{lemma}\label{lem:short_1}
    For every $T\subset [n]$, the matrices $M^+_T:= ((1+\beta)I - (\delta +\mu)G_T)^{-1}$ and $M^-_T:= ((1-\beta)I - (\delta -\mu)G_T)^{-1}$ exist and are non-negative.
\end{lemma}

\begin{proof}
    Note that the non-zero eigenvalues of $G_T$ are identical to those of the $n\times n$ matrix
    $$
        G'_T :=
        \begin{pmatrix}
            G_T & O_{t\times t}\\
            O_{s\times t}   & O_{s\times s},
        \end{pmatrix}
    $$
    where $t:=|T|$ and $s:=n-t$. Hence, $\rho(G_T) = \rho(G_T')$, where $\rho(\cdot)$ denotes the spectral radius.
    
    Now, $G_T$ being a principal sub-matrix of $G$ implies that there exists a permutation matrix $P$ (equivalently, a relabeling of the nodes of the graph) such that $G_T$ is obtained by truncating  $G=P G P^\top$ to retain only those of its entries that are given by the first $t$ rows and first $t$ columns of $PGP^\top$. It follows from the non-negativity of $G$ that  $O_{n\times n}\le  G'_T \le PGP^\top = P G P^{-1}$, where the last equality holds because the inverse of a permutation matrix equals its transpose. Therefore, we have
    $$
        \rho(G_T) =  \rho(G_T') \stackrel{(a)}\le \rho(PGP^{-1}) \stackrel{(b)}= \rho (G),
    $$
    where $(a)$ follows from \cite[(7.10.13)] {meyer2000matrix} and $(b)$ follows from the fact that similar matrices have identical eigenvalues. In light of the above, Assumption~\ref{assum:uniqueness} now implies 
    $$ 
        \left(\frac{\delta\pm \mu}{1\pm \beta}\right)\rho(G_T)<1,
    $$
    which guarantees the existence of $M_T^\pm$. The above inequality also shows that $M_T^+$ and $M_T^-$ have Neumann series expansions given by $M_T^\pm= (1\pm \beta)^{-1} \sum_{k=0}^\infty \left(\frac{\delta\pm \mu}{1\pm \beta}\right)^kG^k$. Since $G $ is non-negative, these expansions together with Assumption~\ref{assum:mu-delta} (which is implied by Assumption~\ref{assum:domination}) imply that $M^+_T$ and $M^-_T$ are both non-negative. 
\end{proof}

    \begin{lemma}\label{lem:continuity}
    Under Assumptions~\ref{assum:uniqueness},~\ref{assum:domination}, and~\ref{assum:positive_pre-intervention}, the agent's effort levels at the (unique) equilibrium with positive sustainable effort levels are continuous in $p^A$ and $p^B$ over the open sets $p^A > p^{A0}$ and $p^B < p^{B0}$, respectively.
\end{lemma}

\begin{proof}
    It is sufficient to prove that the equilibrium effort levels are continuous in $p^A$ for a given $p^B$, as similar arguments can be used to establish continuity in $p^B$. We therefore assume $p^B=p^{B0}$ (for an arbitrary $p^{B0}\in\R_{\ge 0}^n$) in the rest of this proof.
    
    For all $p^A$ satisfying $p^{A0}\le p^A\le \plim$, we know from Theorem~\ref{thm:non-negative} that the required equilibrium is $( x^{A*},x^{B*})$ (defined in~\eqref{eq:equilibrium}). Since the right-hand sides of~\eqref{eq:equilibrium} are continuous in $p^A$, the equilibrium effort levels are continuous on the set $p^{A0}< p^{A}<\plim$.

    Now, suppose $p^A\nleq\plim$ and that the right-hand side of~\eqref{subeq:x-B} contains negative values in a subset of its entries. We then know from the proof of Theorem~\ref{thm:non-negative}-(iii)-(b) that there exists a set $\{s_1,s_2,\ldots,s_r\}=S\subset [n]$ with $r:=|S|$ and a sequence of policies $p^{A0}\le p^{A1}\le \cdots \le p^{Ar}\le p^A$ such that 
    $$
        \min_{s\in S\setminus \{s_1,s_2,\ldots, s_{\ell-1}\}}\bigg\{\bigg(\bigg(M_{\tilde{T}}^{+}+M_{\tilde{T}}^{-}\bigg)\bigg(p_{\tilde{T}}^B+\mu G_{\tilde{T}} {x}_{\tilde{S}}^{A}(p^{A\ell}) \bigg)+\bigg(M_{\tilde{T}}^{+}-M_{\tilde{T}}^{-}\bigg)\bigg( p_{\tilde{T}}^{A\ell}+\delta G_{\tilde{T} \tilde{S}} {x}_{\tilde{S}}^{A}(p^{A\ell})\bigg)\bigg)_s\bigg\} = 0
    $$
    for all $\ell\in [r]$, where $\tilde S=\{s_1,s_2,\ldots, s_{\ell-1}\}$. As shown in the same proof, the Nash equilibrium $(x^{A\ell},x^{B\ell})$ of the game with sustainable good prices given by $p^{A\ell}$,  utility structure given by \eqref{eq:main}, and effort constraints given by $\{x_{\tilde S}=\allzero\}$, equals $(\check x^A(p^{A\ell}),\check x^B(p^{A\ell}) )$ (as defined in \eqref{eq:compare_one} and \eqref{eq:next-s-a}) for $\tilde S=\{s_1,s_2,\ldots, s_{\ell-1}\}$, and this equilibrium satisfies $\frac{\partial u_i}{\partial x_i^B}\Big\lvert_{(x^A,x^B)=(x^{A\ell},x^{B\ell})}\le 0$ for all $i\in [n]$. 
    
   Next, recall from the definition of $p^{A1}$ that there exists an $i\in [n]$ such that $((M^+ + M^-)p^B + (M^+ - M^-)p^{A1})_i=0$. By the positivity of $M^- - M^+$ (implied by Theorem~\ref{thm:one}-(i) under Assumption~\ref{assum:domination}) and $p^{A2}\ge p^{A1}$ with $p^{A2}\ne p^{A1}$, we have $((M^+ + M^-)p^B + (M^+ - M^-)p^{A2})_i<0$, implying that $\{s_1\}$ is the \textit{minimal} set $\tilde S$ for which $\check x^B(p^{A2})$ is non-negative. Repeating these arguments shows that $\tilde S=\{s_1,\ldots, s_{\ell-1}\}$ is also the minimal subset of $[n]$ for which $\check x^B(p^{A\ell})$ is non-negative.

    Now, for $\ell=r$, we have $\tilde S=\{s_1,\ldots, s_r\}=S$, and these effort levels satisfy the inequalities $$ {x^{Br}_T = \check x^B(p^{Ar}) \ge \check x^B(p^{A})= \hat x^{B*}_T\ge  \allzero}$$, where the first inequality follows from
    Lemma~\ref{lem:short_3} in view of $p^{Ar}\le p^A$. Since we also have $x^{Br}_S= \allzero$, it follows that $x^{Br}\ge \allzero$. Similarly, we can  use the fact that $(x^{Ar},x^{Br}) = (\check x^A(p^{Ar}), \check x^B(p^{Ar}) )$ and $(x^{A\, r-1}, x^{B \, r-1})=(\check x^A(p^{A\, r-1}), \check x^B(p^{A\, r-1}) )$ with $\tilde S=\{s_1,\ldots, s_{r-1}\}$ to show that $x^{B\, r-1}\ge \allzero$. By induction, we can repeat these arguments to show that $x^{B\ell}\ge \allzero$ for all $\ell\in [r]$. Also, recall that $x^{A\ell}\ge \allzero$, as we have already shown in the proof of Theorem~\ref{thm:non-negative}-(iii)-(b). It now follows from Lemma~\ref{lem:(1)} (applied after setting $p^{A(1)}=\check p^A=p^{A\ell}$) that $(x^{A\ell},x^{B\ell})$ is a Nash equilibrium of the original game $\left([n],\R^n_{\ge 0}\times \R^n_{\ge 0},\{u_i\}_{i\in [n]} \right)$.  
    
    Next, we consider a policy $\hat p^A$ that satisfies $p^{A\,\ell}\le \hat p^A\le p^{A\,\ell+1}$  for a given $\ell\in \{0,1,\ldots,r-1\}$, and we define $(\hat x^A,\hat x^B)=(\check x^A(\hat p^A), \check x^B(\hat p^A))$ for $\tilde S=\{s_1,\ldots, s_\ell\}$. We know that for $\tilde S = \{s_1,\ldots,s_\ell\}$, we have $(x^{A\ell},x^{B\ell})= (\check x^A(p^{A\ell} ),\check x^B(p^{A\ell} )) $   and $(x^{A\, \ell+1}, x^{B\, \ell+1})=(\check x^A(p^{A\,\ell+1}),\check x^B(p^{A\,\ell+1}))$. Due to the non-negativity of $M_T^\pm$ (established in Lemma~\ref{lem:short_1}), the entries of  $\check x^A$ are monotonically non-decreasing in the entries of $\check p^A$. Since $x^{A\ell}$ and $x^{A\, \ell+1}$ are both non-negative, our choice of $\hat p^A$ implies that $x^{A\ell}\le \hat x^A\le x^{A\,\ell+1}$, which further implies $\hat x^A\ge \allzero$. Similarly, we can use the non-negativity of $M_T^- - M_T^+$ (established in Lemma~\eqref{lem:short_3}) to show that $\hat x^B\ge \allzero$. An application of Lemma~\ref{lem:(1)} with $p^{A(1)}=p^{A\ell}$ and $\check p^A = \hat p^A$ now implies that $(\hat x^A,\hat x^B )$ is an equilibrium of the game $\left([n],\R^n_{\ge 0}\times \R^n_{\ge 0},\{u_i\}_{i\in [n]} \right)$. We have thus shown that for any policy $\hat p^A$ satisfying  $p^{A\ell}\le \hat p^A\le p^{A\,\ell+1}$, the equilibrium is given by $(\check x^A(\hat p^A), \check x^B(\hat p^A))$ with $\tilde S= \{s_1,\ldots, s_\ell\}$.  As $\check x^A(\cdot)$ and $\check x^B(\cdot)$ are continuous functions, it follows that equilibrium effort levels are continuous at every $\hat p^A$ belonging to the set $\{\hat p^A: p^{A\ell}< \hat p^A< p^{A\,\ell+1}\}$ for each $\ell\in [r]$. Similarly, if $p^A>p^{Ar}$, we can show that the equilibrium effort levels are continuous at $\hat p^A$ on the set $p^{Ar}< \hat p^{A} < p^A$. It remains to show that these equilibrium effort levels are  continuous at $ \hat p^A$ for $p^{Ar}\le \hat p^A\le p^A$ if there exists a subset of indices $W\subset [n]$ such that $p^A_i=p^{Ar}_i$ for all $i\in W$. We will assume $W=[n]$ (i.e., $p^A=\hat p^A =p^{Ar}$) in the sequel of the proof, as similar arguments hold for $W\ne [n]$. 
    
    Thus, it suffices to show that these equilibrium effort levels are  continuous at $p^{A\ell}$ for each $\ell\in [r]$. To this end, note that  for every non-zero $w\in \R^n$,  there exists  $\varepsilon>0$ for which we have $p^{A\,\ell-1} \le p^{A\ell}+\varepsilon w\le p^{A\, \ell+1}$. Since the vector-valued function  $\R^n_{\ge \allzero}\ni p\to (M^+ + M^-)\hat p^B(V,p) + (M^+ - M^-)p$ is continuous in $p$ for all $V\subset [n]$, if we have $\min_{i\in [n]}\left( (M^+ + M^-)\hat p^B(V,p^{A\ell}) + (M^+ - M^-)p^{A\ell}\right)_i < 0$ for some $V\subset [n]$, then we also have $\min_{i\in [n]}\left( (M^+ + M^-)\hat p^B(V,p^{A\ell}+\varepsilon w) + (M^+ - M^-)(p^{A\ell}+\varepsilon w)\right)_i < 0$ for $\varepsilon_V$ small enough. Since the number of such sets $V\subset [n]$ is finite, there exists $\varepsilon_m:=\min_{V\subset [n]}\varepsilon_V>0$ such that for every $V\subset [n]$ and $\varepsilon<\varepsilon_m$, we have
    $\min_{i\in [n]}\left( (M^+ + M^-)\hat p^B(V,p^{A\ell}+\varepsilon w) + (M^+ - M^-)(p^{A\ell} + \varepsilon w)\right)_i < 0$ whenever $\min_{i\in [n]}\left( (M^+ + M^-)\hat p^B(V, p^{A\ell}) + (M^+ - M^-)p^{A\ell}\right)_i < 0$. Therefore, the minimal set $S_\varepsilon$ such that $\min_{i\in [n]}\left( (M^+ + M^-)\hat p^B(S_\varepsilon,p^{A\ell}+\varepsilon w) + (M^+ - M^-)(p^{A\ell} + \varepsilon w)\right)_i \ge 0$ is a super-set of the minimal set $S$ for which $\min_{i\in [n]}\left( (M^+ + M^-)\hat p^B(S,p^{A\ell}) + (M^+ - M^-)p^{A\ell}\right)_i \ge 0$, i.e., $S_\varepsilon\supseteq \{s_1,\ldots, s_{\ell-1}\}$. On the other hand, the non-negativity of $M^-_{\tilde T} - M^+_{\tilde T}$ for $\tilde T =[n]\setminus \{s_1,\ldots, s_\ell\}$  (established in Lemma~\ref{lem:short_3}) along with $p^{A\ell}+\varepsilon w\le p^{A\,\ell+1}$ implies that $\check x^B(p^{A\ell}+\varepsilon w)\ge \check x^B(p^{A\,\ell+1})=x^{B\,\ell+1}\ge \allzero$. Due to the equivalence of \eqref{eq:compare_one} and \eqref{eq:aliter}, this means that $\left( (M^+ + M^-)\hat p^B(\tilde S, p^{A\ell}+\varepsilon w) + (M^+ - M^-)(p^{A\ell}+\varepsilon w)\right)_i\ge 0$ for all $i\in[n]$ provided $\left( (M^+ + M^-)\hat p^B(\tilde S, p^{A\,\ell+1}) + (M^+ - M^-)p^{A\,\ell+1}\right)_i\ge 0$ for all $i\in [n]$. As $\tilde S = \{s_1,\ldots, s_\ell\}$ is the minimal set $\tilde S$ such that $\left( (M^+ + M^-)\hat p^B(\tilde S, p^{A\,\ell+1}) + (M^+ - M^-)p^{A\,\ell+1}\right)_i\ge 0$ for all $i\in [n]$,  this implies that $S_\varepsilon\subset \{s_1,s_2,\ldots, s_\ell\}$. Therefore, $\{s_1,\ldots s_{\ell-1}\} \subseteq S_\varepsilon \subseteq \{s_1,\ldots s_\ell\}$, which means that $S_\varepsilon$ equals either $\{s_1,\ldots s_{\ell-1}\}$ or $\{s_1,\ldots s_{\ell}\}$. It now follows from Theorem~\ref{thm:non-negative} and the equivalence between \eqref{eq:aliter} and \eqref{eq:compare_one}  that the equilibrium effort levels $(x^{A\varepsilon}, x^{B\varepsilon})$
    corresponding to $p^{A\ell}+\varepsilon w$ satisfy $(x^{A\varepsilon}, x^{B\varepsilon}) = (\check x^A(p^{A\ell}+\varepsilon w ), \check x^B(p^{A\ell}+\varepsilon w) )$ for either $\tilde S = \{s_1,\ldots, s_{\ell-1}\}$ or $\tilde S=\{s_1,\ldots, s_\ell\}$.
    
    On the other hand, the facts that $x^{B\ell}_{s_\ell}=0$ (which follows from the definition of $s_\ell$ and $p^{A\ell}$) and $x^{B\ell}_{\{s_1,\ldots, s_{\ell-1}\}}=\allzero$ (which follows from \eqref{eq:compare_one}) can be used to verify that the equilibrium effort pair $(x^{A\ell},x^{B\ell})$ satisfies $(x^{A\ell},x^{B\ell}) =(\check x^A(p^{A\ell}),\check x^B(p^{A\ell}))$ not only for $\tilde S = \{s_1,\ldots, s_{\ell-1}\}$ but also for $\tilde S = \{s_1,\ldots, s_\ell\}$. By the conclusion of the preceding paragraph and the continuity of $\check x^A(\cdot )$ and $\check x^B(\cdot)$, this implies that $\lim_{\varepsilon\to 0} x^{A\varepsilon}= x^{A\ell}$ and $\lim_{\varepsilon\to 0} x^{B\varepsilon} = x^{B\ell}$. Thus, the equilibrium effort levels are  continuous at $p^{A\ell}$ for each $\ell\in [r]$. 
\end{proof}

\begin{lemma}\label{lem:short_3}
    Under Assumptions~\ref{assum:uniqueness},~\ref{assum:domination} and~\ref{assum:positive_pre-intervention}, the matrix $M_T^- - M_T^+$ is non-negative for all $T\subset [n]$.
\end{lemma}

\begin{proof}
    Note that the definitions of $M_T^\pm$ are obtained by replacing $G$ with $G_T$ in the definitions of $M^\pm$. Moreover, the proof of Lemma~\ref{lem:short_1} guarantees that $M_T^+$ (respectively, $M_T^-$) has a Neumann series expansion similar to that of  $M^+$ (respectively $M^-$) for each $T\subset[n]$, and this series expansion can be obtained by simply replacing $G$ with $G_T$ in the series expansion of $M^+$ (respectively $M^-$). Therefore, we can repeat  the proof of Theorem~\ref{thm:one}-(i) after  replacing $G$ with $G_T$ throughout the proof in order to show that $M_T^- - M_T^+$ has positive entries. 
\end{proof}

    \begin{lemma}\label{lem:monotonic}
    Suppose Assumptions~\ref{assum:uniqueness},~\ref{assum:domination}, and~\ref{assum:positive_pre-intervention} hold. Then, for every $S\subset [n]$, both $ x^{A*}$ and $\hat x^{A*}(S)$ (respectively, $ x^{B*}$ and $\hat x^{B*}(S)$) are monotonically non-decreasing  in every entry of $p^A$ (respectively, $p^B$) and non-increasing in every entry of $p^B$ (respectively, $p^A$). Moreover, for any two policies $(p^{A},p^{B})$ and $(p^{A'},p^{B'})$ with $p^{A'}\ge p^{A}$, $p^{B'}\le p^{B}$ and $p^{A'}+p^{B'} \ge p^{A} + p^{B}$, every entry of $x^{A*}+x^{B*}$ (respectively, $\hat x^{A*}(S) + x^{B*}(S)$) increases on replacing $p^A$ with $p^{A'}$ in \eqref{eq:equilibrium} (respectively, \eqref{eq:aliter}). 
\end{lemma}
\begin{proof}
        First, suppose $p^A\le \plim$. By Theorem~\ref{thm:non-negative}, the equilibrium effort levels are given by~\eqref{eq:equilibrium}. 
        
        The monotonic decrease of $x^{A*}$ and $x^{B*} $ follows from the non-negativity of $M^+$ (established in the proof of Lemma~\ref{lem:positive-definite}-(i)), $M^-$ (established in the proof of Proposition~\ref{prop:one}), and $M_\Delta=M^- - M^+$ (established in Theorem~\ref{thm:one}-(i)).

    Therefore, we now consider $\hat x^{A*}(S)$ and $\hat x^{B*}(S)$. 
    We  know from the proof of Theorem~\ref{thm:non-negative}-(iii) that \eqref{eq:aliter} is equivalent to \eqref{eq:everything}. We  also know that \eqref{eq:subeq-x-S-A} implies \eqref{eq:yes-s-a}, which, according to Lemma~\ref{lem:spectral}, implies 
    \begin{align}\label{eq:what?}
        \hat x^{A*}_S &= \frac{1}{2}\sum_{k=0}^\infty H^k \Big((\delta +\mu)G_{ST} M_T^+\left(p^A + p^B\right)\cr
        &\quad + (\delta-\mu)G_{ST} M_T^- (p^A - p^B)  \Big),
    \end{align}
    where $H := \delta G_{ S} + (\delta^2+\mu^2)G_{ST}G_{TS}$. It now follows from Lemma~\ref{lem:short_1} and Assumption~\ref{assum:mu-delta} that each entry of $\hat x^{A*}_S$ is non-decreasing in each entry of $p^A$. Using this derivation and Lemma~\ref{lem:short_1}, we can also verify  that each entry of $\hat x_T^{A*}$ is non-decreasing in each entry of $p^A$. Thus, each entry of $\hat x^{A*}$ is non-decreasing in each entry of $p^A$. 

    Next, note that the coefficient of $p_T^B$ on the right-hand side of \eqref{eq:what?} is one-half of the following matrix 
    \begin{align*}
        &(\delta+\mu) G_{ST} M_T^+ - (\delta -\mu) G_{ST} M_T^-\cr
        &= G_{ST}\left((\delta+\mu)M_T^+ + (\delta -\mu)M_T^- \right)\cr
        &= G_{ST}\left(\sum_{k=0}^\infty\left(\frac{\delta+\mu}{1+\beta} \right)^{k+1} G^k - \sum_{k=0}^\infty\left(\frac{\delta-\mu}{1-\beta} \right)^{k+1} G^k  \right)\cr
        &\stackrel{(a)}\le G_{ST}\Bigg(\sum_{k=0}^\infty\left(\frac{\delta+\beta\delta}{1+\beta} \right)^{k+1} G^k\cr
        &\qquad\quad \quad - \sum_{k=0}^\infty\left(\frac{\delta-\beta\delta}{1-\beta} \right)^{k+1} G^k  \Bigg)\cr
        &\stackrel{(b)}=O,
    \end{align*}
    where $(a) $ is a consequence of Assumption~\ref{assum:domination} and $\beta\in(0,1)$, and $(b)$ holds because $\left(\frac{\delta\pm\beta\delta}{1\pm \beta}\right)=\delta$. Thus, each entry of $\hat x^{A*}_S$ is non-increasing in each entry of $p^B$. As for $\hat x^{A*}_T$, we have the following as a rearrangement of \eqref{eq:subeq3}:
    \begin{align}
        \hat x^{A*}_T &= (M_T^+ + M_T^-)p^A_T + (M_T^+ - M_T^-)p_T^B\cr
        &\quad+ ((\delta +\mu )M_T^+ +(\delta-\mu)M_T^-) \hat x^{A*}_S,
    \end{align}
    whose every entry is non-increasing in $p_T^B$: the second summand is non-increasing in $p_T^B$ by Lemma~\ref{lem:short_3} and  the third is non-increasing in $p_T^B$ by our observation on $\hat x_S^{A*}$, Assumption~\ref{assum:mu-delta}, and Lemma~\ref{lem:short_1}. 

    We have thus shown that each entry of $\hat x^{A*}$ is monotonically non-decreasing (respectively, non-increasing) in each entry of $p^A$ (respectively, $p^B$). We can similarly show that each entry of $\hat x^{B*}$ is non-decreasing (respectively, non-increasing) in each entry of $p^B$ (respectively, $p^A$).
    
    To prove the last assertion, we first observe from \eqref{eq:equivalent} that $x^{A*}+x^{B*}=M^+(p^A + p^B)$. Since $M^+$ is non-negative, it follows from $p^{A'}+p^{B'}\ge p^A + p^B$ that replacing $(p^A,p^B)$ with  $(p^{A'},p^{B'})$ cannot cause $x^{A*}+x^{B*}$ to decrease. As for $\hat x^{A*}(S)+\hat x^{B*}(S)$, we note the fact that every entry of $\hat x^{A*}_S(S)$ is monotonically non-decreasing in every entry of $p^A$ and monotonically non-increasing in every entry of $p^B$ implies that no entry of $\hat x_S^{A*}$  decreases on replacing $(p^A,p^B)$ with $(p^{A'},p^{B'})$.  It now follows from \eqref{eq:sum} and the non-negativity of $M_T^\pm$ (established in Lemma~\ref{lem:short_1}) that no entry of $\hat x^{A*}(S)+\hat x^{B*}(S)$ decreases on replacing $(p^A,p^B)$ with $(p^{A'},p^{B'})$.
\end{proof}

\begin{lemma}\label{lem:transformed_utility}
    For all $x^A,x^B\in\R^n_{\ge 0}$, the utility of agent $i\in [n]$ can be expressed as \begin{align}\label{eq:transformed_utility}
        u_i(\sigma,\Delta) = u_i^S(\sigma) + u_i^D (\Delta),
    \end{align}
    where $\sigma =\frac{1}{2} (x^A + x^B)$, $\Delta = \frac{1}{2}(x^A - x^B)$,
    $$
        u_i^S(\sigma):= (p_i^A + p_i^B) \sigma_i - (1+\beta)\sigma_i^2 + 2(\delta+\mu)\sum_{j=1}^{n}g_{ij} \sigma_i\sigma_j,
    $$
    and
    $$
        u_i^D(\Delta):= (p_i^A - p_i^B) \Delta_i - (1-\beta)\Delta_i^2 + 2(\delta-\mu)\sum_{j=1}^{n}g_{ij} \Delta_i\Delta_j.
    $$
    \def\G{\mathcal G}
    Therefore, the coupled-activity network game $\left([n],\R^n_{\ge 0}\times \R^n_{\ge 0} ,\{u_i\}_{i\in[n]} \right)$ is a pair of single-activity games, the first of which is $\G_1 =\left([n],\R^n_{\ge 0},\{u_i^S\}_{i\in[n]} \right)$ and the second of which is determined by the effort levels in $\G_1$, and is given by  $\G_2:=\left([n],\prod_{i=1}^n [-\sigma_i,\sigma_i] ,\{u_i^D\}_{i\in[n]} \right)$, wherein the constraint set $\Delta_i\in [-\sigma_i,\sigma_i]$ is equivalent to the set of non-negativity constraints $x_i^A\ge \allzero$ and $x_i^B\ge\allzero.$  
\end{lemma}

\begin{proof}
    This is a straightforward consequence of substituting $x_i^A = \sigma_i + \Delta_i$ and $x_i^B = \sigma_i - \Delta_i$ into the utility expression $\eqref{eq:main}.$ 
\end{proof}

We are now equipped to prove Theorem~\ref{thm:redistribution}.

\begin{proof}{Proof of Theorem~\ref{thm:redistribution}.}
    We first prove that choosing $\rhomax\ge \rho^0$ and $b\ge b^0$, where $(\rho^0,b^0)$ satisfies \eqref{eq:vanish}, ensures $\sum_{i=1}^n x_i^{B*}=0$ in the post-intervention equilibrium corresponding to the policy $(p^A,p^B)=(p^{A*}_{\text R}, p^{B*}_{\text R})$. From Theorem~\ref{thm:non-negative}, we know that a sufficient condition to ensure $x_i^{B*}=0$ for all $i\in [n]$ is that $p^A=  \plim$, i.e.,  $p^A= M_{\Delta}^{-1}(M^+ + M^-)p^B$. For $(p^A,p^B)=(p^{A*}_{\text R}, p^{B*}_{\text R}) = (p^{A0}+\rhomax + b, p^{B0}-b)$, this condition is equivalent to $p^{A0}+\rhomax + b = M_{\Delta}^{-1}(M^+ + M^-)(p^{B0}-b)$, which is in turn equivalent to \eqref{eq:vanish}. This proves the following: given a solution $(\rho^0,b^0)$ of \eqref{eq:vanish}, setting $\rhomax = \rho^0$ and $b=b^0$ is sufficient to ensure $x_i^{B*}=\allzero$ for all $i\in [n]$. On the other hand, the fact that the equilibrium effort levels in activity $B$ are  monotonically non-increasing in the entries of $p^A$ and non-decreasing in the entries of $p^B$ (as implied by Lemma~\ref{lem:monotonic} and Lemma~\ref{lem:continuity}) and the fact that the entries of $p^A=p^{A*}_{\text R}$ (respectively, $p^B=p^{B*}_{\text R}$) are monotonically non-decreasing (respectively, non-increasing) in the entries of $\rhomax$ and $b$ together imply that the entries of $x^{B*}$ are monotonically non-increasing in the entries of $\rhomax$ and $b$. It follows that $x^{B*}\le \allzero$ if $\rhomax\ge \rho^0$ and $b\ge b^0$. Since effort levels are non-negative, it follows that $\hat x^{B*}=\allzero$ if $\rhomax\ge \rho^0$ and $b\ge b^0$. Hence, $\sum_{i=1}^n x_i^{B*}=0$ if $\rhomax\ge \rho^0$ and $b\ge b^0$.
    
    \textit{Case 1 (Equilibrium effort levels satisfy \eqref{eq:equilibrium}):} Suppose $\rhomax$ and $b$ are small enough for the unsustainable effort levels $\{x_i^{B*}\}$ given by~\eqref{eq:equilibrium} to be non-negative in the post-intervention equilibrium resulting from the policy $(p^A,p^B)=( p^{A*}_{\text{R}},  p^{B*}_{\text{R}}) $. It now follows from the non-negativity of $M^+ + M^-$ (implied by Lemma~\ref{lem:short_1} for $T=[n]$) and the non-negativity of $M_\Delta = M^- - M^+$ (guaranteed by Theorem~\ref{thm:one}-(i) under Assumption~\ref{assum:domination}) that these effort levels are also non-negative for all $(p^A,p^B)$ satisfying $p^{A0}\le p^A\le p^{A*}_{\text{R}}$ and $p^{B0}\ge p^B\ge p^{B*}_{\text R}$. Hence, they are non-negative for every feasible policy. It follows from Theorem~\ref{thm:non-negative} that the unique equilibrium with $x^{A*}>\allzero$ is given by \eqref{eq:equilibrium}. This also means that these effort levels satisfy the first-derivative conditions $\left.\frac{\partial u_i}{\partial x_i^A}\right|_{\left(x^A, x^B\right)=\left({x}^{A*}, {x}^{B *}\right)}=\left.\frac{\partial u_i}{\partial x_i^B}\right|_{\left(x^A, x^B\right)=\left({x}^{A*}, {x}^{B *}\right)}=0$.

    After replacing $\delta/(1+\beta)$ (respectively, $\delta/(1-\beta)$) with $(\delta+\mu)/(1+\beta)$ (respectively, $(\delta-\mu)/(1-\beta)$) in the proof of \cite[Proposition 9]{chen2018multiple}, we can now repeat the arguments used therein to show that the  utility of agent $i$ in the post-intervention equilibrium is given by
    \begin{align}\label{eq:both_cases}
        &u_i(x^{A*}, x^{B*}) \cr
        &= (1+\beta)\left(\frac{x_i^{A *}+x_i^{B *}}{2}\right)^2+(1-\beta)\left(\frac{x_i^{A *}-x_i^{B *}}{2}\right)^2.
    \end{align}
    
    Now,~\eqref{eq:equilibrium} implies that $x^{A*} + x^{B*} = M^+ (p^A + p^B)$ and $x^{A*} - x^{B*} = M^- (p^A - p^B)$. It follows that
    \begin{align}\label{eq:utility_eqbm}
        u_i^*(p^A, p^B) &= \frac{1}{4} (1+\beta)\Big(\left(M^+(p^A+p^B)\right)_i\Big)^2 &\cr
        & + \frac{1}{4}(1-\beta) \Big(\left(M^-(p^A-p^B)\right)_i\Big)^2.
    \end{align}
    On the other hand, we make two observations on the policy $(p^{A*}_{\text R}, p^{B*}_{\text R})$. First, we have
    $$
        p^{A*}_{\text R} + p^{B*}_{\text R} = p^{A0}+p^{B0}+ b\ge p^A + p^B
    $$
    for all feasible $(p^A,p^B)$ according to \eqref{eq:redist_budget}. By the non-negativity of prices and the non-negativity of $M^+$ (established in the proof of Proposition~\ref{prop:two}), this implies 
    \begin{align}\label{eq:nested_1}
        M^+(p^{A*}_{\text R} + p^{B*}_{\text R}) \ge |M^+(p^A+p^B)|.
    \end{align}
    For the second observation, if the budget-penalty sufficiency condition holds \eqref{eq:half-budget}, then we have
    $$
        p^{A*}_{\text R} - p^{B*}_{\text R} = p^{A0} - p^{B0} + b + 2\rhomax \stackrel{(a)} \ge  p^{A0} - p^{B0} + 2(p^{B0 }-p^{A0}) = p^{B0} - p^{A0}\stackrel{(b)}\ge p^B - p^A,
    $$
    for all feasible $(p^A, p^B)$, where $(a)$ is a consequence of budget-penalty condition~\eqref{eq:half-budget} and $(b)$ follows from the constraints~\eqref{eq:price_up_redist} and~\eqref{eq:price_down_redist}. Thus, $p^{A*}_{\text R} - p^{B*}_{\text R} \ge p^B - p^A$ for all feasible $(p^A,p^B)$. Additionally, by~\eqref{eq:price_down_redist} and \eqref{eq:redist_budget} we have $p^{A*}_{\text R} - p^{B*}_{\text R} \ge p^A - p^B$ for all feasible $(p^A,p^B)$. Combining this with the preceding inequality results in $p^{A*}_{\text R} - p^{B*}_{\text R} \ge |p^A - p^B|$. 

    It now follows from the non-negativity of $M^-$ (established in the proof of Proposition \ref{prop:two}) that
    \begin{align}\label{eq:nested_2}
        M^-(p^{A*}_{\text R} - p^{B*}_{\text R}) &\ge M^-|p^A-p^B|\cr
        &\ge | M^-(p^A-p^B)|,
    \end{align}
    where the second inequality follows from the triangle inequality. Combining \eqref{eq:nested_1} and \eqref{eq:nested_2} with \eqref{eq:utility_eqbm} and using $\beta\in(0,1)$ yields $u_i^*(p^{A*}_{\text R},p^{B*}_{\text R})\ge u_i^*(p^A,p^B)$ for all feasible $(p^A,p^B)$, thereby proving that $(p^{A*}_{\text R} , p^{B*}_{\text R} )$ is optimal.  

    \textit{Case 2 (Equilibrium effort levels violate \eqref{eq:equilibrium}):} Now, suppose the effort levels given by \eqref{eq:equilibrium} are not all non-negative for $(p^A,p^B) = (p^{A*}_{\text R}, p^{B*}_{\text R})$. To address this case, we leverage the arguments used in Case 1 in addition to (a) the fact that the sum of equilibrium effort levels across the two activities is monotonically non-decreasing in the total price incentives across the two activities (as established in Lemma~\ref{lem:monotonic} under certain conditions), and (b) the continuity of the equilibrium effort levels in the prices of sustainable and unsustainable goods (as established in Lemma~\ref{lem:continuity}).

 We begin by noting the following: Assumption~\ref{assum:positive_pre-intervention} implies that $x^{A0}>\allzero$. Since $p^{A*}_{R}\ge p^{A0}$ and $p^{B*}_R\le p^{B0}$,  we know from the non-negativity of the $M^--M^+$ (implied by Theorem~\ref{thm:one}-(i) under Assumption \ref{assum:domination}) and the non-negativity of $M^+$ and $M^-$ (established in Lemma~\ref{lem:short_1}) that the entries of the vector on the right-hand side  of \eqref{subeq:x-A} are all positive.  It now follows from the definition of this case (Case  2) that a subset of the entries of the expression on the right-hand  side of \eqref{subeq:x-B} are negative. It  follows from Theorem~\ref{thm:non-negative} that the equilibrium effort levels resulting from the policy $(p^{A*}_{\text R}, p^{B*}_{\text R})$  are given by $(\hat x^{A*}(S),\hat x^{B*}(S))$ for some $S\subset[n]$. Consider now any feasible solution $(p^A,p^B)$ of $\pbr$. Since we have $p^A\le p^{A*}_\text{R}$, $p^B \ge p^{B*}_\text{R}$, and $p^A + p^B\le p^{A*}_\text{R}+p^{B*}_\text{R}$, Lemma~\ref{lem:monotonic} and Lemma~\ref{lem:continuity} imply that  the equilibrium effort levels $(x^{A\star},x^{B\star})$ corresponding to the policy $(p^A,p^B)$ satisfy $x^{A\star}+x^{B\star}\le \hat x^{A*}+\hat x^{B*}$. Using the change of variables $(\sigma,\Delta):=(\frac{1}{2}(x^A + x^B),\frac{1}{2} (x^A - x^B))$, this can be expressed as $\sigma^{\star}\le \hat\sigma^*$.

    Hence, if we let $S=\{i\in [n]: \hat x_i^{B*}=0\}$ (note that $S$ is non-empty by virtue of Theorem~\ref{thm:non-negative}-(iii)-(b)), then it follows for all $i\in S$ that $\hat x_i^{A*} - \hat x_i^{B*} = \hat x_i^{A*}+\hat x_i^{B*} \ge x_i^{A\star}+x_i^{B\star} = |x_i^{A\star}| + |x_i^{B\star}| \ge |x_i^{A\star} - x_i^{B\star}|$ (where the second-last step holds because of the non-negativity of equilibrium effort levels and the last step follows from the triangle inequality). In terms of $(\sigma,\Delta)$, this can be expressed as $\hat \Delta_S^*\ge |\Delta_S^\star|$.

    On the other hand, since $\hat x_i^{B*}>0$ for all $i\in [n]\setminus S$, we can repeat the arguments used in the proof of Theorem~\ref{thm:non-negative}-(i) to show that  $\frac{\partial u_i}{\partial x_i^B}\big\lvert_{(x^A,x^B)= (\hat x^{A*},\hat x^{B*}) }=0$ for all $i\in [n]\setminus S$. Similarly, the positivity of sustainable effort levels (i.e., $\hat x^{A*}>\allzero$) implies $\frac{\partial u_i}{\partial x_i^A}\big\lvert_{(x^A,x^B)= (\hat x^{A*},\hat x^{B*}) }=0$ for all $i\in [n]\setminus S$. It now follows that $\frac{\partial u_i}{\partial \Delta_i}\big\lvert_{(\sigma,\Delta)= (\hat x^{A*}+\hat x^{B*},\hat x^{A*}-\hat x^{B*} ) }=0$ for all $i\in [n]\setminus S$. We now use the utility expressions provided in Lemma \ref{lem:transformed_utility} to express these first derivative conditions as
    $$
        \hat \Delta_i^* = \frac{1}{2(1-\beta)}\left( (p^{A*}_{\text R } - p^{B*}_{\text R })_i + 2(\delta - \mu)\sum_{j=1}^n g_{ij}\hat \Delta_j^* \right) 
    $$
    for all $i\in T:=[n]\setminus S$, which can be expressed compactly as
    $$
        \hat \Delta_T^* =  \frac{1}{2(1-\beta)}\left( (p^{A*}_{\text R } - p^{B*}_{\text R })_T + 2(\delta-\mu) G_T\hat \Delta_T^* + G_{TS}\hat\Delta_{S}^* \right).
    $$
    Rearranging this yields:
    $$
        2((1-\beta)I - (\delta-\mu)G_T)\hat \Delta^*_T =  (p^{A*}_{\text R } - p^{B*}_{\text R })_T + G_{TS}\hat \Delta^*_S.
    $$
    Multiplying both sides by $\frac{1}{2}M_T^- = \frac{1}{2}((1-\beta)I - (\delta-\mu)G_T)^{-1}$ now gives
    \begin{align}\label{eq:dreg_1}
        \hat \Delta_T^* = M_T^-( (p^{A*}_{\text R } - p^{B*}_{\text R })_T + G_{TS}\hat \Delta^*_S)
    \end{align}
    
    Next, we note the following, since $p^B\ge p^{B*}_{\text{R}}$ and $p^A\le p^{A*}_{\text{R}}$, Lemma~\ref{lem:monotonic} in light of Lemma \ref{lem:continuity} implies  $x^{B\star}\ge \hat x^{B*}$, and hence, $x_i^{B\star}>0$ for all $i\in [n]\setminus S$. Therefore, we can show the following, which is analogous to \eqref{eq:dreg_1}:
    \begin{align}\label{eq:dreg_2}
         \Delta_T^\star = M_T^-((p^A-p^B)_T + G_{TS} \Delta^\star_S).
    \end{align}
    Now, we use \eqref{eq:dreg_1} and \eqref{eq:dreg_2} to show  
    \begin{align*}
        \hat \Delta_T^* &= M_T^-( (p^{A*}_{\text R } - p^{B*}_{\text R })_T + G_{TS}\hat \Delta^*_S)\cr 
        &\stackrel{(a)}\ge M_T^- ( |p^A_T - p^B_T| + G_{TS}\hat \Delta^*_S)\cr
        &\stackrel{(b)}\ge M_T^-  ( |p^A_T - p^B_T| + G_{TS}|\Delta_S^\star|)\cr
        &\stackrel{(c)}\ge |M_T^-(p_T^A-p_T^B)+ G_{TS}\Delta_S^\star|\cr
        &= |\Delta _T^\star|,
    \end{align*}
    where $(a)$ follows from the inequality $p^{A*}_{\text R } - p^{B*}_{\text R } \ge |p^A-p^B|$ (established earlier in this proof) and the non-negativity of $M_T^-$, $(b)$ follows from the inequality $\hat \Delta_S^*\ge |\Delta_S^\star|$ (established earlier in this proof), and $(c)$ follows from the triangle inequality and the non-negativity of $M_T^-$ and $G$. We have thus shown  $\hat \Delta^*_V\ge |\Delta^\star_V |$ for both $V=T$ and $V=S$. As $T+S=[n]$, this implies $\hat \Delta^*\ge |\Delta^\star|$.

    We now compute the equilibrium utilities for $i\in T$. Since we have $\left.\frac{\partial u_i}{\partial x_i^A}\right|_{\left(x^A, x^B\right)=\left(\hat{x}^{A *}, \hat{x}^{B *}\right)}= \left.\frac{\partial u_i}{\partial x_i^B}\right|_{\left(x^A, x^B\right)=\left(\hat{x}^{A *}, \hat{x}^{B *}\right)}=0$ for all $i\in T$, we can repeat the arguments of Case 1 to show that the equilibrium utilities of the agents indexed by $T$ have the same functional form as~\eqref{eq:both_cases}, i.e., the following holds for all $i\in T$:
    $$
        u_i(\hat x^{A*}, \hat x^{B*}; p^{A*}_{\text R},p^{B*}_{\text R}) 
        = (1+\beta)\left(\frac{x_i^{A *}+x_i^{B *}}{2}\right)^2+(1-\beta)\left(\frac{x_i^{A *}-x_i^{B *}}{2}\right)^2.
    $$
    We now compute the equilibrium utilities  for $i\in S$ using \eqref{eq:equilibrium} and the fact that $\hat x_i^{B*}$ as follows:
    \begin{align}\label{eq:annoying_one}
        &u_i\left(\hat x^{A*}, \hat x^{B*}; p^{A*}_{\text R }, p^{B*}_{\text R } \right)\cr
        &=  p^{A*}_{\text R i} \hat x_i^{A*}-\frac{1}{2}(\hat x_i^{A*})^2 +\sum_{j=1}^n g_{i j} \hat x_i^{A*} \left(\delta  \hat x_j^{A*}+\mu  \hat x_j^{B*}\right).
    \end{align}
    On the other hand, we have $\left.\frac{\partial u_i}{\partial x_i^A}\right|_{\left(x^A, x^B\right)=\left(\hat{x}^{A *}, \hat{x}^{B *}\right)}=0$, which is equivalent to 
    \begin{align}\label{eq:annoying_two}
        \hat x_i^{A*} = p^{A*}_{\text R i} + \sum_{j=1}^n g_{ij}(\delta \hat x_j^{A*} + \mu\hat x_j^{B*})
    \end{align}
    for $i\in [n]$ since $\hat x_i^{B*}=0$ for all such $i$. Combining \eqref{eq:annoying_one} and \eqref{eq:annoying_two} now gives
    \begin{align}\label{eq:enough_1}
        &u_i\left(\hat x^{A*}, \hat x^{B*}; p^{A*}_{\text R }, p^{B*}_{\text R } \right)\cr
        &= \frac{1}{2}\left(p^{A*}_{\text R i} + \sum_{j=1}^n g_{ij}(\delta \hat x_j^{A*} + \mu\hat x_j^{B*}) \right)^2 \cr
        &= \frac{1}{2}(\hat x_i^{A*})^2\cr
        &= (1+\beta)\left(\frac{\hat x_i^{A *}+\hat x_i^{B *}}{2}\right)^2+(1-\beta)\left(\frac{\hat x_i^{A *}-\hat x_i^{B *}}{2}\right)^2,
    \end{align}
    where the last step holds because $\hat x_i^{B*}=0$. Thus, regardless of whether $i\in S$ or $i\in T$, the expression for the equilibrium utility  of agent $i$ has the same functional form as~\eqref{eq:both_cases} even though a subset of the entries of the right-hand side of \eqref{subeq:x-B} (i.e., a subset of the entries of the vector $(M^+ + M^-)p^{B*}_{\text R} + (M^+ - M^-)p^{A*}_{\text R}$) are negative. In other words, the above functional form applies to both Case 1 and Case 2. In light of this, we also have
    \begin{align}\label{eq:enough_2}
        &u_i\left( x^{A\star},  x^{B\star}; p^A, p^B \right)\cr
        &= (1+\beta)\left(\frac{ x_i^{A\star}+ x_i^{B \star}}{2}\right)^2+(1-\beta)\left(\frac{ x_i^{A\star }-x_i^{B\star }}{2}\right)^2,
    \end{align}
    Using \eqref{eq:enough_1} and \eqref{eq:enough_2}, we note that for all $i\in [n]$, we have
    \begin{align*}
        &u_i\left(\hat x^{A*}, \hat x^{B*}; p^{A*}_{\text R }, p^{B*}_{\text R } \right) - u_i\left( x^{A\star},  x^{B\star}; p^A, p^B \right)\cr
        &= (1+\beta)\left(\left(\frac{\hat x_i^{A *}+\hat x_i^{B *}}{2}\right)^2- \left(\frac{ x_i^{A \star}+ x_i^{B \star}}{2}\right)^2 \right)\cr
        &\quad+(1-\beta)\left(\left(\frac{\hat x_i^{A *}-\hat x_i^{B *}}{2}\right)^2- \left(\frac{ x_i^{A \star}- x_i^{B \star}}{2}\right)^2 \right)\cr
        &=(1+\beta)((\hat \sigma_i^*)^2 - (\sigma_i^\star)^2) \cr
        &\quad+(1-\beta)((\hat \Delta_i^*)^2 - (\Delta_i^\star)^2)\cr
        &\ge 0,
    \end{align*}
    where the inequality holds because $\hat\sigma^*\ge \sigma^\star\ge \allzero$ and $\hat\Delta^*\ge |\Delta^\star|$ as established earlier. We have thus shown that no agent's utility can decrease on replacing $(p^A,p^B)$ with $(p^{A*}_{\text R},p^{B*}_{\text R})$. Since $(p^A,p^B)$ was an arbitrary feasible policy, the proof is complete. 
\end{proof}

\section{Proof of Proposition \ref{prop:pessimistic}}
\begin{proof} 
If $b = \allzero$, then we have
    $$
        p^{A*}_{\text R}  - p^{B*}_{\text R}  = p^{A0} - p^{B0} + 2\rhomax < p^{A0} - p^{B0} + 2(p^{B0} - p^{A0}) = p^{B0} - p^{A0},
    $$
where the inequality follows from our assumption on $\rhomax$. Thus, $p^{B0 }-p^{A0}\ge p^{A*}_{\text R} - p^{B*}_{\text R} $. Since $p^A\le p^{A*}_{\text R} $ and $p^B \ge p^{B*}_{\text R} $ for all feasible $(p^A,p^B)$, it follows that $p^{B0 }-p^{A0}\ge p^A - p^B$ for all feasible $(p^A,p^B)$. On the other hand, $p^A\le p^{A*}_{\text R} $ and $p^B \ge p^{B*}_{\text R} $ also implies $p^{B0} - p^{A0}\ge p^B - p^A$ for all feasible $(p^A,p^B)$. Therefore, we have $p^{B0} - p^{A0}\ge |p^A - p^B|$ for all feasible $(p^A,p^B)$. As a result, we have $$(p^A - p^B)^\top (M^-)^2(p^A - p^B) \le |p^A - p^B|^\top (M^-)^2 (p^A - p^B) \le (p^{A0} - p^{B0})^\top (M^-)^2(p^{A0} - p^{B0})$$ for all feasible $(p^A,p^B)$. 

On the other hand, we have $p^A +p^B = p^{A0} + \rhomax +p^{B0} - \rhomax = p^{A0}+p^{B0}$  for all feasible $(p^A,p^B)$ because $b=\allzero$. Therefore, $$(p^A + p^B)^\top (M^+)^2(p^A + p^B) \le (p^{A0} + p^{B0})^\top (M^+)^2(p^{A0} + p^{B0})$$ for all feasible $(p^A,p^B)$. Combining this with the conclusion of the preceding paragraph and~\eqref{eq:utility_eqbm} implies that $\phi(p^{A0},p^{B0})\ge \phi(p^A,p^B)$ for all feasible $(p^A,p^B)$, which completes the proof. 
\end{proof}

\section{Proof of Theorem~\ref{thm:component_wise_pricing}}\label{app:proof_component_wise}

\begin{proof}
We follow the steps outlined in the proof sketch provided in the main text in Section \ref{sec:component-wise}.

\begin{enumerate}
    \item To begin, we make the change of variables $p_\ell=\frac{\tilde p_\ell}{b_{\Delta\ell}} +\frac{v_\ell}{2q_\ell}$ in order to express $\widetilde\pbold$ as
\begin{align}
    \text{Minimize}\quad \sum_{\ell=1}^c & d_\ell \tilde p_\ell^2\cr 
   \text{s.t.}\quad \sum_{\ell=1}^c \tilde p_\ell &\ge \tilde k_0,\cr
    \tilde  p_{0\ell} \le \tilde p_{(\ell)} &\le \tilde  u_{\ell},\, \text{for all }\ell\in [c],
\end{align}
where $d_\ell := - \frac{q_{(\ell)} }{b_{\Delta(\ell)}^2}$, $\tilde k_0:=k_0 - \sum_{\ell=1}^c\left(\frac{b_{\Delta(\ell)}v_{(\ell)}} {2q_{(\ell)}} \right)$, $\tilde p_{0\ell} := \left(p^{A0}_{(\ell)} - \frac{v_{(\ell)} }{2q_{(\ell)}} \right)b_{\Delta(\ell)}$,  $\tilde u_\ell := \left(\tpmax_{(\ell)} - \frac{v_{(\ell)} }{2q_{(\ell)}} \right)b_{\Delta (\ell)}$, and $\tilde p_{(\ell)} := \left(p^A_{(\ell)} - \frac{v_{(\ell)}}{2q_{(\ell)}} \right)b_{\Delta (\ell)}$.

We now perform the steps outlined in~\cite[Pages 245-246] {li2023efficient} to obtain the following convex relaxation of $\widetilde\pbold$:
\begin{align}\label{eq:primal}
    \text{Minimize}\quad \sum_{\ell=1}^c &g_\ell^{**} (\tilde p_\ell)\cr 
   \text{s.t.}\quad \sum_{\ell=1}^c \tilde p_\ell &\ge \tilde k_0.
\end{align}
Here, $g_\ell^{**}$ is the biconjugate of the following function:
$$
    g_\ell(\tilde p_\ell) := 
    \begin{cases}
        \frac{1}{2}d_\ell \tilde p_\ell^2 \quad&\text{if }\tilde p_{0\ell}\le \tilde p_\ell \le \tilde u_\ell,\\
        \infty\quad&\text{otherwise}.
    \end{cases}
$$

The definition of $g_\ell^{**}$ implies that \eqref{eq:primal} can be expressed as the minimization of a continuous convex objective over a compact set. Hence, there exists an optimal solution to \eqref{eq:primal} according to the Extreme Value Theorem.

\item Observe that the Lagrangian of the above problem is $L(\tilde z, \lambda) := \sum_{\ell=1}^c g_\ell^{**}(\tilde z_\ell) + \lambda\left(\tilde k_0 - \sum_{\ell=1}^c \tilde z_\ell\right).$ Consequently, we compute the objective function of the dual problem as follows.
    \begin{align}\label{eq:dual_function}
    \boldsymbol d(\lambda) &:= \min_{\tilde z} L(\tilde z,\lambda)\cr
    &= \lambda \tilde k_0 + \min_{\tilde z_1, \ldots, \tilde z_c} \sum_{\ell=1}^c \left( g_\ell^{**}(\tilde z_\ell) - \lambda \tilde z_\ell\right)\cr
    &=\lambda \tilde k_0 +  \sum_{\ell=1}^c \min_{\tilde z_\ell} \left( g_\ell^{**}(\tilde z_\ell) - \lambda \tilde z_\ell\right)\cr
    &=\lambda\tilde k_0 - \sum_{\ell=1}^c \max_{\tilde z_\ell} \left(\lambda \tilde z_\ell - g_\ell^{**}(\tilde z_\ell)\right)\cr
    &=\lambda\tilde k_0 - \sum_{\ell=1}^c (g_\ell^{**})^*(\lambda)\cr
    &\stackrel{(a)}=\lambda\tilde k_0 - \sum_{\ell=1}^c g_\ell^*(\lambda),
\end{align}
where $(a)$ holds because, according to~\cite[Theorem 11.1]{rockafellar2009variational}, the convex hull of $g_\ell$ being proper implies that $g_\ell^*$ is a closed and proper convex function. A second application of \cite[Theorem 11.1]{rockafellar2009variational} now yields $(g_\ell^{**})^*=(g_\ell^*)^{**} = g_\ell^*$.

These derivations enable us to formulate the Lagrange dual problem as $\text{Maximize}_{\lambda\ge 0} \left\{ \lambda\tilde k_0 - \sum_{\ell=1}^c g_\ell^*(\lambda)  \right\}$, or equivalently, as
\begin{align}\label{eq:dual_problem}
    \text{Minimize}_{\lambda\ge 0} \left\{\sum_{\ell=1}^c g_\ell^*(\lambda)- \lambda\tilde k_0  \right\}
\end{align}

As there is no nonlinear inequality constraint in the primal problem~\eqref{eq:primal}, Slater conditions assure us that strong duality holds between the primal problem and the dual problem~\eqref{eq:dual_problem}.

To solve \eqref{eq:dual_problem}, we first compute $g_\ell^*$ for all $\ell\in [n]$ using the relations $g_\ell^*(\lambda) = (g_\ell^{**})^*(\lambda)=\sup_{\tilde p_\ell}\{\lambda\tilde p_\ell - g_\ell^{**}(\tilde p_\ell) \}$. As $g^{**}_\ell$ is affine over $[\tilde p_{0\ell}, \tilde u_\ell]$, the supremum is attained at either $\tilde p_{0\ell}$ or $\tilde u_{0\ell}$. The complete computation results in 
\begin{align}\label{eq:g_ell_star}
    g_\ell^*(\lambda) = 
    \begin{cases}
        \tilde u_\ell\lambda + |d_\ell|\tilde u_\ell^2 \quad &\text{if }\lambda\ge -|d_\ell|(\tilde p_{0\ell} +\tilde u_\ell)\\
        \tilde p_{0\ell} \lambda + |d_\ell|\tilde p_{0\ell}^2 \quad &\text{if }\lambda< -|d_\ell|(\tilde p_{0\ell} +\tilde u_\ell).
    \end{cases}    
\end{align}

\def\d{\boldsymbol{d}}

We are now ready to characterize the optimal solution $\lambda^*$ of~\eqref{eq:dual_problem}. To this end, note that $\boldsymbol d(\lambda)$, being the Lagrange dual function, is concave in $\lambda$, and hence, $\sum_{\ell=1}^c g_\ell^*(\lambda) - \tilde k_0\lambda=-\boldsymbol d(\lambda)$ is convex in $\lambda$. Now, \eqref{eq:g_ell_star} implies that $-\boldsymbol d(\lambda)$ has at most $c$ distinct kinks that are contained in the set $\{\kappa_\ell:\ell\in [c]\}$, where $\kappa_\ell:=-|d_\ell|(\tilde p_{0\ell} + \tilde u_\ell)$. Using the expressions defining $d_\ell$, $\tilde p_{0\ell}$ and $\tilde u_\ell$, we can show that $\kappa_\ell = - \left|\frac{q_{(\ell)}}{b_{\Delta (\ell)}}\right| \left(p^{A0}_{(\ell)} + \tpmax_{(\ell)} - \frac{v_{(\ell)} }{q_{(\ell)}} \right) $. Now, $q_{(\ell)}= \allone^\top Q_{(\ell)} \allone$ is non-negative according to Lemma~\ref{lem:positive-definite}-(i), and $b_{\Delta(\ell)}=M_\Delta\allone$ is positive because $M_\Delta$ is positive according to Theorem~\ref{thm:one}-(i). It follows that
$$
    \kappa_\ell = - \left(\frac{q_{(\ell)}}{b_{\Delta (\ell)}}\right) \left(p^{A0}_{(\ell)} + \tpmax_{(\ell)} - \frac{v_{(\ell)} }{q_{(\ell)}} \right).
$$
Thus, $\kappa_\ell = -\gamma_\ell$ for all $\ell\in [c]$. Therefore, our ordering/indexing of the connected components of the strategic interaction network (as described in the theorem statement) now implies $\kappa_{1} \le \kappa_{2}\le \cdots\le \kappa_{c}$. Since our goal is to minimize $-\d(\lambda)$ over $\lambda\ge 0$, we define $\tilde S_L:= \{\ell\in [c]: \kappa_\ell \ge 0\}$. Observe from the definition of $\kappa_\ell$ that $\tilde S_L = S_L$ (where $S_L$ is as defined above~\eqref{eq:only_c_policies}). Therefore, $\tilde S_L=\{\ell', \ell'+1,\ldots, c\}$ and $\ell'=\min\{\ell\in [c]: \kappa_\ell\ge 0\}$.

Now, from the fact that $-\boldsymbol d$ is piece-wise affine, we observe that the first derivative of $-\boldsymbol d(\lambda)$ on the  interval $(\kappa_{\ell},\kappa_{\ell+1})$ (with $\kappa_{c+1}:=\infty$) is a constant given by 
$$
    \boldsymbol d'_\ell := \sum_{m=1}^\ell \tilde u_{m } + \sum_{m=\ell+1}^{c} \tilde p_{0m }  - \tilde k_0 = \sum_{m=1}^\ell (u_{m} - \tilde p_{0m}) + \sum_{m=1}^{c} \tilde p_{0m }  - \tilde k_0
$$
for every $\ell\in \tilde S_L$. Additionally, we define $\d'_{\ell'-1}:= \sum_{m=1}^{\ell'-1} \tilde u_{m } + \sum_{m=\ell'}^{c} \tilde p_{0m }  - \tilde k_0= \sum_{m=1}^{\ell'-1} (u_{m} - \tilde p_{0m}) + \sum_{m=1}^{c} \tilde p_{0m }  - \tilde k_0$, which is the first derivative of $-\d(\lambda)$ on the interval $(0, \kappa_{\ell'})$.  Now, we have $\d'_{\ell'-1}\le \d'_{\ell'}\le \cdots\le \d'_{c}$  because  $\tilde u_{m} \ge \tilde p_{0m}$ by definition for all $m\in [c]$. Therefore, three cases arise: (a) $\d'_\ell< 0$ for all $\ell\in \tilde S_L\cup \{\ell'-1\}$, (b) $\d'_\ell\ge 0$ for all $\ell \in \tilde S_L\cup \{\ell'-1\}$, (c) there exists an $\ell_T\in \tilde S_L$ such that $\d'_\ell < 0$ for all $\ell'-1\le \ell < \ell_T$ and $\d'_\ell \ge 0$ for all $ \ell_T\le \ell\le c$.

Note that $\d_1'\le \cdots \le \d_c'$ implies that Case (a) is equivalent to $\d'_{c}<0$, which is in turn equivalent to $\sum_{m=1}^c \tilde u_{m} -\tilde k_0 < 0$. Using the definitions of $\tilde u_\ell$ and $\tilde k_0$, this can be verified to imply  $\sum_{\ell=1}^{c} b_{\Delta (\ell)} \tpmax_{(\ell)} <k_0$. Since $b_\Delta = M_\Delta\allone\ge 0$ as implied by Theorem~\ref{thm:one}-(i), this  further implies that $\sum_{\ell=1}^c b_{\Delta(\ell)}p^A_{(\ell)}\le \sum_{\ell=1}^c b_{\Delta(\ell)}\tpmax_{(\ell)}< k_0$ for all $p^A$ in the range $p^{A0}\le p^A\le  \tpmax$. Thus, the tolerance constraint~\eqref{eq:modified_unsustainable_constraint} is violated by every bounded policy, making $\widetilde\pbold$ infeasible. This contradicts our assumption that $\widetilde\pbold$ is feasible, thereby establishing that Case (a) never occurs.

Similarly, Case (b) is equivalent to $\d'_{\ell-1} \ge 0$, which in turn is equivalent to $\sum_{m=1}^{\ell'-1} \tilde u_{0\ell} + \sum_{m=\ell'}^c \tilde p_{\ell}  \ge \tilde k_0$, which in turn is equivalent to 
\begin{align}\label{eq:needed}
    \sum_{\ell=1}^{\ell'-1} b_{\Delta(\ell)}\tpmax_{(\ell)} + \sum_{\ell=\ell'}^c b_{\Delta(\ell)}p^{A0}_{(\ell)}  \ge  k_0.
\end{align}
Now, $\tilde S_L = S_L$ implies that $\ell<\ell'$ is equivalent to $p^{A0}_{(\ell)}+\tpmax_{(\ell)}\ge\frac{v_{(\ell)}}{q_{(\ell)}}$, which further implies $p^{(0*)}_{(\ell)} = \tpmax_{(\ell)}$ for all $\ell<\ell'$. Similarly, we have $p^{(0*)}_{(\ell)} = p^{A0}_{(\ell)}$ for all $\ell\ge \ell'$. Therefore,~\eqref{eq:needed} is equivalent to $b_{\Delta}^\top p^{(0*)} \ge k_0$, which implies that $p^{(0*)}$ satisfies the tolerance constraint~\eqref{eq:modified_unsustainable_constraint}. Thus, $p^{(0*)}$ is a feasible solution of $\widetilde\pbold$. Since $p^{(0*)}$ maximizes the objective function of $\widetilde\pbold$ over the superset of the feasible region defined by $p^{A0}\le p^A\le \tpmax$ (see the in-text discussion in Section~\ref{sec:component-wise}), it follows that $p^{(0*)}$ is the optimal solution of $\widetilde\pbold$ in Case (b). In this case, we can use~\eqref{eq:p-knot-star} and~\eqref{eq:only_c_policies} along with the observation 
$$
    p^{(0*)}_{(\ell)} =
    \begin{cases}
        \tpmax_{(\ell)}\quad&\text{if }\ell<\ell'\\
        p^{A0}_{(\ell)} \quad&\text{if }\ell\ge\ell'.
    \end{cases}
$$
to show that $\ell^*=\ell'-1$. It now follows from~\eqref{eq:relaxed_optimal} that 
$$
    \bar p_{(\ell)} =
    \begin{cases}
        p^{(\ell')}_{(\ell)}\quad&\text{if }\ell\ne \ell'\\
        p^{A0}_{(\ell)} \quad&\text{if }\ell=\ell'.
    \end{cases}
$$

By using the definition of $p^{(\ell)}$ from~\eqref{eq:only_c_policies}, we can show that the above implies $\bar p = p^{(0*)}$. Therefore, $\bar p$ is the optimal solution of $\widetilde\pbold$.

The only non-trivial case that remains is Case (c), on which we focus throughout the remainder of this proof. In this case, we have $\ell_T>\ell'-1$ by definition. Also, the definition of $\ell_T$ implies that $-\d(\lambda)$ is non-increasing in $\lambda$ over the intervals $[0,\kappa_{\ell' } ),(\kappa_{\ell' },\kappa_{\ell'+1} ) ,\ldots, (\kappa_{\ell_T-1 },\kappa_{\ell_T})$ and non-decreasing over the intervals $(\kappa_{\ell_T}, \kappa_{\ell_T+1}),\ldots, (\kappa_{c},\infty)$. By the continuity of $\d$, this means that $-\boldsymbol d(\lambda)$ is non-increasing over $[0,\kappa_{\ell_T}]$ and non-decreasing over $[\kappa_{\ell_T},\infty)$. Therefore, ${-\boldsymbol d(\lambda)\ge -\boldsymbol d(\kappa_{\ell_T})}$ for all $\lambda\in [0,\kappa_{\ell_T}]$ and $-\boldsymbol d(\kappa_{\ell_T})\le-\boldsymbol d(\lambda)$ for all $\lambda\in [\kappa_{\ell_T},\infty)$. Equivalently, $-\boldsymbol d(\kappa_{\ell_T})\le -\boldsymbol d(\lambda)$ for all $\lambda\in [0,\infty)$. Hence, $\lambda^* = \kappa_{\ell_T}$. 

In fact, we can show that $\lambda^* = \kappa_{\ell^*}$ (with $\ell^*$ as defined below~\eqref{eq:only_c_policies}) as follows. Observe from the definition of $\ell_T$ that $\ell_T:=\min\{\ell\in \tilde S_L: \d_\ell'\ge 0 \}$. Using the definitions of $\d'_\ell$, $\tilde u_\ell$, $\tilde p_{0\ell}$ and $\tilde k_0$, this can be shown to be equivalent to $\ell_T = \min\left\{\ell\in \tilde S_L: \sum_{m=1}^\ell b_{\Delta (m)} \tpmax_{(\ell)} + \sum_{m=\ell+1}^c b_{\Delta (m)} p^{A0}_{(\ell)} \ge k_0 \right\}$.  The definition of $p^{(\ell)}$ (in~\eqref{eq:only_c_policies}) and the fact that $\tilde S_L = S_L$ now suggest that this can be expressed compactly as
\begin{align}\label{eq:needless}
    \ell_T = \min\left\{\ell\in S_L: b_{\Delta} ^\top p^{(\ell)} \ge k_0 \right\}.
\end{align}
That is, $\ell_T = \ell^*$. Therefore, $\lambda^* = \kappa_{\ell_T} = \kappa_{\ell^*}$. Since $\ell_T>\ell'-1$, this also implies that $\ell^*>\ell'-1$, an inequality we will need later in the proof.

On the other hand,  the definitions of $p^{(0*)}$ and $\kappa_{\ell}$ imply that $p^{(0*)}_{(\ell)} = \tpmax_{(\ell)}$ if and only if $\kappa_\ell\le 0$. 

Now, observe from the definition of $\ell^*$ that $\ell^*:=\min\{\ell\in S_L: \d_\ell'\ge 0 \}$. Since $\kappa_\ell$ is increasing in $\ell$, this implies that $\lambda^* = \kappa_{\ell^*} = \min\{\kappa_{\ell}: \d'_\ell \ge 0\} = \min \{-|d_\ell|(\tilde p_{0\ell}+\tilde u_\ell): \d_\ell'\ge 0\}$.

\item Before using subgradients, we   define the function $$\tilde g_\ell^*(\lambda):=
    \begin{cases}
    g_\ell^*(\lambda)\quad&\text{if }\lambda\ge 0\\
    \infty\quad&\text{if }\lambda<0
    \end{cases}$$
in order to express the Lagrange dual as the following unconsrained optimization problem, where the choice of $m\in [n]$ is arbitrary: 
\begin{align}\label{eq:dual_problem}
    \text{Minimize}_{\lambda\in\R}  \left\{ \tilde g_m^*(\lambda)+\sum_{\ell\in [c]\setminus\{m\} }  g_\ell^*(\lambda)- \lambda\tilde k_0  \right\}.
\end{align} Denoting the set of subgradients of a function $f$ by $\partial f$, we have $0\in \partial\Phi(\lambda^*)$, where $\Phi(\lambda):=\tilde g_m^*(\lambda)+\sum_{\ell\in [c]\setminus\{m\} }  g_\ell^*(\lambda)- \lambda\tilde k_0$ for all $\lambda\in\R$. 
This can be shown to imply that 
\begin{align} \label{eq:subgradient}
    \tilde k_0 \in \partial \tilde g_m^*(\lambda^*) + \sum_{\ell\in[c]\setminus\{m\}} \partial g_\ell^*(\lambda^*).
\end{align}

    Given that $g_\ell^* = (g_\ell^{**})^*$ and that $g_\ell^{**}$ is convex and closed,~\cite[Theorem 11.3]{rockafellar2009variational} implies that $\partial g_\ell^*(\lambda) = \arg\max_{\tilde z_\ell}\left\{\lambda\tilde z_\ell - g_\ell^{**}(\tilde z_\ell)\right\}$ (as also discussed briefly in \cite{li2023efficient}). Substituting the expression for $g_\ell^{**}$ and evaluating the resulting maximizer gives
\begin{align}\label{eq:subgrad_one}
    \partial g_\ell^*(\lambda) =
    \begin{cases}
        \{\tilde u_\ell\}\quad&\text{if }\lambda >-|d_\ell|(\tilde u_\ell + \tilde p_{0\ell})\\
        [\tilde p_{0\ell}, \tilde u_\ell] \quad & \text{if }\lambda=-|d_\ell|(\tilde u_\ell+\tilde p_{0\ell})\\
        \{\tilde p_{0\ell}\} \quad&\text{if } \lambda < -|d_\ell|(\tilde u_\ell + \tilde p_{0\ell}).
    \end{cases}
\end{align}

Extending this analysis further, we now show that the subgradients of $\{g_\ell\}$ are closely related to the optimal solution of $\widetilde\pbold$. For this purpose, we  
first obtain the following as a consequence of strong duality between the primal and dual problems:

\begin{align}\label{eq:strong_duality}
    \sum_{\ell=1}^c g_{\ell}^{**}(\tilde p_\ell^*)=\lambda^*\tilde k_0 - \sum_{\ell\in [c]\setminus\{m\} } g_\ell^*(\lambda^*) - \tilde g_m^*(\lambda^*)
\end{align}

From the definition of $\tilde g_m^*$, we know that $\lambda^*\in [0,\infty)$. On this basis, we now consider two cases.

\subsubsection*{Case 1: $\lambda^*>0$.} 
As $\tilde g_m^*(\lambda) = g_m^*(\lambda)$ for all $\lambda\ge 0$, we have, $\partial\tilde g_m^*(\lambda) = \partial g_m^*(\lambda)$ for all $\lambda>0$. Hence,  $\partial\tilde g_m^*(\lambda^*) = \partial g_m^*(\lambda^*)$. Hence,~\eqref{eq:subgradient} is now equivalent to $\tilde k_0 \in \sum_{\ell=1}^c \partial g_\ell^*(\lambda^*)$. In other words, there exist $y_\ell\in \partial g_\ell^*(\lambda^*)$ such that $\sum_{\ell=1}^c y_\ell = \tilde k_0$. For such a choice of $\{y_\ell:\ell\in[c]\}$, we can simplify~\eqref{eq:strong_duality} as follows:
\begin{align*}
    &\sum_{\ell=1}^c g_{\ell}^{**}(\tilde p_\ell^*)\cr
    &=\lambda^*\tilde k_0 - \sum_{\ell=1}^c  g_\ell^*(\lambda^*) \cr
    &=\lambda^*\tilde k_0 - \sum_{\ell=1}^c \sup_{\tilde p_\ell\in[\tilde p_{0\ell}, \tilde u_\ell]} \{\lambda^* \tilde p_\ell - g_\ell^{**}(\tilde p_\ell)\}\cr
    &\stackrel{(a)}{=}\lambda^*\tilde k_0 -\sum_{\ell=1}^c (\lambda^* y_\ell - g_\ell^{**}(y_\ell))\cr
    &= \lambda^*\left(\tilde k_0 - \sum_{\ell=1}^c  y_\ell\right) +\sum_{\ell=1}^c g_\ell^{**}(y_\ell)\cr
    &=\sum_{\ell=1}^c g_\ell^{**}(y_\ell).
\end{align*}
Here, $(a)$ follows from the relation $\partial g_\ell^*(\lambda) = \arg\max_{\tilde p_\ell}\left\{\lambda\tilde p_\ell - g_\ell^{**}(\tilde p_\ell)\right\}$. 

\subsubsection*{Case 2: $\lambda^*=0$.} Here, given $y_\ell\in \partial g_\ell^*(\lambda^*)$ for all $\ell\in[c]$, ~\eqref{eq:strong_duality} implies 
\begin{align*}
    \sum_{\ell=1}^c g_\ell^{**}(\tilde p_\ell^*) &= - \sum_{\ell=1}^c g_\ell^*(0)\cr
    &= - \sum_{\ell=1}^c \sup_{\tilde p_\ell} \{0\cdot w_\ell - g_\ell^{**}(\tilde p_\ell)\}\cr
    &\stackrel{(a)}=-\sum_{\ell=1}^c (0\cdot y_\ell -  g_{\ell}^{**}(y_\ell))\cr
    &=\sum_{\ell=1}^c g_{\ell}^{**}(y_\ell),
\end{align*}
where $(a)$ again follows from the fact that $\partial g_\ell^* (\lambda) = \arg\max_{\tilde p_\ell}\{\lambda\tilde p_\ell - g_\ell^{**}(\tilde p_\ell)\}$. 

We have thus shown that $y$ is an optimal solution of the primal problem \eqref{eq:primal} in both Case 1 and Case 2. In other words, there exists an optimal solution $y$ of the convex relaxation~\eqref{eq:primal} of $\widetilde\pbold$ such that $y_\ell \in \partial g_\ell^*(\lambda^*)$ for all $\ell\in [c]$.    

We now compute $\partial g_\ell^*(\lambda^*)$. Recall that $\lambda^* = \kappa_{\ell^*}$. Thus, the inequality $\lambda^*> - |d_{\ell} |(\tilde u_{\ell} + \tilde p_{0\ell })$ is equivalent to
$\kappa_{\ell^*} > \kappa_{\ell}$, which, in light of $\kappa_{\ell'}\le \cdots \le \kappa_c$, is equivalent to $\ell^*> \ell$. By the definition of $\ell^*$, this is in turn equivalent to $\d'_\ell < 0$, which is further equivalent to
$$
    \sum_{m=1}^{\ell} \tilde u_{m} + \sum_{m=\ell+1}^{c} \tilde p_{0m} <\tilde k_0
$$
by the definition of $\d'_\ell$. Using the definitions of $\tilde u_\ell$ and $\tilde p_{0\ell}$, the above is in turn equivalent to 
$\sum_{m=1}^{\ell} b_{\Delta (m)} \tpmax_{(m)}  + \sum_{m=\ell+1}^{c}  
b_{\Delta (m)} p^{A0}_{(m)} < k_0$, i.e., the policy $p^{(\ell)}$ (defined in~\eqref{eq:only_c_policies})
violates~\eqref{eq:modified_unsustainable_constraint}.

Similarly, we can show that the inequality $\lambda^*< - |d_{\ell} |(\tilde u_{\ell} + \tilde p_{0\ell })$ is equivalent to $\ell>\ell^*$ and also  to $p^{(\ell)}$ satisfying \eqref{eq:modified_unsustainable_constraint} with inequality, and $\lambda^*= - |d_{\ell} |(\tilde u_{\ell} + \tilde p_{0\ell })$ is equivalent to $p^{(\ell)}$ satisfying \eqref{eq:modified_unsustainable_constraint} with equality.

Therefore,~\eqref{eq:subgrad_one} at $\lambda=\lambda^*$ is equivalent to 
$$
    \partial g_{\ell}^*(\lambda^*) =
    \begin{cases}
        \{\tilde u_{\ell}\}\quad&\text{if }p^{(\ell)}\text{ violates }\eqref{eq:modified_unsustainable_constraint}  \\
        [\tilde p_{0{\ell}}, \tilde u_{\ell}] \quad &\text{if }p^{(\ell)}\text{ satisfies }\eqref{eq:modified_unsustainable_constraint} \text{ with equality}  \\
        \{\tilde p_{0{\ell}}\} \quad & \text{if }p^{(\ell)}\text{ satisfies }\eqref{eq:modified_unsustainable_constraint} \text{ without equality}.
    \end{cases}
$$
for all $\ell\in [c]$.

Now, the existence of an optimal solution $\tilde p^*_{\ell} \in\partial g_{\ell}^* (\lambda^*)$ implies that 
$$
    \tilde p^*_{\ell} =
    \begin{cases}
        \tilde u_{\ell} \quad&\text{if }p^{(\ell)}\text{ violates }\eqref{eq:modified_unsustainable_constraint}  \\
        \tilde p_{0{\ell}} \quad & \text{if }p^{(\ell)}\text{ satisfies }\eqref{eq:modified_unsustainable_constraint} \text{ without equality},
    \end{cases}
$$
and $\tilde p^*_\ell \in [\tilde p_{0\ell}, \tilde u_\ell]$ otherwise. Transforming the above expression from the system of variables $\{\tilde p_\ell:\ell\in [c]\}$ back into the original system of variables $\{p^A_{(\ell)}:\ell\in[c]\}$ yields the following for the optimal solution $p^*_\ell$ of the convex relaxation of $\widetilde\pbold$:
$$
    p_{(\ell)}^* =
    \begin{cases}
        \tpmax_{(\ell)} \quad&\text{if }p^{(\ell)}\text{ violates }\eqref{eq:modified_unsustainable_constraint}  \\
        p^{A0}_{(\ell)} \quad & \text{if }p^{(\ell)}\text{ satisfies }\eqref{eq:modified_unsustainable_constraint} \text{ without equality},
    \end{cases}
$$
and $p^*_{(\ell)}\in [p^{A0}_{(\ell)}, \tpmax_{(\ell)}]$ otherwise. As proved above, among the policies in $\{p^{(\ell)}:\ell\in[c]\}$, those that violate~\eqref{eq:modified_unsustainable_constraint} are in $\{p^{(\ell)}: \ell'\le \ell<\ell^*\}$, for which we have $p^{(\ell)}_{\ell}$  and those that satisfy~\eqref{eq:modified_unsustainable_constraint} without equality are in $\{p^{(\ell)}:  \ell^*\le \ell\le c\}$. Therefore, we also have
\begin{align}\label{eq:concise}
     p_{(\ell)}^* =
    \begin{cases}
        \tpmax_{(\ell)} \quad&\text{if }\ell < \ell^* \\
        p^{A0}_{(\ell)} \quad & \text{if }\ell \ge \ell^*.
    \end{cases}
\end{align}
\end{enumerate}

Comparing~\eqref{eq:only_c_policies} with~\eqref{eq:concise}  enables us to infer that $p^*_{(\ell)} = p^{(\ell^*)} _{(\ell)}$ for all $\ell\in [c]\setminus\{\ell^*\}$. To compute $p_{(\ell^*)}^*$, we first use our assumption that $\tpmax_{(\ell)} \ne p^{A0}_{(\ell)} - \frac{v_{(\ell)}}{q_{(\ell)}}$ to verify that  $\kappa_\ell\ne 0$ for all $\ell\in [c]$. Consequently, $\lambda^* = \kappa_{\ell^*}\ne 0$. It thus follows by complementary slackness that~\eqref{eq:modified_unsustainable_constraint} is active (i.e., it holds with equality). Therefore, we have $\sum_{\ell=1}^c b_{\Delta (\ell)} p^*_{(\ell)} = k_0$, which implies that $$\sum_{\ell<\ell^*}^c b_{\Delta (\ell)} \tpmax_{(\ell)} + \sum_{\ell>\ell^*}^c b_{\Delta (\ell)} p^{A0}_{(\ell)} + b_{\Delta \ell^*} p^*_{(\ell^*)}  = k_0.$$ This yields the following expression for $p^*_{(\ell^*)}$, thereby showing that
\begin{align}\label{eq:relaxed_optimal_sup}  & p^*_{(\ell)}\cr
&:= 
    \begin{cases}
     p^{(\ell^*)}_{(\ell)}  \quad &\text{if }\ell\ne \ell^*\\
     \frac{\left(k_0 - \sum_{\ell\ne \ell^*} b_{\Delta(\ell)}p^{(\ell)}_{(\ell)} \right)}{b_{\Delta(\ell^*)}}  &\text{if }\ell=\ell^*.
\end{cases}
\end{align}
Since $\ell^*>\ell'-1$, this implies that $p^* = \bar p$, as required.  

\def\phirel{\tilde \phi_{\text{rel}}}

We have thus shown that $\bar p$ is the optimal solution of the convex relaxation~\eqref{eq:primal_in_text} of $\widetilde\pbold$. To show that $\phi(p^*)\le \phirel(\bar p)$, we first observe that $-g_\ell^{**}(\tilde p_\ell) = -d_{\ell}\tilde p_{\ell}^2$ for $\tilde p_\ell \in \{\tilde p_{0{}}, \tilde u_{\ell} \}$, i.e., $-g_\ell^{**}(\tilde p_\ell) = q_{(\ell)} (p^A_{(\ell)} )^2 - v_{(\ell)} p_{(\ell)}^A + \frac{v_{(\ell)}}{4q^2_{(\ell) }}$ for $p^A_{(\ell)} \in \left\{p^{A0}_{(\ell)}, \tpmax_{(\ell)} \right\}$.  This shows that $-g_\ell^{**}(\tilde p_\ell) = -g_{\ell}^{**}\left(\left(p^A_{(\ell)} - \frac{v_{(\ell)}}{2q_{(\ell)}} \right)b_{\Delta (\ell)}  \right) $ is an affine function whose range over the interval $p^{A0}_{(\ell)}\le p^A_{(\ell)}\le \tpmax_{(\ell)}$ is the set of $y$-coordinates of the chord joining the points $(p^{A0}_{(\ell)}, \phi_\ell(p^{A0}_{(\ell)}))$  and $(\tpmax_{(\ell)}, \phi_\ell(\tpmax_{(\ell)}))$ in a two-dimensional $(x,y)$-space, where $\phi_\ell(p^A_{(\ell) } ) :=q_{(\ell)} (p^A_{(\ell)} )^2 - v_{(\ell)} p_{(\ell)}^A + \frac{v_{(\ell)}}{4q^2_{(\ell) }}$ is a quadratic function of $p_{(\ell)}^A$. The convexity of $\phi_\ell$ now implies that the aforementioned chord lies above the function, i.e., $-g_\ell^{**}(\tilde p_\ell) = -g_\ell^{**}\left(\left(p^A_{(\ell)} - \frac{v_{(\ell)}}{2q_{(\ell)}} \right)b_{\Delta (\ell)}\right)  \ge \phi_\ell(p^A_{(\ell)} )$. Summing over all $\ell\in [c]$ now gives
\begin{align}\label{eq:complicated}
    \sum_{\ell=1}^c\left\{ -g_\ell^{**}(\tilde p_\ell)\right\}  \ge \sum_{\ell=1}^c \left\{\phi_\ell(p^A_{(\ell)} )\right\}.
\end{align}

\def\phirel{\tilde \phi_{\text{rel}}}

We now define the constraint sets defining $\widetilde\pbold$ before and after the change-of-variable transformation: $ {\mathcal C}:= \{ p^A_{(\ell)}: p^{A0}_{(\ell)} \le p^A_{(\ell)} \le \tpmax_{(\ell)}, b_{\Delta}^\top p^A \ge k_0 \}$, $\tilde {\mathcal C}:= \{ \tilde p_\ell: \tilde p_{0\ell} \le \tilde p_\ell \le \tilde u_\ell, \sum_{\ell=1}^c \tilde p_\ell \ge \tilde k_0\}$. These sets are equivalent in that $\tilde p_\ell \in \tilde {\mathcal C}$ if and only if $p^{A}_{(\ell)}\in \mathcal C$. Therefore, we can maximize both sides of \eqref{eq:complicated} to show that
$$
    \phi(p^*) + \frac{v_{(\ell)}^2 }{4 q^2_{(\ell)}} - q_{(\ell)} \left(p^{B0}_{(\ell)} \right)^2  = \max_{p^{A}_{(\ell)} \in \mathcal C} \sum_{\ell=1}^c \phi_\ell(p^A_{(\ell)} ) \le   \max_{\tilde p_\ell \in \tilde {\mathcal C}} \sum_{\ell=1}^c \left\{- g_\ell^{**}(\tilde p_\ell)\right\} = \phirel(\bar p) + \frac{v_{(\ell)}^2 }{4 q^2_{(\ell)}} - q_{(\ell)} \left(p^{B0}_{(\ell)} \right)^2,
$$
which proves the second assertion of the theorem that $\phi(p^*) \le \phirel(\bar p)$.

Finally, if $\bar p_{(\ell^*)} \in \left\{p^{A0}_{(\ell^*)}, \tpmax_{(\ell^*)} \right\}$, then a simple computation shows that $\bar p = p^{(\ell^*)}$. In this case, we have $p^{(\ell^*)}_{(\ell)}\in \left\{p^{A0}_{(\ell)}, \tpmax_{(\ell)}\right\}$ for all $\ell\in [c]$, which means that $-g^{**}_\ell(p^{(\ell^*)}) = \phi_\ell(p^{(\ell^*)} ) $ for all $\ell\in [c]$. Therefore, $\phi(p^{(\ell^*)}) =  \phirel(p^{(\ell^*)})$. Therefore,
$$
     \phi(p^{(\ell^*)} ) = \phirel(p^{(\ell^*)}) \stackrel{(a)}= \phirel(\bar p) \stackrel{(b)} \ge  \phi (p^*) \stackrel{(c)}\ge \phi(p^{(\ell^*)} ),
$$
where $(a)$ follows from $\bar p = p^{(\ell^*)}$, $(b)$ follows from the preceding assertion of the theorem, and $(c)$  holds because $p^*$ is defined as the optimal solution of $\widetilde\pbold$ (the maximizer of $\phi$ over $\mathcal C$) and because $p^{(\ell^*)}\in\mathcal C$. We have thus shown that $\phi(p^{(\ell^*)}) = \phi (p^*)$. This proves that $p^{(\ell^*)}$ is the optimal solution of $\widetilde\pbold$, as required. 
\end{proof}

\subsection{Proof of Corollary~\ref{cor:second}}

\begin{proof} We prove the assertions one by one.
\begin{enumerate}
    \item [(i)] Recalling from the proof of Theorem~\ref{thm:component_wise_pricing} that 
    $$
        p^{(0*)}_{(\ell)} =
    \begin{cases}
        \tpmax_{(\ell)}\quad&\text{if }\ell<\ell'\\
        p^{A0}_{(\ell)} \quad&\text{if }\ell\ge\ell',
    \end{cases}
    $$
    we observe that $p^{(0*)}$ satisfying \eqref{eq:modified_unsustainable_constraint} is equivalent to~\eqref{eq:needed}, which is in turn equivalent to Case (b) in the proof of Theorem~\ref{thm:component_wise_pricing}. Repeating the arguments of Case (b), therefore, shows that $\bar p = p^{(0*)}$. It now follows from Theorem~\ref{thm:component_wise_pricing} that $p^{(0*)}$ is the optimal solution of the convex relaxation~\eqref{eq:primal_in_text} of $\widetilde\pbold$. Since we also have $\bar p_{(\ell)}= p^{(0*)}_{(\ell)}\in \left\{p^{A0}_{(\ell)}, \tpmax_{(\ell)}\right\}$, it follow from the last assertion of Theorem~\ref{thm:component_wise_pricing} that $p^{(0*)}=\bar p$ is the optimal solution of $\widetilde\pbold$ as well.
    \item [(ii)] Note that for all $\ell\in [c]$ such that $\tpmax_{(\ell)}\ge \frac{v_{(\ell)}  }{q_{(\ell)}} - p^{A0}_{(\ell)}$, we have 
    \begin{align}\label{eq:too_short}
        \tilde \phi_{(\ell)}(\tpmax_{(\ell)})\ge \tilde \phi_{(\ell)}(p^A_{(\ell)})
    \end{align}
    for all $p^A_{(\ell)}\in [p^{A0}_{(\ell)},\tpmax_{(\ell)}]$, because $\tpmax$ maximizes the quadratic function $\phi$ over the interval $[p^{A0}_{(\ell)},\tpmax_{(\ell)}]$ as shown in Section~\ref{sec:component-wise}.
    Now,  suppose there exists an $\ell_0\in [c]$ such that $\tpmax_{(\ell_0)}\ge \frac{v_{(\ell_0)}  }{q_{(\ell_0)}} - p^{A0}_{(\ell_0)}$ and $p^*_{(\ell_0)} < \tpmax_{(\ell_0)}$ for some optimal policy $p^*$. In addition, consider another policy $p'$, defined as follows:
    $$
        p'_{(\ell)}:=
        \begin{cases}
            \tpmax_{(\ell)}&\quad\text{if }\ell=\ell_0\\
            p^*_{(\ell)}&\quad\text{if }\ell\ne \ell_0.
        \end{cases}
    $$
    Note that $p'\ge p^*$.  
    
    Now, $p^*$ being optimal implies that it is also a feasible policy. Hence, $p^*$ satisfies~\eqref{eq:modified_unsustainable_constraint}, i.e., $b_\Delta^\top p^*\ge k_0$.  Since $b_\Delta =M_\Delta\allone>\allzero$ by Theorem~\ref{thm:one}-(i), it follows that $b_{\Delta}^\top p' \ge b_{\Delta}^\top p^* \ge k_0$, which means $p'$ is also a feasible policy. Furthermore, we have $\tilde \phi_{(\ell)}(p'_{(\ell)}) =  \tilde \phi_{(\ell)}(p^*_{(\ell)})$ for all $\ell\ne \ell^)$ by the definition of $p'$, and we have $\tilde \phi_{(\ell_0)}(p'_{(\ell_0)}) \ge  \tilde \phi_{(\ell_0)}(p^*_{(\ell_0)})$ by \eqref{eq:too_short}. Therefore, $\tilde \phi(p')\ge \tilde \phi(p^*)$. Since $p^*$ is optimal, this implies that $p'$ is optimal as well. 
    
    We have thus shown the following: if we have $p^*_{(\ell)}<\tpmax_{(\ell)}$ for an $\ell\in [c]$ that satisfies $\tpmax_{(\ell)}\ge \frac{v_{(\ell)}  }{q_{(\ell)}} - p^{A0}_{(\ell)}$, then $p^*$ remains optimal even if we replace its $\ell$-th entry with $\tpmax_{(\ell)}$. It follows that there exists an optimal solution whose $\ell$-th entry equals $\tpmax_{(\ell)}$ for all $\ell\in [c]$ satisfying $\tpmax_{(\ell)}\ge \frac{v_{(\ell)}  }{q_{(\ell)}} - p^{A0}_{(\ell)}$. We know from~\eqref{eq:p-knot-star} that we also have $p^*_{(\ell)}=p^{(0*)}_{(\ell)}$ for all such $\ell$.
    \item [(iii)] Note that $\widetilde\pbold$ involves convex maximization over the polytope $\tilde{\mathcal P}$, defined as the set of all policies $p^A$ satisfying~\eqref{eq:constraints_component_wise}. Therefore, there exists an extreme point of $\tilde{\mathcal P}$ that solves $\widetilde\pbold$ optimally. Now, every extreme point of this polytope is a point at which at least $c$ out of the $2c+1$ constraints in \eqref{eq:constraints_component_wise} are active. Since at most $1$ out of these $c$ constraints can be given by~\eqref{eq:modified_unsustainable_constraint}, it follows that the remaining $c-1$ constraints must be given by either~\eqref{eq:modified_price_up_constraint} or \eqref{eq:modified_price_threshold_constraint}. In other words, every extreme point policy vector $p$ has at least $c-1$ entries, say $p_{\ell_1}, p_{\ell_2}, \ldots, p_{\ell_{c-1}}$, such that $p_{\ell_i}\in \left\{p^{A0}_{\ell_i},\tpmax_{\ell_i} \right\}$ for every $i\in[c-1]$. It follows that the optimal solution among these extreme point policies also satisfies the aforementioned set condition in at least $c-1$ of its entries. 
\end{enumerate}
\end{proof}

\end{document}